\documentclass[11pt,reqno]{amsart}

\usepackage{geometry}
\usepackage[utf8]{inputenc}
\usepackage{amstext,latexsym,amsbsy,amsmath,amssymb,amsthm,mathtools,relsize,geometry,enumerate}
\usepackage{hyperref}
\hypersetup{
	colorlinks   = true, %Colours links instead of ugly boxes
	urlcolor     = blue, %Colour for external hyperlinks
	linkcolor    = red, %Colour of internal links
	citecolor   = red %Colour of citations
}
\usepackage{mathtools, enumerate,enumitem}

\usepackage{mathrsfs}

\usepackage{comment}
\usepackage{graphicx}
\usepackage{epstopdf}
\setkeys{Gin}{width=\linewidth,totalheight=\textheight,keepaspectratio}
\graphicspath{{./images/}}

\usepackage[normalem]{ulem}

\usepackage{multicol}
\usepackage{booktabs}
\usepackage{mathrsfs}
\usepackage{color}

\numberwithin{equation}{section}

\usepackage{bbm}

\newcommand{\E}{\mathcal{E}}

\renewcommand{\div}{\mbox{\rm div\,}}

\newcommand{\mP}{\mathbb{P}}

\newcommand{\mE}{\mathbb{E}}
\newcommand{\mV}{\mathbb{V}}

\newcommand{\mZ}{\mathbb{Z}}

\newcommand{\mQ}{\mathbb{Q}}

\newcommand{\Law}{\textup{Law}}

\newcommand{\ubar}{\bar{u}}

\newcommand{\Ome}{\Omega}
\newcommand{\p}{\partial}
\newcommand{\nab}{\nabla}

\newcommand{\M}{\mathcal{M}}

\def\E{\mathbb{E}}
\def\P{\mathbb{P}}
\numberwithin{equation}{section}

\theoremstyle{plain}
\newtheorem{theorem}{Theorem}[section]

\newtheorem{remark}{Remark}[section]
\newtheorem{lemma}[theorem]{Lemma}

\newtheorem{proposition}[theorem]{Proposition}

\allowdisplaybreaks[4]

\begin{document}
	
	\title{Finite element approximations of the stochastic Benjamin-Bona-Mahony equation with multiplicative noise}
	%\markboth{HUNG D. NGUYEN and LIET VO}{STOCHASTIC BBM EQUATION}		
	
	\author{
		Hung D. Nguyen$^1$, Thoa Thieu$^2$ and 
		Liet Vo$^3$ 
	}
	
	\thanks{$^1$ Department of Mathematics, The University of Tennessee, Knoxville, TN 37996, U.S.A. ({\tt hnguye53@utk.edu}).} 
	
	\thanks{$^2$ School of Mathematical and Statistical Sciences, The University of Texas Rio Grande Valley, Edinburg, TX 78539, U.S.A.  ({\tt thoa.thieu@utrgv.edu}).}
	
	\thanks{$^3$ School of Mathematical and Statistical Sciences, The University of Texas Rio Grande Valley, Edinburg, TX 78539, U.S.A.  ({\tt liet.vo@utrgv.edu}). This author was partially supported by the NSF grant DMS-2530211.}

	\begin{abstract} This paper is devoted to the numerical analysis of a fully discrete finite element approximation for the stochastic Benjamin-Bona-Mahony equation driven by multiplicative noise. We first establish the existence and uniqueness of solutions to the stochastic BBM equation within an appropriate variational framework and derive several stability estimates for the continuous problem, including an exponential stability result. For the numerical approximation, a conforming finite element method is employed for spatial discretization and is coupled with the implicit Euler-Maruyama scheme for time integration. The convergence of the fully discrete scheme is investigated under two different classes of multiplicative noise. When the noise coefficient is bounded, we obtain optimal strong error estimates in full expectation by combining exponential stability properties of both the stochastic BBM solution and its fully discrete counterpart with a stochastic Gronwall inequality. In the case of general multiplicative noise, where boundedness assumptions are no longer valid, a localization technique based on high-probability events in the sample space is introduced, leading to sub-optimal convergence rates in probability. Finally, numerical experiments are presented to corroborate the theoretical results and to demonstrate the performance of the proposed method.
	\end{abstract}
	\maketitle

	{\bf Key words.} Stochastic partial differential equations, multiplicative noise, Wiener process, It\^o stochastic integral, Euler scheme, finite element method, error estimates, stochastic Benjamin–Bona–Mahony.
	
	\medskip
	
	{\bf AMS subject classifications.} 65N12, %Stability and convergence of numerical methods
	65N15, %Error bounds
	65N30. %Finite elements, Rayleigh-Ritz and Galerkin methods, finite methods

	\section{Introduction}\label{sec-1}
	We consider the stochastic Benjamin--Bona--Mahony (BBM) equation driven by It\^o multiplicative noise, given by
	\begin{subequations}\label{eq1.1}
		\begin{align}\label{eq1.1a}
			du - d(\Delta u) + \big[ \nu u + \div\big(F(u)\big) \big]\,dt
			&= G(u)\,dW(t),
			\qquad \text{in } (0,T)\times D,\ \text{a.s.},\\
			u(0) &= u_0, \qquad \text{on } D,\ \text{a.s.},
		\end{align}
	\end{subequations}
	where $D=[0,L]^2\subset\mathbb{R}^2$ with $L>0$, and the solution
	$u$ is assumed to be spatially periodic with period $L$.
	The unknown $u$ represents the wave amplitude of the underlying dispersive wave field.
	
	In \eqref{eq1.1}, the parameter $\nu\ge 0$ denotes a (possibly vanishing)
	linear damping coefficient accounting for dissipative effects.
	The nonlinear transport term is defined by
	\[
	F(u) = \big\langle u + \tfrac12 u^2,\; u + \tfrac12 u^2 \big\rangle^{T},
	\]
	so that $\div(F(u))$ describes nonlinear wave propagation characteristic
	of BBM-type dynamics.
	The stochastic forcing is introduced through a real-valued Wiener process
	$\{W(t): t\ge 0\}$, and the diffusion operator $G$ depends explicitly on the
	solution $u$, leading to a multiplicative noise structure.
	Consequently, the state variable $u$ not only governs the deterministic
	nonlinear evolution but also modulates the intensity of the random
	perturbations acting on the system.	A precise definition of the operator $G$, together with the assumptions
	imposed on its regularity and growth, is given in Section~\ref{sec2}.

	The deterministic BBM equation is a classical dispersive wave model arising in the study of long-wave propagation in nonlinear media. Originally introduced as a regularized alternative to the Korteweg-de Vries equation \cite{benjamin1972model,bona1981evaluation}, the BBM equation enjoys improved well-posedness properties while retaining the essential nonlinear and dispersive mechanisms of interest in applications such as water waves and nonlinear acoustics. Owing to these features, the BBM equation has been extensively investigated and has become a standard model for the analysis and numerical approximation of nonlinear dispersive partial differential equations; see, for example, \cite{li2020optimal,omrani2008finite,jiao2025efficient} and the references therein.
	
	When stochastic forcing is incorporated, the stochastic BBM equation provides a natural framework for modeling wave phenomena subject to random perturbations, such as environmental fluctuations or unresolved small-scale effects. In particular, multiplicative noise allows the random forcing to depend on the solution itself, leading to more realistic but significantly more challenging dynamics. Despite its physical relevance, the stochastic BBM equation has received comparatively little attention in the literature, especially from a rigorous analytical and numerical perspective.
	
	From a mathematical standpoint, the stochastic BBM equation with multiplicative noise poses several fundamental difficulties. The nonlinear convective term is non-Lipschitz and interacts strongly with the stochastic forcing, preventing the direct application of standard SPDE techniques that rely on global Lipschitz or monotonicity assumptions \cite{da2014stochastic}. Moreover, the dispersive structure of the BBM operator complicates the derivation of classical energy estimates, particularly in the presence of multiplicative noise, where stochastic perturbations may be amplified in regions of large solution amplitude. These effects do not act independently, but rather interact in a subtle manner, making both the analytical study and the design of convergent numerical schemes highly nontrivial. 
	
	In \cite{li2018random, wang2009random, chen2022random, chen2023multivalued, chen2024asymptotically}, the long-time dynamics and random attractors of stochastic BBM equations were investigated in both two and three spatial dimensions. More recently, the existence of a mild solution for the one-dimensional stochastic BBM equation with multiplicative noise was established in \cite{dhayal2025stability}. For generalized stochastic BBM equations with additive noise in two dimensions, the existence and uniqueness of mild solutions were obtained in \cite{bhar2025strong}.
	
	It is important to note that in all the aforementioned works, an additional Laplacian term $\Delta u$ was introduced for analytical convenience. While this regularizing term facilitates the analysis, it modifies the structure of the original Benjamin-Bona-Mahony equation. In contrast, the present work preserves, to a large extent, the original form of the classical BBM equation. We only introduce a lower-order linear damping term $\nu u$, which is used exclusively to derive exponential stability estimates needed for the numerical analysis. Importantly, this damping term does not affect the well-posedness theory developed in this paper: the existence and uniqueness results remain valid even when the damping is removed, thereby recovering the original stochastic BBM equation.
	
	We emphasize that the presence of the damping term allows us to establish exponential stability, which in turn yields strong convergence of the proposed numerical methods. In the absence of damping, although the well-posedness results remain unchanged, the available stability estimates are weaker, and the numerical analysis can only guarantee the convergence in probability.
	
	The first objective of this paper is to establish the existence and uniqueness of a strong solution to the stochastic Benjamin-Bona-Mahony equation with multiplicative noise in two dimensions within a suitable variational framework. In addition, we derive several stability estimates for the continuous problem, including an exponential stability result that plays a central role in the subsequent numerical analysis. To the best of our knowledge, this work provides the first systematic and rigorous well-posedness theory for the original stochastic BBM equation with multiplicative noise.
	
	On the numerical side, a wide variety of methods have been developed for the
	deterministic BBM equation, including finite difference
	schemes \cite{omrani2008finite, berikelashvili2014convergence, bayarassou2021fourth},
	spectral and collocation methods \cite{ahmad2024applications, mulimani2024new, gheorghiu2016stable},
	and finite element methods
	\cite{kumbhar2023galerkin, kadri2008methods, omrani2006convergence,
		karakoc2019galerkin, shi2020new,buli2018local}.
	Among these approaches, finite element methods are particularly appealing due
	to their flexibility in handling boundary conditions and their solid
	theoretical foundation.
	
	In contrast, the numerical analysis of the stochastic BBM equation remains
	largely undeveloped, especially in the fully discrete setting and in the
	presence of multiplicative noise.
	To date, only limited results are available.
	In \cite{bhar2025strong}, a finite element approximation was studied for a
	generalized stochastic BBM equation driven by additive noise.
	Although numerical simulations of stochastic dispersive equations have been
	reported in related contexts, rigorous, strong convergence results for fully
	discrete finite element approximations of the stochastic BBM equation with multiplicative noise are, to the best of our knowledge, not available in the
	literature.
	
	The second objective of this paper is therefore to develop a rigorous
	convergence theory for a fully discrete finite element approximation of the
	stochastic Benjamin-Bona-Mahony equation. We employ a conforming finite element method for the spatial discretization and couple it with the implicit Euler-Maruyama scheme for time integration.
	The convergence analysis is carried out for two distinct classes of
	multiplicative noise. When the noise coefficient is bounded, we establish an optimal strong error
	estimates in full expectation by combining exponential stability properties
	of both the exact and numerical solutions with a discrete stochastic Gronwall
	inequality \cite{kruse2018discrete}.
	This strategy has recently proven effective for treating nonlinearities in
	numerical approximations of the stochastic Navier-Stokes equations and the
	stochastic fourth-order Kuramoto-Sivashinsky equations
	\cite{nguyen2025fully, feng2025full}.
	
	For general multiplicative noise, where boundedness assumptions are no longer
	applicable, we adopt a localization technique based on high-probability events
	in the underlying sample space, following the approach of
	\cite{carelli2012rates}.
	This allows us to derive sub-optimal convergence rates in probability for the
	fully discrete scheme.
	
	Beyond the stochastic BBM equation itself, the analytical and numerical
	techniques developed in this work, most notably the well-posedness framework, the derivation of exponential stability estimates, and the combined use of stochastic Gronwall and localization arguments are expected to be applicable
	to a broader class of nonlinear dispersive stochastic partial differential
	equations with multiplicative noise.

	The remainder of the paper is organized as follows.
	In Section~\ref{sec2}, we introduce the functional setting and state the assumptions on the noise operator.
	Section~\ref{sec-3} is devoted to the existence, uniqueness, and stability of the stochastic BBM equation with multiplicative noise, including an exponential stability estimate.
	In Section~\ref{sec-4}, we present the fully discrete finite element scheme and establish stability properties of the numerical solution, including a discrete exponential stability estimate.
	Subsections~\ref{sub-sec4.2} and~\ref{sub-sec4.3} contain the derivation of error estimates for the proposed scheme under bounded and general multiplicative noise, respectively.
	In Section~\ref{sec-5}, we provide numerical experiments to validate the theoretical results.
	Finally, Appendix~\ref{appendxiA} recalls a discrete stochastic Gronwall inequality, and Appendix~\ref{auxiliary} contains the proof of the well-posedness of the stochastic BBM equation.

	\section{Preliminaries}\label{sec2}
	\subsection{Notations}
	Standard function and space notation will be adopted in this paper. 
	We denote ${L}^p(D)$ and ${ H}^{m}(D)$ as the Lebesgue and Sobolev spaces of the functions that are periodic with period $L$ and have vanishing mean. In particular, $L^2(D)=H^0(D)$. 
	
	For each $k\in \mZ^2$, we denote $e_k=e^{\frac{2\pi \textup{i}}{L}k\cdot x }$. Then $\{e_k\}_{k\in\mZ^2}$ is an orthonormal basis in $L^2$ that diagonalizes the Laplacian through the relation
	\begin{align} \label{eqn:A.e_k=alpha_k.e_k}
		-\triangle e_k = \alpha_k e_k,
	\end{align}
	where $\alpha_k=\left(\frac{2\pi}{L}\right)^2|k|^2$. For $n\ge 1$, we denote $H_n=\textup{span}\{e_k:|k|\le n\}$ and the projection map
	\begin{align*}
		P_n u = \sum_{k\in\mZ^2:|k|\le n}(u,e_k)e_k. 
	\end{align*}
	
	For notational convenience, let $\M: H^2(D) \rightarrow L^2(D)$ be the operator defined as
	\begin{align}
		\M u := (I - \Delta)u, \quad u\in H^2(D),
	\end{align}
	where $I$ is the identity operator. Then, $\M^{-1}$ exists and $\M$ is a self-adjoint operator, which satisfies
	\begin{align*}
		(\M u, v) = (u, \M v),
		\quad u,v\in H^2(D).
	\end{align*}
	In particular, from relation \eqref{eqn:A.e_k=alpha_k.e_k}, we get for all $\beta\in\mathbb{R}$,
	\begin{align*}
		\M^{\beta}e_k = (1+\alpha_k)^{\beta}e_k,\quad k\in\mZ^2.
	\end{align*}
	
	In addition, we recall the Ladyzhenskaya inequality in two dimensions:
	\begin{align}\label{Lady_ineq}
		\|u\|^2_{L^{4}} \leq C_e\|u\|_{L^2}\|u\|_{H^1},
	\end{align}
	where $C_e>0$ is a pure constant. Also, throughout the rest of the paper, $C$ denotes a generic constant that is independent of the mesh parameters $h$ and $k$.

	\subsection{Assumptions}
	
	Concerning the stochastic forcing, in this paper, we will make the following assumptions on $G: L^2(D) \rightarrow L^2(D)$.
	\begin{enumerate}[label=(\Alph*)]
		
		\item For all $m\ge 0$, suppose that $G: H^{m}(D) \rightarrow H^m(D)$. Moreover, there exists a constant $C_{G}>0$ such that
		\begin{align}\label{Assump_Lipschitz}
			\|G(u) - G(v)\|_{H^m} \leq C_{G}\|u-v\|_{H^m}.
		\end{align}
		
		\item There exists a constant $L_0>0$ such that
		\begin{align}\label{Assump_Lineargrowth}
			\|G(u)\|_{L^2} \leq L_0.
		\end{align}
		
	\end{enumerate}
	
	\begin{remark} 1.  We remark that while the Lipschitz condition \eqref{Assump_Lipschitz} will be employed throughout the present paper, the boundedness condition \eqref{Assump_Lineargrowth} will only be needed to derive exponential moment bounds on the solutions. In turn, these bounds will allow us to conclude the optimal strong error estimates in full expectation. See Section \ref{sub-sec4.2} for a further discussion of this point.
		
		2. It should be noted that \eqref{Assump_Lipschitz} implies
		\begin{align} \label{cond:|G(u)|_Hm<|u|_Hm}
			\|G(u)\|_{H^m} \leq c_G (1 + \|u\|_{H^m}),
		\end{align}
		where $c_G = \max\{C_G, \|G(0)\|_{H^m}\}$. 
		
	\end{remark}

	\section{Solution concepts and PDE results}\label{sec-3}
	
	In this section, we establish the existence and uniqueness of a strong solution, which is also a weak solution in the PDE sense, to equation \eqref{eq1.1} in Theorem \ref{thm:well-posed}. We also state and prove high-moment estimates of the solution in Lemma \ref{lem:moment-bound:H1} and Lemma \ref{lem:moment-bound:H2}, which will be used later to derive error estimates for the fully discrete finite element method. More importantly, we will leverage condition \eqref{Assump_Lineargrowth} on $G$ to derive an exponential stability estimate for the solution $u$. The result of which is crucial for achieving full-moment error estimates for the proposed scheme.
	
	Throughout, we will fix a probability basis $(\Omega, \mathcal{F}, \mathbb{P}, \{\mathcal{F}_t\}_{t \geq 0} )$  satisfying the usual conditions \cite{karatzas2012brownian} and let $W$ be an ${\mathbb R}$-valued Wiener process defined on this space. We now state the well-posedness of \eqref{eq1.1} below in Theorem \ref{thm:well-posed}, whose proof is deferred to Appendix \ref{auxiliary}.
	
	\begin{theorem} \label{thm:well-posed}  
		Suppose ${u}_0\in L^2(\Omega; H^1(D))$ {and that $G$ satisfies condition \eqref{Assump_Lipschitz}}. Then, there exists a unique solution $u \in L^2(\Omega; L^2([0,T]; H^1(D)))$ that satisfies $\mP$-a.s.
		\begin{align}\label{Weak_formulation}
			&(u(t), \phi) + \left(\nab u(t), \nab \phi\right) + \nu\int_{0}^{t}\bigl( u(s) , \phi\bigr)\, ds  - \int_{0}^t \bigl( F(u(s)), \nab \phi\bigr)\, ds
			\\\nonumber&\qquad= \bigl(u_0,\phi\bigr) + \bigl(\nab u_0,\nab\phi\bigr)+ \left(\int_0^t G(u)\, dW(s), \phi\right),\qquad \forall \phi \in H^1(D).
		\end{align}

	\end{theorem}

	Next, we derive high-moment stability estimates in the energy norm for the
	variational solution. These estimates play a crucial role in the subsequent
	numerical analysis, as they are essential for establishing the error bounds of
	the proposed numerical method.

	\begin{lemma} \label{lem:moment-bound:H1}
		Letting $n\in \mathbb{N}$, suppose that $u_0\in L^{2n}(\Omega;H^1(D))$ and that condition \eqref{Assump_Lipschitz} holds. Then, the following holds
		\begin{align} \label{ineq:moment-bound:H1:poly-moment}
			\E\left[\sup_{t\in[0,T]}\|u(t)\|^{2n}_{H^1} \right]\le C_{1,n},
		\end{align}
		where $C_{1,n} = C(T,n)\big(\E\|u_0\|^{2n}_{H^1}+1\big)$.

	\end{lemma}
	
	\begin{proof} First of all, from equation \eqref{eq1.1}, we apply It\^o's formula to $\|\M^{\frac12}u\|^2_{L^2}=\|u\|^2_{L^2}+\|\nab u\|^2_{L^2}$ and obtain
		\begin{align} \label{eqn:d(|u|^2_L2+|nab.u|^2_L2)}
			d \big(\|u\|^2_{L^2}+\|\nab u\|^2_{L^2}\big) + 2\nu\|u\|^2_{L^2}dt = \|\M^{-\frac{1}{2}}G(u)\|^2_{L^2}dt + 2\big(u,G(u)dW\big) .
		\end{align}
		In the above, we have employed the cancellation identity $(u,\div(F(u))=0$. 
		
		Next, letting $n\ge 1$ be given, we apply It\^o's formula to $\big(\|u\|^2_{L^2}+\|\nab u\|^2_{L^2}\big)^n$ and obtain
		\begin{align}  \label{eqn:d(|u|^2_L2+|nab.u|^2_L2)^n}
			&d \big(\|u\|^2_{L^2}+\|\nab u\|^2_{L^2}\big)^n \notag \\
			& =n \big(\|u\|^2_{L^2}+\|\nab u\|^2_{L^2}\big)^{n-1}\Big(-2\nu\|u\|^2_{L^2}dt+\|\M^{-\frac{1}{2}}G(u)\|^2_{L^2}dt + 2\big(u,G(u)dW\big)\Big) \notag \\
			&\qquad+ \frac{1}{2}n(n-1) \big(\|u\|^2_{L^2}+\|\nab u\|^2_{L^2}\big)^{n-2}\cdot 4|(u,G(u))|^2 .
		\end{align}
		Since $\|\M^{-\frac{1}{2}}G(u)\|_{L^2}\le \|G(u)\|_{L^2}$, together with estimate \eqref{cond:|G(u)|_Hm<|u|_Hm} and Young's inequality, we infer
		\begin{align*}
			\big(\|u\|^2_{L^2}+\|\nab u\|^2_{L^2}\big)^{n-1}\|\M^{-\frac{1}{2}}G(u)\|^2_{L^2} & \le \big(\|u\|^2_{L^2}+\|\nab u\|^2_{L^2}\big)^{n-1}\|G(u)\|^2_{L^2}\\
			&\le C \big(\|u\|^2_{L^2}+\|\nab u\|^2_{L^2}\big)^{n}+C.
		\end{align*}
		Likewise, we invoke Holder's inequality to deduce that
		\begin{align*}
			\big(\|u\|^2_{L^2}+\|\nab u\|^2_{L^2}\big)^{n-2}|(u,G(u))|^2\le C \big(\|u\|^2_{L^2}+\|\nab u\|^2_{L^2}\big)^{n}+C.
		\end{align*}
		It follows from \eqref{eqn:d(|u|^2_L2+|nab.u|^2_L2)^n} that
		\begin{align} \label{ineq:d(|u|^2_L2+|nab.u|^2_L2)^n}
			&d  \big(\|u\|^2_{L^2}+\|\nab u\|^2_{L^2}\big)^{n}\notag \\
			&\le  C \big(\|u\|^2_{L^2}+\|\nab u\|^2_{L^2}\big)^{n}dt+Cdt+2n \big(\|u\|^2_{L^2}+\|\nab u\|^2_{L^2}\big)^{n-1}(u,G(u)dW).
		\end{align}
		As a consequence, we obtain
		\begin{align*}
			&\E\left[\big(\|u(t)\|^2_{L^2}+\|\nab u(t)\|^2_{L^2}\big)^n\right] \\
			&\le \E\left[\big(\|u_0\|^2_{L^2}+\|\nab u_0\|^2_{L^2}\big)^n\right]+ C\int_0^t  \E\left[\big(\|u(s)\|^2_{L^2}+\|\nab u(s)\|^2_{L^2}\big)^n\right] \,ds+Ct. 
		\end{align*}
		In turn, Gronwall's inequality implies that
		\begin{align} \label{ineq:E[|u|^2_L2+|nab.u|^2_L2)^n]}
			\E\big[\big(\|u(t)\|^2_{L^2}+\|\nab u(t)\|^2_{L^2}\big)^n \big] \le C\big(\E\big[\big(\|u_0\|^2_{L^2}+\|\nab u_0\|^2_{L^2}\big)^n\big]+1\big), \quad t\in[0,T],
		\end{align}
		where $C=C(T,n)$ is a positive constant independent of $u_0$.
		
		Turning back to \eqref{ineq:moment-bound:H1:poly-moment}, from \eqref{ineq:d(|u|^2_L2+|nab.u|^2_L2)^n}, we get
		\begin{align*}
			&\sup_{s\in[0,t]} \big(\|u(s)\|^2_{L^2}+\|\nab u(s)\|^2_{L^2}\big)^n\\
			&\le \big(\|u_0\|^2_{L^2}+\|\nab u_0\|^2_{L^2}\big)^n + C\int_0^t \big(\|u(\ell)\|^2_{L^2}+\|\nab u(\ell)\|^2_{L^2}\big)^nd\ell+Ct + \sup_{s\in[0,t]}M(s),
		\end{align*}
		where $M(s)$ is the semi-martingale process given by
		\begin{align*}
			dM = 2n \big(\|u\|^2_{L^2}+\|\nab u\|^2_{L^2}\big)^{n-1}(u,G(u)dW).
		\end{align*}
		On the one hand, \eqref{ineq:E[|u|^2_L2+|nab.u|^2_L2)^n]} implies
		\begin{align*}
			\int_0^t \E\left[\big(\|u(\ell)\|^2_{L^2}+\|\nab u(\ell)\|^2_{L^2}\big)^n\right] \,d\ell \le C\left\{\E\left[\big(\|u_0\|^2_{L^2}+\|\nab u_0\|^2_{L^2}\big)^{n}\right]+1\right\}.
		\end{align*}
		On the other hand, we invoke Burkholder's inequality to infer
		\begin{align*}
			\E \left[\sup_{s\in[0,t]}|M(s)|\right] \le C\int_0^t\E \Big[\big(\|u(\ell)\|^2_{L^2}+\|\nab u(\ell)\|^2_{L^2}\big)^{2n-2}\big|(u(\ell),G(u(\ell)))\big|^2\Big]d\ell+C.
		\end{align*}
		We employ Holder's inequality while making use of estimate \eqref{cond:|G(u)|_Hm<|u|_Hm} to further deduce
		\begin{align*}
			& \int_0^t\E \Big[\big(\|u(\ell)\|^2_{L^2}+\|\nab u(\ell)\|^2_{L^2}\big)^{2n-2}\big|(u(\ell),G(u(\ell)))\big|\Big]d\ell\\
			&\quad \le \int_0^t\E \Big[\big(\|u(\ell)\|^2_{L^2}+\|\nab u(\ell)\|^2_{L^2}\big)^{2n-2}\|u(\ell)\|^2_{L^2}\|G(u(\ell)\|^2_{L^2}\Big]d\ell \\
			&\quad\le C\int_0^t\E \Big[\big(\|u(\ell)\|^2_{L^2}+\|\nab u(\ell)\|^2_{L^2}\big)^{2n}\Big]d\ell+Ct.
		\end{align*}
		Once again, in view of \eqref{ineq:E[|u|^2_L2+|nab.u|^2_L2)^n]}, we readily have
		\begin{align*}
			\int_0^t\E \Big[\big(\|u(\ell)\|^2_{L^2}+\|\nab u(\ell)\|^2_{L^2}\big)^{2n}\Big]d\ell \le C\Big[\E\big(\|u_0\|^2_{L^2}+\|\nab u_0\|^2_{L^2}\big)^{2n}+1\Big],
		\end{align*}
		Altogether, we arrive at the bound
		\begin{align*}
			\E\left[\sup_{s\in[0,t]} \big(\|u(s)\|^2_{L^2}+\|\nab u(s)\|^2_{L^2}\big)^n \right] \le C\E\left[\big(\|u_0\|^2_{L^2}+\|\nab u_0\|^2_{L^2}\big)^{2n}\right] + C.
		\end{align*}
		for a suitable positive constant $C=C(T,n)$ independent of $u_0$. This establishes \eqref{ineq:moment-bound:H1:poly-moment}, thereby finishing the proof.
		
	\end{proof}
	
	\medskip
	
	Since the error analysis of the finite element method developed in this paper
	requires higher regularity of the solution, we establish high-moment stability
	estimates in the $H^2$ norm for the variational solution in the following lemma.

	\begin{lemma} \label{lem:moment-bound:H2}
		Letting $n\in \mathbb{N}$, suppose that $u_0\in L^{12n}(\Omega;H^1(D))\cap L^{4n}(\Omega;H^2(D))$ and that condition \eqref{Assump_Lipschitz} holds. Then, for all $T\ge 0$, the following holds
		\begin{align} \label{ineq:moment-bound:H2:poly-moment}
			\E\left[\sup_{t\in[0,T]}\|u(t)\|^{2n}_{H^2} \right]\le C_{2,n},
		\end{align}
		where $C_{2,n} = C(T,n)\big(\E\big[\|u_0\|^{12n}_{H^1}+\|u_0\|^{4n}_{H^2}\big]+1\big)$.
		
	\end{lemma}

	\begin{proof}
		We apply It\^o's formula to $\|\nab u\|^2_{L^2}+\|\triangle u\|^2_{L^2}$ and obtain the identity
		\begin{align} \label{eqn:d(|nab.u|^2+|triangle.u|^2)}
			& d \big(\|\nab u\|^2_{L^2}+\|\triangle u\|^2_{L^2} \big)  \notag \\
			& =-2\nu\|\nab u\|^2_{L^2}dt -2 (\nab \div(F(u)),\nab u)dt + 2 (\nab u,\nab G(u)dW)+\|\nab G(u)\|^2_{L^2}dt.
		\end{align}
		Concerning the term involving $\div F(u)= \div (\langle u+\frac{1}{2}u^2,u+\frac{1}{2}u^2\rangle^T)$ on the above right-hand side, we note that $(\nab (\partial_x  +\partial_y) u ,\nab u)=0$. So,
		\begin{align*}
			(\nab \div(F(u)),\nab u) & = (\nab (u\,(\partial_x  +\partial_y) u ),\nab u)\\
			&= (\nab u \,(\partial_x  +\partial_y) u,\nab u)+ (u\, \nab (\partial_x  +\partial_y) u,\nab u).
		\end{align*}
		On the one hand, we employ Holder's and Ladyzhenskaya's inequalities to estimate
		\begin{align*}
			\big|(\nab u \,(\partial_x  +\partial_y) u,\nab u)\big| \le C\|\nab u\|^3_{L^3}\le C\|\nab u\|^2_{L^4}\|\nab u\|_{L^2}&\le C\|\nab u\|_{L^2}\|\triangle u\|_{L^2}\|\nab u \|_{L^2}\\
			& = C\|\nab u\|^2_{L^2}\|\triangle u\|_{L^2}.
		\end{align*}
		On the other hand, we invoke Agmon's inequality to infer
		\begin{align*}
			\big|(u\, \nab (\partial_x  +\partial_y) u,\nab u)\big| \le C\|u\|_{L^\infty}\|\triangle u\|_{L^2}\|\nab u \|_{L^2} & \le C \|\nab u\|^{\frac{1}{2}}_{L^2}\|\triangle u\|^{\frac{1}{2}}_{L^2}\cdot \|\triangle u\|_{L^2}\|\nab u\|_{L^2}\\
			& = C\|\nab u\|^{\frac{3}{2}}_{L^2}\|\triangle u\|^{\frac{3}{2}}_{L^2}.
		\end{align*}
		It follows from Young's inequality that
		\begin{align*}
			\big|  (\nab \div(F(u)),\nab u) \big| &\le C\big(\|\nab u\|^{2}_{L^2}\|\triangle u\|_{L^2}+\|\nab u\|^{\frac{3}{2}}_{L^2}\|\triangle u\|^{\frac{3}{2}}_{L^2}\big) \\ &\le C\big(\|\nab u\|^6_{L^2} + \|\triangle u\|^2_{L^2}+1\big).
		\end{align*}
		Also, estimate \eqref{cond:|G(u)|_Hm<|u|_Hm} implies that
		\begin{align*}
			\|\nab G(u)\|^2_{L^2} \le C\big(\|\nab u\|^2_{L^2}+1\big).
		\end{align*}
		As a consequence, from \eqref{eqn:d(|nab.u|^2+|triangle.u|^2)}, we get 
		\begin{align} \label{ineq:d(|nab.u|^2+|triangle.u|^2)}
			d \big(\|\nab u\|^2_{L^2}+\|\triangle u\|^2_{L^2} \big)  & \le C\big(\|\nab u\|^6_{L^2} + \|\triangle u\|^2_{L^2}+1\big)dt + 2 (\nab u,\nab G(u)dW).
		\end{align}
		
		Next, for $n\ge 1$, we compute
		\begin{align*}
			&d \big(\|\nab u\|^2_{L^2}+\|\triangle u\|^2_{L^2} \big)^{n}\\
			& = n \big(\|\nab u\|^2_{L^2}+\|\triangle u\|^2_{L^2} \big)^{n-1}d \big(\|\nab u\|^2_{L^2}+\|\triangle u\|^2_{L^2} \big)\\
			&\qquad+\frac{1}{2}n(n-1)  \big(\|\nab u\|^2_{L^2}+\|\triangle u\|^2_{L^2} \big)^{n-2}\big\langle d \big(\|\nab u\|^2_{L^2}+\|\triangle u\|^2_{L^2} \big),d \big(\|\nab u\|^2_{L^2}+\|\triangle u\|^2_{L^2} \big)\big\rangle\\& = I_1+I_2.
		\end{align*}
		In view of \eqref{ineq:d(|nab.u|^2+|triangle.u|^2)}, we employ Young's inequality to see that
		\begin{align*}
			I_1 & \le C\big(\|\nab u\|^2_{L^2}+\|\triangle u\|^2_{L^2} \big)^{n}dt+Cdt+ C\|\nab u\|^{6n}_{L^2}dt \\
			&\qquad+ 2n \big(\|\nab u\|^2_{L^2}+\|\triangle u\|^2_{L^2} \big)^{n-1}(\nab u,\nab G(u)dW),
		\end{align*}
		whereas
		\begin{align*}
			I_2& \le C \big(\|\nab u\|^2_{L^2}+\|\triangle u\|^2_{L^2} \big)^{n-2}\big|(\nab u,\nab G(u))\big|^2 dt \\
			&\le C\big(\|\nab u\|^2_{L^2}+\|\triangle u\|^2_{L^2} \big)^{n-2}\|\nab u\|^2_{L^2}\|\nab G(u)\|^2_{L^2} dt \\
			&\le C\big(\|\nab u\|^2_{L^2}+\|\triangle u\|^2_{L^2} \big)^{n}dt+C dt.
		\end{align*}
		It follows that
		\begin{align} \label{ineq:d(|nab.u|^2+|triangle.u|^2)^n}
			d \big(\|\nab u\|^2_{L^2}+\|\triangle u\|^2_{L^2} \big)^{n}
			&\le C \big(\|\nab u\|^2_{L^2}+\|\triangle u\|^2_{L^2} \big)^{n} dt +Cdt +C \|\nab u\|^{6n}_{L^2}dt+dM_n,
		\end{align}
		where $M_n$ is the martingale process given by
		\begin{align*}
			dM_n = 2n \big(\|\nab u\|^2_{L^2}+\|\triangle u\|^2_{L^2} \big)^{n-1}(\nab u,\nab G(u)dW).
		\end{align*}
		We take expectations on both sides of the above inequality to obtain the bound
		\begin{align*}
			\E\big[\big(\|\nab u(t)\|^2_{L^2}+\|\triangle u(t)\|^2_{L^2} \big)^{n}\big] & \le \E\big[\big(\|\nab u_0\|^2_{L^2}+\|\triangle u_0\|^2_{L^2} \big)^{n}\big]+C\int_0^t \E\big[ \|\nab u(s)\|^{6n}_{L^2}  \big]ds +Ct\\
			& \quad +C\int_0^t \E\big[ \big(\|\nab u(s)\|^2_{L^2}+\|\triangle u(s)\|^2_{L^2} \big)^{n}  \big]ds.
		\end{align*}
		Recalling estimate \eqref{ineq:moment-bound:H1:poly-moment}, we readily have
		\begin{align*}
			\int_0^t \E\big[ \|\nab u(s)\|^{6n}_{L^2}  \big]ds \le C(T) \big(\E\|u_0\|^{6n}_{H^1} +1 \big),
		\end{align*}
		whence
		\begin{align*}
			\E\big[\big(\|\nab u(t)\|^2_{L^2}+\|\triangle u(t)\|^2_{L^2} \big)^{n}\big] & \le \E\big[\big(\|\nab u_0\|^2_{L^2}+\|\triangle u_0\|^2_{L^2} \big)^{n}\big]+C \big(\E\|u_0\|^{6n}_{H^1} +1 \big)+Ct\\
			&\qquad+C\int_0^t \E\big[ \big(\|\nab u(s)\|^2_{L^2}+\|\triangle u(s)\|^2_{L^2} \big)^{n}  \big]ds.
		\end{align*}
		An application of Gronwall's inequality implies that
		\begin{align} \label{ineq:E(|nab.u|]^2+|triangle.u|^2)^n}
			\E\big[\big(\|\nab u(t)\|^2_{L^2}+\|\triangle u(t)\|^2_{L^2} \big)^{n}\big] \le C( \E\big[\|\nab u_0\|^{6n}_{L^2}+\|\triangle u_0\|^{2n}_{L^2} \big]+1),\quad 0\le t\le T,
		\end{align}
		where $C=C(T,n)$ is a positive constant independent of $t$ and $u_0$. 
		
		Turning back to \eqref{ineq:moment-bound:H2:poly-moment}, we note that the process $M_n$ from \eqref{ineq:d(|nab.u|^2+|triangle.u|^2)^n} satisfies
		\begin{align*}
			d\langle M_n\rangle  & =  4n^2\big(\|\nab u\|^2_{L^2}+\|\triangle u\|^2_{L^2} \big)^{2n-2}\big|(\nab u,\nab G(u))\big|^2dt\\
			&\le C\big(\|\nab u\|^2_{L^2}+\|\triangle u\|^2_{L^2} \big)^{2n}dt+Cdt.
		\end{align*}
		We invoke \eqref{ineq:E(|nab.u|]^2+|triangle.u|^2)^n} to infer that
		\begin{align*}
			\E\big[ \langle M_n(t)\rangle \big] & \le Ct+\int_0^t \E \big[\big(\|\nab u(s)\|^2_{L^2}+\|\triangle u(s)\|^2_{L^2} \big)^{2n}\big]ds\\
			&\le C( \E\big[\|\nab u_0\|^{12n}_{L^2}+\|\triangle u_0\|^{4n}_{L^2} \big]+1).
		\end{align*}
		We combine the above estimate with \eqref{ineq:d(|nab.u|^2+|triangle.u|^2)^n} and Burkholder's inequality to deduce
		\begin{align*}
			&\E\Big[\sup_{s\in[0,t]} \big(\|\nab u(s)\|^2_{L^2}+\|\triangle u(s)\|^2_{L^2} \big)^{n} \Big]\\
			& \le \E\big[\big(\|\nab u_0\|^2_{L^2}+\|\triangle u_0\|^2_{L^2} \big)^{n}\big]+ C\big(\E\big[\|\nab u_0\|_{L^2}^{6n}\big]+1\big)+\E\big[\langle M_n(t)\rangle\big]\\
			&\qquad +C\int_0^t\E\Big[\sup_{\ell\in[0,s]} \big(\|\nab u(\ell)\|^2_{L^2}+\|\triangle u(\ell)\|^2_{L^2} \big)^{n} \Big]ds\\
			&\le C( \E\big[\|\nab u_0\|^{12n}_{L^2}+\|\triangle u_0\|^{4n}_{L^2} \big]+1)+ C\int_0^t\E\Big[\sup_{\ell\in[0,s]} \big(\|\nab u(\ell)\|^2_{L^2}+\|\triangle u(\ell)\|^2_{L^2} \big)^{n} \Big]ds.
		\end{align*}
		In turn, Gronwall's inequality yields
		\begin{align*}
			\E\Big[\sup_{s\in[0,t]} \big(\|\nab u(s)\|^2_{L^2}+\|\triangle u(s)\|^2_{L^2} \big)^{n} \Big] \le C( \E\big[\|\nab u_0\|^{12n}_{L^2}+\|\triangle u_0\|^{4n}_{L^2} \big]+1),\quad 0\le t\le T.
		\end{align*}
		This establishes \eqref{ineq:moment-bound:H2:poly-moment}, thereby finishing the proof.
		
	\end{proof}
	
	\medskip
	
	Using the previously established stability estimates, we derive time-regularity
	(Hölder continuity) estimates for the variational solution, which are essential
	for the subsequent error analysis.

	\begin{lemma} \label{lem:Lipschitz:H^1:u}
		Letting $n\ge 1$, suppose that $u_0\in L^{12n}(\Omega;H^1(D))\cap L^{4n}(\Omega;H^2(D))$ and that condition \eqref{Assump_Lipschitz} holds. Then, for all $T\ge 0$, the following holds
		\begin{align} \label{ineq:Lipschitz:H^1:u}
			\E\big[\|u(t)-u(s)\|^{2n}_{H^1} \big]\le C_{3,n}(t-s)^{n},\quad \quad 0\le s\le t\le T,
		\end{align}
		where $C_{3,n} = C(T,n)\Big(\E\big[\|u_0\|^{12n}_{H^1}+\|u_0\|^{4n}_{H^2}\big]+1\Big)$.
	\end{lemma}

	Owing to the presence of the multiplicative noise, we follow the framework of \cite{chen2023multivalued,wang2009random} to establish H\"older property for the solution $u$. To do so, we consider the linear equation
	\begin{align} \label{eqn:z}
		d z - d(\triangle z) = G(u) dW(t),\quad z(0)=0. 
	\end{align}
	Setting $v=u-z$, observe that $v$ solves the equation
	\begin{align} \label{eqn:v}
		\frac{d}{dt} v - \frac{d}{dt}(\triangle v)+\nu u +\div F(u)= 0,\quad v(0)=u_0. 
	\end{align}
	The two main ingredients to establish Lemma \ref{lem:Lipschitz:H^1:u} are stated below through Lemmas \ref{lem:Lipschitz:H^1:z} and \ref{lem:Lipschitz:H^1:v}.
	
	\begin{lemma} \label{lem:Lipschitz:H^1:z}
		Letting $n\ge 1$, suppose that $u_0\in L^{2n}(\Omega;H^1(D))$ and that condition \eqref{Assump_Lipschitz} holds. Then, for all $T\ge 0$, there exists a positive constant $C=C(T,n)$ independent of $u_0$ such that
		\begin{align} \label{ineq:Lipschitz:H^1:z}
			\E\big[\|z(t)-z(s)\|^{2n}_{H^1} \big]\le C(t-s)^{n}\big(\E\|u_0\|_{L^2}^{2n}+1\big),\quad 0\le s\le t\le T.
		\end{align}
	\end{lemma}

	\begin{lemma} \label{lem:Lipschitz:H^1:v}
		Letting $n\ge 1$, suppose that $u_0\in L^{12n}(\Omega;H^1(D))\cap L^{4n}(\Omega;H^2(D))$ and that condition \eqref{Assump_Lipschitz} holds. Then, for all $T\ge 0$, there exists a positive constant $C=C(T,n)$ independent of $u_0$ such that
		\begin{align} \label{ineq:Lipschitz:H^1:v}
			\E\big[\|v(t)-v(s)\|^{2n}_{H^1} \big]\le C(t-s)^{2n}\Big(\E\big[\|u_0\|^{12n}_{H^1}+\|u_0\|^{4n}_{H^2}\big]+1\Big),\quad \quad 0\le s\le t\le T.
		\end{align}
	\end{lemma}
	
	For the sake of clarity, the proofs of Lemmas \ref{lem:Lipschitz:H^1:z} and \ref{lem:Lipschitz:H^1:v} will be deferred after a few lines. Assuming their results, let us conclude Lemma \ref{lem:Lipschitz:H^1:u}.
	\begin{proof}[Proof of Lemma \ref{lem:Lipschitz:H^1:u}] 
		In view of Lemmas \ref{lem:Lipschitz:H^1:z} and \ref{lem:Lipschitz:H^1:v}, we readily have
		\begin{align*}
			\E\big[\|u(t)-u(s)\|^{2n}_{H^1} \big] & \le C\Big(\E\big[\|z(t)-z(s)\|^{2n}_{H^1} \big]+ \E\big[\|v(t)-v(s)\|^{2n}_{H^1} \big]\Big)\\
			&\le C(t-s)^{n}\big(\E\|u_0\|_{L^2}^{2n}+1\big)+  C(t-s)^{2n}\Big(\E\big[\|u_0\|^{12n}_{H^1}+\|u_0\|^{4n}_{H^2}\big]+1\Big)\\
			&\le C(t-s)^{n}\Big(\E\big[\|u_0\|^{12n}_{H^1}+\|u_0\|^{4n}_{H^2}\big]+1\Big).
		\end{align*}
		The proof is thus finished.
	\end{proof}

	\begin{proof}[Proof of Lemma \ref{lem:Lipschitz:H^1:z}] 
		Recalling the basis $\{e_k\}_{k\in \mZ^2 }$ and the eigenvalues $\{\alpha_k\}_{k\in \mZ^2}$ from \eqref{eqn:A.e_k=alpha_k.e_k}, we observe that the process $z$ solving \eqref{eqn:z} satisfies the identity
		\begin{align} \label{eqn:|z(t)-z(s)|^2n_H1}
			\|z(t)-z(s)\|^{2n}_{H^1} = \Big\|\int_s^t \M^{-1}G(u(\ell))d\ell \Big\|^{2n}_{H^1} = \Big(\sum _{k\in\mZ^2}\frac{\alpha_k}{(1+\alpha_k)^2}\Big| \int_s^t (G(u(\ell)),e_k)dW(\ell)\Big|^2\Big)^n.
		\end{align}
		There are two cases to be considered depending on the value of $n$. 
		
		Case 1: $n=1$. In this case, we invoke It\^o's isometry to compute
		\begin{align*}
			\E\big[\|z(t)-z(s)\|^{2}_{H^1}\big] &= \sum _{k\ge 1}\frac{\alpha_k}{(1+\alpha_k)^2}\E\Big| \int_s^t (G(u(\ell)),e_k)dW(\ell)\Big|^2 \\& =\sum _{k\ge 1}\frac{\alpha_k}{(1+\alpha_k)^2}\E\int_s^t| (G(u(\ell)),e_k)|^2d \ell\\
			&\le \int_s^t \E\|G(u(\ell))\|^2_{L^2}d\ell.
		\end{align*}
		This together with estimate \eqref{cond:|G(u)|_Hm<|u|_Hm} produces 
		\begin{align*}
			\E\big[\|z(t)-z(s)\|^{2}_{H^1}\big] \le C(t-s)\Big(\E\Big[\sup_{\ell\in[0,T]}\|u(\ell)\|^2_{L^2}\Big]+1\Big).
		\end{align*}
		In view of Lemma \ref{lem:moment-bound:H1}, cf. \eqref{ineq:moment-bound:H1:poly-moment}, we obtain
		\begin{align*}
			\E\big[\|z(t)-z(s)\|^{2}_{H^1}\big] \le C(t-s)\Big(\E\big[\|u_0\|^2_{L^2}\big]+1\Big).
		\end{align*}
		This establishes \eqref{ineq:Lipschitz:H^1:z} for the case $n=1$.
		
		Case 2: $n>1$. In this case, we apply Holder's inequality to the right-hand side of \eqref{eqn:|z(t)-z(s)|^2n_H1} and obtain the bound
		\begin{align*}
			\|z(t)-z(s)\|^{2n}_{H^1} \le \Big( \sum_{k\in\mZ^2}\frac{1}{(1+\alpha_k)^{\frac{n}{n-1}}} \Big)^{n-1} \Big(\sum _{k\in\mZ^2}\Big| \int_s^t (G(u(\ell)),e_k)dW(\ell)\Big|^{2n}\Big).
		\end{align*}
		It follows immediately that 
		\begin{align*}
			\E\big[ \|z(t)-z(s)\|^{2n}_{H^1}\big] \le C \Big(\sum _{k\in\mZ^2}\E\Big[\Big| \int_s^t (G(u(\ell)),e_k)dW(\ell)\Big|^{2n}\Big]\Big).
		\end{align*}
		We combine Burkholder's and Holder's inequalities to further estimate
		\begin{align*}
			&\E\Big[\Big| \int_s^t (G(u(\ell)),e_k)dW(\ell)\Big|^{2n}\Big]\\
			&\le C \E \Big[\Big(\int_s^t \big|(G(u(\ell)),e_k)\big|^2d\ell\Big)^n\Big] \le C(t-s)^{n-1}\int_s^t \E\big[\big|(G(u(\ell)),e_k)\big|^{2n}\big]d\ell,
		\end{align*}
		whence
		\begin{align*}
			\E\big[ \|z(t)-z(s)\|^{2n}_{H^1}\big] &\le C(t-s)^{n-1}\int_s^t\E\Big[ \sum_{k\in\mZ^2}\big|(G(u(\ell)),e_k)\big|^{2n}\Big]d\ell\\
			&\le C(t-s)^{n-1}\int_s^t\E \big[\|G(u(\ell))\|^{2n}_{L^2}\big]d\ell\\
			&\le C(t-s)^n \sup_{\ell\in[0,T]}\E \big[\|G(u(\ell))\|^{2n}_{L^2}\big].
		\end{align*}
		We once again invoke \eqref{cond:|G(u)|_Hm<|u|_Hm} and \eqref{ineq:moment-bound:H1:poly-moment} to deduce 
		\begin{align*}
			\E\big[ \|z(t)-z(s)\|^{2n}_{H^1}\big] &\le C(t-s)^n \Big( \sup_{\ell\in[0,T]}\E \big[\|u(\ell)\|^{2n}_{L^2}\big]+1 \Big)\\
			&\le C(t-s)^n \Big( \E \big[\|u_0\|^{2n}_{H^1}\big]+1 \Big).
		\end{align*}
		This establishes \eqref{ineq:Lipschitz:H^1:z} for the case $n>1$. The proof is thus finished.
	\end{proof}

	\begin{proof}[Proof of Lemma \ref{lem:Lipschitz:H^1:v}]
		From equation \eqref{eqn:v}, for $0\le s\le t\le T$, we use Holder's inequality to see that
		\begin{align*}
			\|v(t)-v(s)\|^2_{H^1} & =\Big\|\int_s^t \M^{-1}\big(-\nu u(\ell)+ \div F(u(\ell))\big)d\ell\Big\|^2_{H^1} \\
			& \le C(t-s)\int_s^t \big\|\M^{-1}\big( u(\ell)\big)\big\|^2_{H^1}+\big\|\M^{-1}(\div F(u(\ell)))\big\|^2_{H^1}d\ell .
		\end{align*}
		Since $\|\M^{-1}u\|_{H^1}\le \|u\|_{L^2}$, we get
		\begin{align*}
			\|v(t)-v(s)\|^{2n}_{H^1} & \le C (t-s)^n\Big|\int_s^t \|u(\ell)\|^2_{L^2}+\big\|\div F(u(\ell))\big\|^2_{L^2}d\ell\Big|^n\\
			&\le C(t-s)^{2n-1}\int_s^t\|u(\ell)\|^{2n}_{L^2}+ \big\|\div F(u(\ell))\big\|^{2n}_{L^2}d\ell.
		\end{align*}
		To estimate the intergrand on the above right-hand side, we note that
		\begin{align*}
			\big\|\div F(u(\ell))\big\|^{2}_{L^2} &\le 2\big(\|(\partial_x  +\partial_y) u\|^2_{L^2}  + \|u\,(\partial_x  +\partial_y) u\|^2_{L^2}   \big)\\
			&\le 2\|u\|^2_{L^\infty}\big(\|(\partial_x  +\partial_y) u\|^2_{L^2}+1\big).
		\end{align*}
		We invoke Agmon's inequality, i.e.,   $\|u\|^2_{L^\infty} \le C\|u\|_{L^2}\|\triangle u\|_{L^2}$, to further deduce
		\begin{align*}
			\big\|\div F(u(\ell))\big\|^{2}_{L^2}  & \le C\|u\|_{L^2}\|\triangle u\|_{L^2}\big(\|(\partial_x  +\partial_y) u\|^2_{L^2}+1\big) \\
			&\le C(\|\nab u\|^6_{L^2}+\|\triangle u\|^2_{L^2}+1),
		\end{align*}
		whence
		\begin{align*}
			\|v(t)-v(s)\|^{2n}_{H^1} \le C(t-s)^{2n} \sup_{\ell\in[0,T]}\Big(\|\nab u(\ell)\|^{6n}_{L^2}+\|\triangle u(\ell)\|^{2n}_{L^2}+1\Big).
		\end{align*}
		In view of \eqref{ineq:moment-bound:H1:poly-moment} and \eqref{ineq:moment-bound:H2:poly-moment}, we arrive at
		\begin{align*}
			\E\big[ \|v(t)-v(s)\|^{2n}_{H^1}\big] & \le C(t-s)^{2n} \Big(\E\Big[\sup_{\ell\in[0,T]}\Big(\|\nab u(\ell)\|^{6n}_{L^2}+\|\triangle u(\ell)\|^{2n}_{L^2}\Big]+1\Big)\\
			&\le C(t-s)^{2n}\Big(\E\big[\|u_0\|^{12n}_{H^1}+\|u_0\|^{4n}_{H^2}\big]+1\Big).
		\end{align*}
		Since $C=C(T,n,\nu)$ does not depend on $u_0$ and $s,t$, we establish \eqref{ineq:Lipschitz:H^1:v}, as claimed.

	\end{proof}
	
	Lastly, we turn to the exponential moment bound on the solution $u$ provided that the noise is bounded and that the damping parameter $\nu$ is strictly positive.
	
	\begin{lemma} \label{lemma_expo_moment_H}
		Suppose that $\nu>0$, $u_0\in L^2(\Omega;H^1(D))$ and that conditions \eqref{Assump_Lipschitz} and \eqref{Assump_Lineargrowth} hold. Then, for all $\gamma\in(0,1)$ sufficiently small and for all $T>0$, the following holds
		\begin{align} \label{ineq:expo_moment_H}
			\E\Big[ \exp \Big\{ \gamma \sup_{t\in[0,T]}\|u(t)\|^2_{H^1} \Big\} \Big] \le C_{exp,1} ,
		\end{align}
		where $C_{exp,1}=2 \E\Big[ \exp\Big\{  2\gamma \big(L_0^2 T+\|u_0\|^2_{H^1}\big)  \Big\} \Big]$.
	\end{lemma}
	\begin{proof} We aim to establish \eqref{ineq:expo_moment_H} using the exponential martingale inequality. To see this, we recall that for a given martingale process $M(t)$ and positive constants $\lambda,R$, it holds that
		\begin{align*} 
			\P\Big( \sup_{t\ge 0}\Big[M(t)-\frac{1}{2}\lambda \langle M\rangle(t)\Big]\ge R\Big) \le e^{-\lambda R},\quad \lambda>0,\, R>0.
		\end{align*}
		As a consequence, for $\lambda\ge 2$, we deduce
		\begin{align} \label{ineq:expo-martingale}
			\E \Big[ \exp \Big\{\sup_{t\ge 0}\Big[M(t)-\frac{1}{2}\lambda \langle M\rangle(t)\Big]\Big\}  \Big] \le 1 +\int_1^\infty e^{-\lambda \log R}dR \le 2. 
		\end{align}
		Now, from \eqref{eqn:d(|u|^2_L2+|nab.u|^2_L2)} with condition \eqref{Assump_Lineargrowth}, we have  
		\begin{align*} 
			d\big(\|u\|^2_{L^2}+\|\nab u\|^2_{L^2}\big) + 2\nu\|u\|^2_{L^2}dt & =  \|\M^{-\frac{1}{2}}G(u)\|^2_{L^2}dt + 2\big(u,G(u)dW\big) \\
			&\le  L_0^2 dt + 2\big(u,G(u)dW\big),
		\end{align*}
		where $L_0$ is the constant as in \eqref{Assump_Lineargrowth}. Setting $M(t) = 2\int_0^t \big(u,G(u)dW\big)$, the corresponding quadratic variation process satisfies
		\begin{align*}
			d\langle M\rangle(t)  =  4|(u,G(u))|^2\le 4L_0^2 \|u\|^2_{L^2}.
		\end{align*}
		It follows that for all $\gamma>0$,
		\begin{align*}
			d\gamma\big(\|u\|^2_{L^2}+\|\nab u\|^2_{L^2}\big) - \gamma L_0^2 dt & \le   d  \gamma M(t) - \frac{1}{2}\cdot \frac{\nu}{\gamma L_0^2} \cdot d \gamma^2\langle M\rangle(t),
		\end{align*}
		whence
		\begin{align*}
			& \gamma\big(\|u(t)\|^2_{L^2}+\|\nab u(t)\|^2_{L^2}\big) - \gamma L_0^2 t-\gamma\big(\|u_0\|^2_{L^2}+\|\nab u_0\|^2_{L^2}\big) \\
			& \le \gamma M(t) - \frac{1}{2}\cdot \frac{\nu}{\gamma L_0^2} \cdot \gamma^2\langle M\rangle(t).
		\end{align*}
		In view of \eqref{ineq:expo-martingale}, provided that $\gamma$ is sufficiently small, e.g.,
		\begin{align*}
			\frac{\nu}{\gamma L_0^2} \ge 2,
		\end{align*}
		we obtain
		\begin{align*}
			\E\Big[ \exp\Big\{ \sup_{t\in[0,T]}\Big[  \gamma\big(\|u(t)\|^2_{L^2}+\|\nab u(t)\|^2_{L^2}\big) - \gamma L_0^2 t-\gamma\big(\|u_0\|^2_{L^2}+\|\nab u_0\|^2_{L^2}\big) \Big]   \Big\} \Big] \le 2.
		\end{align*}
		We invoke the Holder inequality to deduce further that
		\begin{align*}
			\E\Big[ \exp\Big\{  \frac{1}{2}\sup_{t\in[0,T]}\Big[ \gamma\big(\|u(t)\|^2_{L^2}+\|\nab u(t)\|^2_{L^2}\big)\Big]   \Big\} \Big]
			&\le  2 \E\Big[ \exp\Big\{  \gamma \big(L_0^2 T+\|u_0\|^2_{L^2}+\|\nab u_0\|^2_{L^2}\big)  \Big\} \Big]  .
		\end{align*}
		This establishes \eqref{ineq:expo_moment_H}, as claimed.
	\end{proof}
	
	\section{Fully discrete finite element method} \label{sec-4}
	\subsection{Formulation and stability of the fully discrete finite element method}
	Let $\mathcal{T}_h$ be the uniform triangular mesh over $D$. We denote $h$ as the mesh size over the domain $D$. In addition, for any positive integer $M$, we define $k =  T/M$ as the time step size and so $t_{n+1} = t_n  + k$ for all $n = 0, 1, \cdots, M-1$.  We consider the associated finite element space, generated by piecewise linear polynomials:
	\begin{align}
		V_h = \left\{ \varphi \in H^1(D) : \varphi_{|_{K}} \in \mathbb{P}_{1}, \,\, K \in \mathcal{T}_h \right\},
	\end{align}
	where $\mathbb{P}_1$ denotes the set of polynomials of degree at most $1$.

	Let $P_h$ denote the elliptic projection from $H^1(D)$ into $V_h$, which satisfies
	\begin{align}\label{elliptic_projection}
		(P_h v - v, \phi) + (\nab(P_hv - v), \nab \phi) = 0\qquad\forall\phi \in V_h.
	\end{align}
	We also recall the following inequality, whose proof can be found in \cite{brenner2008mathematical}: 
	\begin{align}\label{projection_ineq}
		\|v -  P_h v\|_{L^2} + h \|\nab(v - P_h v)\|_{L^2} \leq Ch^{2}\|v\|_{H^2}, \qquad\forall v \in H^2(D).
	\end{align}
	
	\bigskip
	
	{\bf Main Algorithm.}
	For $n= 0, 1, ..., M-1$, we seek $u_h^{n+1} \in V_h$, such that
	\begin{align}\label{Scheme_Standard}
		\bigl(u_h^{n+1} - u_h^{n}, \phi_h\bigr) &+ \bigl(\nab(u_h^{n+1} - u_h^{n}), \nab \phi_h\bigr)   + \nu k (u_h^{n+1}, \phi_h) \\\nonumber
		&\qquad= k(F(u_h^{n+1}), \nab\phi_h) + (G(u_h^n)\Delta W_n,\phi_h) \quad\forall \phi_h \in V_h,
	\end{align}
	where $\Delta W_n = W(t_{n+1}) - W(t_n) \sim \mathcal{N}(0, k)$ and \[F(u^{n+1}_h) = \langle u_h^{n+1} + \frac12 (u_h^{n+1})^2, u_h^{n+1} + \frac12 (u_h^{n+1})^2 \rangle^T.\]

	\bigskip
	
	In what follows, we state and prove stability estimates for $\{u_h^n\}$. 
	
	\begin{lemma}\label{lemma_2ndmoment_discrete}  Suppose that $u_0 \in L^2(\Omega; H^1(D))$ and that $G$ satisfies \eqref{Assump_Lipschitz} with $m = 0$. Then, the fully discrete solution $\{u_h^n\}$ satisfies
		\begin{align}\label{eq3.55}
			\mE\left[\sup_{1\leq n \leq M}\|u_h^n\|^2_{H^1} %+ \sum_{n = 1}^{M}\|u_h^{n} - u_h^{n-1}\|^2_{H^1}
			\right] \leq C_{4,1},
		\end{align}
		where $C_{4,1} = C(u_0,T,c_G)>0$.
	\end{lemma}
	
	\begin{proof}
		Taking $\phi_h = u_h^{n+1}$ in \eqref{Scheme_Standard} and using the identity $2a(a-b) = a^2 - b^2 + (a-b)^2$, we obtain
		\begin{align}\label{eq3.4}
			&	\frac12\left[\|u_h^{n+1}\|^2_{H^1} - \|u_h^n\|^2_{H^1} + \|u_h^{n+1} - u_h^n\|^2_{H^1}\right] + \nu k \|u_h^{n+1}\|^2_{L^2} \\\nonumber
			&= k\bigl(F(u_h^{n+1}), \nab u_h^{n+1}\bigr) + (G(u_h^n)\Delta W_n,u_h^{n+1})
			\\\nonumber
			&:= I_1 + I_2.
		\end{align}
		Concerning $I_1$, using the formula of $F$ and integration by parts, we note that
		\begin{align*}
			I_1 &= k \bigl(u_h^{n+1}, \p_x u_h^{n+1}\bigr) + k\bigl(u_h^{n+1}, \p_y u_h^{n+1}\bigr) + \frac{k}{2}\bigl((u_h^{n+1})^2, \p_x u_h^{n+1}\bigr) +  \frac{k}{2}\bigl((u_h^{n+1})^2, \p_y u_h^{n+1}\bigr)\\\nonumber
			&=0.
		\end{align*}
		To estimate $I_2$, we employ Cauchy-Schwarz's inequality and the condition \eqref{Assump_Lipschitz} with $m=0$ as follows:
		\begin{align}\label{eq3.6}
			I_2 &= (G(u_h^n)\Delta W_n,u_h^{n+1} - u_h^n) +  (G(u_h^n)\Delta W_n,u_h^{n})\\\nonumber
			&\leq \frac{1}{4}\|u_h^{n+1} - u_h^n\|^2_{L^2} + \|G(u_h^n)\Delta W_n\|^2_{L^2} +  (G(u_h^n)\Delta W_n,u_h^{n})\\\nonumber
			&\leq \frac{1}{4}\|u_h^{n+1} - u_h^n\|^2_{L^2} + \|G(u_h^n)\|^2_{L^2} |\Delta W_n|^2+  (G(u_h^n)\Delta W_n,u_h^{n})\\\nonumber
			&\leq \frac{1}{4}\|u_h^{n+1} - u_h^n\|^2_{H^1} + c_G^2\|u_h^n\|^2_{L^2} |\Delta W_n|^2 +  (G(u_h^n)\Delta W_n,u_h^{n}),
		\end{align}
		Substituting all the estimates from $I_1, I_2$ into \eqref{eq3.4}, we arrive at
		\begin{align}\label{eq3.10}
			&\frac12\left[\|u_h^{n+1}\|^2_{H^1} - \|u_h^n\|^2_{H^1} + \frac12\|u_h^{n+1} - u_h^n\|^2_{H^1}\right] + \nu k \|u_h^{n+1}\|^2_{L^2} \notag \\
			&\qquad\leq  c_G^2\|u_h^n\|^2_{L^2}|\Delta W_n|^2 +  (G(u_h^n)\Delta W_n,u_h^{n}).
		\end{align}
		Applying the summation $\sum_{n=0}^{\ell}$, for $0\leq \ell <M$, we get
		\begin{align}\label{eq3.11}
			&	\frac12\|u_h^{\ell+1}\|^2_{H^1} + \frac14\sum_{n=0}^{\ell}\|u_h^{n+1} - u_h^n\|^2_{H^1} + \nu k \sum_{n = 0}^{\ell} \|u_h^{n+1}\|^2_{L^2}\\\nonumber
			&\leq \frac12\|u^0_h\|^2_{H^1}  + c_G^2\sum_{n=0}^{\ell}\|u_h^{n}\|^2_{L^2}|\Delta W_n|^2 + \sum_{n = 0}^{\ell}(G(u_h^n)\Delta W_n,u_h^{n}).
		\end{align}
		
		Next, taking expectations on both sides of \eqref{eq3.11}, and using the fact that
		\begin{align*}
			\mE\left[\|u_h^n \|^2_{L^2} |\Delta W_n|^2\right] = \mE\left[\|u_h^n\|^2_{L^2}\right]\mE\left[|\Delta W_n|^2\right] = k \mE\left[\|u_h^n\|^2_{L^2}\right],
		\end{align*}
		and that\begin{align*}
			\mE\left[(G(u_h^n)\Delta W_n,u_h^{n})\right] = 0,
		\end{align*}
		we obtain
		\begin{align}\label{eq3.19}
			&	\mE\left[\|u_h^{\ell+1}\|^2_{H^1} \right]+ \nu k \sum_{n = 0}^{\ell}\mE\left[\|u_h^{n+1}\|^2_{L^2}\right]\leq 2\mE\left[\|u^0_h\|^2_{H^1}\right] + 4c_G^2k\sum_{n=0}^{\ell}\mE\left[\|u_h^{n}\|^2_{L^2}\right].
		\end{align}
		As a consequence, applying the discrete deterministic Gronwall inequality on \eqref{eq3.19}, we get
		\begin{align}\label{eq_3.19}
			\mE[\|u_h^{\ell+1}\|^2_{H^1}] + \nu k \sum_{n = 0}^{\ell}\mE\left[\|u_h^{n+1}\|^2_{L^2}\right]\leq 2\mE\left[\|u^0_h\|^2_{H^1}\right]e^{4c_G^2T}.
		\end{align}
		
		Now, we aim to use \eqref{eq_3.19} to establish \eqref{eq3.55}. To see this, we take $\sup_{0\leq \ell \leq M-1}$ and then expectations on both sides of \eqref{eq3.11} to obtain 
		\begin{align*}
			&\mE\left[\sup_{0 \leq \ell \leq M-1}\|u_h^{\ell+1}\|^2_{H^1}\right]\\
			&\leq 2\mE\left[\|u^0_h\|^2_{H^1}\right] + 4c_G^2k\sum_{n=0}^{M}\mE\left[\|u_h^{n}\|^2_{L^2}\right]  + 4\mE\left[\sup_{0 \leq \ell \leq M-1}\sum_{n = 0}^{\ell}(G(u_h^n)\Delta W_n,u_h^{n})\right].
		\end{align*}
		Using the Burkholder-Davis-Gundy inequality, condition \eqref{Assump_Lipschitz} with $m=0$, and estimate \eqref{eq_3.19}, we obtain 
		\begin{align}\label{eq3.155}
			&\mE\left[\sup_{0\leq \ell \leq M-1}\|u_h^{\ell+1}\|^2_{H^1}\right] \\\nonumber
			&\leq 2\mE\left[\|u^0_h\|^2_{H^1}\right] + 4c_G^2k\sum_{n=0}^{M}\mE\left[\|u_h^{n}\|^2_{L^2}\right]  + 4c_G\mE\left[\left(k\sum_{n = 0}^{M}\|u_h^n\|^2_{L^2}\|u_h^{n}\|^2_{L^2}\right)^{1/2}\right]\\\nonumber
			&\leq 2\mE\left[\|u^0_h\|^2_{H^1}\right] + 4c_G^2k\sum_{n=0}^{M}\mE\left[\|u_h^{n}\|^2_{L^2}\right]  + 4c_G\mE\left[\sup_{0 \leq \ell \leq M}\|u_h^{\ell}\|_{L^2}\left(k\sum_{n = 0}^{M}\|u_h^n\|^2_{L^2}\right)^{1/2}\right]\\\nonumber
			&\leq 2\mE\left[\|u^0_h\|^2_{H^1}\right] + 4c_G^2k\sum_{n=0}^{M}\mE\left[\|u_h^{n}\|^2_{L^2}\right]  + 8c_G^2\mE\left[k\sum_{n = 0}^{M}\|u_h^n\|^2_{L^2}\right] + \frac{1}{2}\mE\left[\sup_{0 \leq \ell \leq M}\|u_h^{\ell}\|_{L^2}^2\right]\\\nonumber
			&\leq 2\mE\left[\|u^0_h\|^2_{H^1}\right] +  24c_G^2T\mE[\|u^0_h\|^2_{H^1}]e^{4c_G^2T} + \frac{1}{2}\mE\left[\sup_{0 \leq \ell \leq M}\|u_h^{\ell}\|_{L^2}^2\right].
		\end{align}
		The last term $\frac{1}{2}\mE\left[\sup_{0 \leq \ell \leq M}\|u_h^{\ell}\|_{L^2}^2\right]$ on the right-hand side of \eqref{eq3.155} will be absorbed to the first term on the left-hand side of \eqref{eq3.155}. The proof is thus complete.

	\end{proof}
	
	\medskip
	
	Next, we show that the fully discrete solution satisfies high-moment stability estimates, which are essential for the subsequent error analysis.

	\begin{lemma}\label{lemma_highmoment_discrete} Suppose that $u_0 \in L^{2^p}(\Omega; H^1(D))$ for $p=1,2,3,$ and that $G$ satisifies \eqref{Assump_Lipschitz} with $m=0$. Then, the fully discrete solution $\{u_h^n\}$ satisfies
		\begin{align}\label{eq_3.11}
			\mE\left[\sup_{1\leq n \leq M}\|u_h^n\|^{2^p}_{H^1} \right] \leq C_{4,p},
		\end{align}
		where {$C_{4,p}= C(u_0,T, c_G, p)>0$. }
	\end{lemma}
	\begin{proof} 
		We note that the case $p = 1$ was established in Lemma \ref{lemma_2ndmoment_discrete}. We only need to give the proofs for the cases $p = 2$ and $p=3$. 
		
		Multiplying \eqref{eq3.10} by $\|u_h^{n+1}\|^2_{H^1}$ and using the identity $2a(a-b) = a^2 - b^2 + (a-b)^2$, we obtain
		\begin{align}\label{eq3.12}
			&\frac{1}{4}\left[\|u_h^{n+1}\|^4_{H^1} - \|u_h^n\|^4_{H^1}\right] + \frac14 \left(\|u_h^{n+1}\|^2_{H^1} - \|u_h^n\|^2_{H^1}\right)^2 \\\nonumber
			&\qquad\qquad+ \frac14\|u_h^{n+1} - u_h^n\|^2_{H^1}\|u_h^{n+1}\|^2_{H^1} + \nu k\| u_h^{n+1}\|^2_{L^2}\|u_h^{n+1}\|^2_{H^1}\\\nonumber
			&\leq  c_G^2\|u_h^{n}\|^2_{L^2}|\Delta W_n|^2 \|u_h^{n+1}\|^2_{H^1} + \left(G(u_h^n)\Delta W_n, u_h^n\right)\|u_h^{n+1}\|^2_{H^1}\\\nonumber
			&=c_G^2\|u_h^{n}\|^2_{L^2}|\Delta W_n|^2\left( \|u_h^{n+1}\|^2_{H^1} -  \|u_h^{n}\|^2_{H^1}\right) + c_G^2\|u_h^{n}\|^2_{L^2}\|u_h^n\|^2_{H^1}|\Delta W_n|^2 \\\nonumber
			&\qquad+ \left(G(u_h^n)\Delta W_n, u_h^n\right)\left(\|u_h^{n+1}\|^2_{H^1} - \|u_h^{n}\|^2_{H^1}\right) + \left(G(u_h^n)\Delta W_n, u_h^n\right)\|u_h^{n}\|^2_{H^1}\\\nonumber
			&\leq 4c_G^4\|u_h^{n}\|^4_{H^1}|\Delta W_n|^4 + 5c_G^2\|u_h^{n}\|^4_{H^1}|\Delta W_n|^2 \\\nonumber
			&\qquad +\frac18\left(\|u_h^{n+1}\|^2_{H^1} - \|u_h^n\|^2_{H^1}\right)^2 + \left(G(u_h^n)\Delta W_n, u_h^n\right)\|u_h^{n}\|^2_{H^1},
		\end{align}
		where the last inequality above is obtained by using the Young inequality and estimate \eqref{cond:|G(u)|_Hm<|u|_Hm} with $m=0$. Next, applying the summation $\sum_{n = 0}^{\ell}$ to \eqref{eq3.12}, we get
		\begin{align}\label{eq_3.13}
			&\frac{1}{4}\|u_h^{\ell+1}\|^4_{H^1} + \frac18 \sum_{n = 0}^{\ell}\left(\|u_h^{n+1}\|^2_{H^1} - \|u_h^n\|^2_{H^1}\right)^2 \\\nonumber
			&\leq \frac14 \|u_h^0\|^4_{H^1}  + 4c_G^4\sum_{n = 0}^{\ell}\|u_h^{n}\|^4_{H^1}|\Delta W_n|^4 + 5c_G^2\sum_{n = 0}^{\ell}\|u_h^{n}\|^4_{H^1}|\Delta W_n|^2 + \sum_{n = 0}^{\ell}  \left(G(u_h^n)\Delta W_n, u_h^n\right)\|u_h^{n}\|^2_{H^1}.
		\end{align}
		Then, taking the expectation and applying the discrete Gronwall inequality to \eqref{eq_3.13}, we obtain 
		\begin{align}\label{eq3.200}
			&\mE\left[\|u_h^{\ell+1}\|^4_{H^1} \right]+  \sum_{n = 0}^{\ell}\mE\left[\left(\|u_h^{n+1}\|^2_{H^1} - \|u_h^n\|^2_{H^1}\right)^2\right] \leq 2\mE\left[\|u_h^0\|^4_{H^1}\right] e^{(4c_G^4 + 5c_G^2)T}.
		\end{align}
		
		Now, we may use \eqref{eq3.200} to establish \eqref{eq_3.11} with $p=2$ by proceeding similarly to the steps as shown from \eqref{eq_3.19}--\eqref{eq3.155}.
		Likewise, we can obtain \eqref{eq_3.11} for $p =3$ by multiplying \eqref{eq3.12} by $\|u_h^{n+1}\|^4_{H^1}$ and proceed as shown above. The proof is complete.
		
	\end{proof}
	
	Lastly, we state and prove an exponential estimate for the discrete solution $\{u_h^n\}$. This exponential stability estimate is used in proving the main error estimates of the paper.
	
	\begin{lemma}\label{lemma_exponential_discrete}  Suppose that $u_0 \in L^2(\Omega; H^1(D))$ and that $G$ satisfies \eqref{Assump_Lineargrowth}. Then, the fully discrete solution $\{u_h^n\}$ satisfies
		\begin{align}\label{expential}
			\sup_{1\leq n \leq M}\mE\left[\exp\left(\gamma\|u_h^n\|^2_{H^1}\right)\right] \leq C_{exp, 2},
		\end{align}
		where 
		\begin{align*}
			C_{exp,2} = \Bigl(\mE\left[\exp\left(4\gamma\|u_h^0\|^2_{H^1} + 16\gamma^2 L_0^2k \|u^0_h\|^2_{L^2} + 8L_0^2\gamma\right)  \right]\Bigr)^{1/2} < \infty,
		\end{align*} and $\gamma \in \left(0, \frac{\nu}{4L_0^2}\right]$, and the time step $k \leq \frac{1}{32L_0^2\gamma}$.
	\end{lemma}
	\begin{proof}
		From \eqref{eq3.4} and \eqref{eq3.6} and the condition \eqref{Assump_Lineargrowth} of $G$, we have
		\begin{align}
			\|u_h^{n}\|^2_{H^1} &- \|u_h^{n-1}\|^2_{H^1} + 2\nu k \|u_h^{n}\|^2_{L^2} 
			\leq 2L_0^2|\Delta W_{n-1}|^2 + 2(G(u_h^{n-1})\Delta W_{n-1}, u_h^{n-1}).
		\end{align}
		Applying the summation $ \sum_{n=1}^{\ell}$ for any $1 \leq \ell \leq M$ produces
		\begin{align}\label{equu3.18}
			\|u_h^{\ell}\|^2_{H^1}  + 2\nu k \sum_{n = 1}^{\ell}\|u_h^{n}\|^2_{L^2} 
			\leq \|u_h^0\|^2_{H^1} + 2L_0^2\sum_{n = 1}^{\ell}|\Delta W_{n-1}|^2 + 2\sum_{n = 1}^{\ell}(G(u_h^{n-1})\Delta W_{n-1}, u^{n-1}_h).
		\end{align}
		
		For any $\gamma > 0$, denote $Z_{\ell} =2\gamma\sum_{n=1}^{\ell}\bigl(G(u_h^{n-1})\Delta W_{n-1}, u^{n-1}\bigr)$, for $1\leq \ell\leq M$. We introduce a martingale by using a piece-wise constant function that is defined as follows: for $s \in [t_{n-1},t_{n})$ with $n =1,2,\cdots, M$, and the convention $t_0=0$, set $\bar{u}(s)= u_h^{n-1}$ and
		\begin{align}
			\tilde{Z}_{t} = 2\gamma\int_{0}^{t}\bigl(G(\bar{u}(s)), \bar{u}(s)\bigr)\,d W(s),   \qquad\forall t\in [0,T].
		\end{align}
		Observe that $\displaystyle Z_{\ell} = \tilde{Z}_{t_{\ell}}$ and that $\tilde{Z}_t$ is a $\mathcal{F}_t$-martingale for $t \in [0,T]$ whose quadratic variation satisfies
		\begin{align}\label{equu3.21}
			\bigl<\tilde{Z}\bigr>_{t_{\ell}}  &=4\gamma^2 \int_{t_0}^t |(G(\bar{u}(s)), \bar{u}(s))|^2\, ds\\\nonumber
			&\leq 4\gamma^2\int_{t_0}^{t_{\ell}} \|G(\bar{u}(s))\|^2_{L^2}\|\bar{u}(s)\|^2_{L^2}\, ds\\\nonumber
			&\leq 4\gamma^2L_0^2\, k\sum_{n=1}^{\ell}\|u_h^{n-1}\|^2_{L^2},
		\end{align}
		We multiply \eqref{equu3.18} by $\gamma >0$ and then take the exponential on both sides to get
		\begin{align}\label{equu323}
			\exp\Bigl(\gamma\|u_h^{\ell}\|^2_{H^1}\Bigr) 
			&\leq \exp\bigl(2\gamma\|u_h^0\|^2_{H^1}\bigr)\times \exp\Bigl(2L_0^2\gamma\sum_{n=1}^{\ell}|\Delta W_{n-1}|^2\Bigr)\\\nonumber
			&\qquad\times\exp\Bigl(\bigl[Z_{t_{\ell}} -2\bigl<\tilde{Z}\bigr>_{t_{\ell}}\bigr]\Bigr)\times \exp\biggl(2\bigl<\tilde{Z}\bigr>_{t_{\ell}} - 2\gamma\nu k\sum_{n=1}^{\ell}\|u_h^n\|^2_{L^2}\biggr).
		\end{align} 
		Moreover, from \eqref{equu3.21} we have
		\begin{align}\label{equu3.24}
			2\bigl<\tilde{Z}\bigr>_{t_{\ell}} - 2\gamma\nu k\sum_{n=1}^{\ell}\|u_h^n\|^2_{L^2} 
			&\leq 8\gamma^2L_0^2\, k\sum_{n=1}^{\ell}\|u_h^{n-1}\|^2_{L^2} -  2\gamma\nu k\sum_{n=1}^{\ell}\|u_h^n\|^2_{L^2}\\\nonumber
			&\leq 2\gamma\bigl(4 L_0^2\gamma - \nu\bigr) k\sum_{n=1}^{\ell}\|u_h^n\|^2_{L^2} + 8\gamma^2 L_0^2k \|u^0_h\|^2_{L^2}.
		\end{align}	
		Now, with this and using the condition that $0<\gamma \leq \frac{\nu}{4L_0^2}$, which implies that $4 L_0^2\gamma - \nu\leq 0$, we get 
		\begin{align}\label{equu3.26}
			\exp\biggl(2\bigl<\tilde{Z}\bigr>_{t_{\ell}} - 2\gamma\nu k\sum_{n=1}^{\ell}\|u_h^n\|^2_{L^2}\biggr) \leq \exp\left(8\gamma^2 L_0^2k \|u^0_h\|^2_{L^2}\right).
		\end{align}
		Next, taking the expectation on \eqref{equu323} and using \eqref{equu3.26} we obtain
		\begin{align}\label{eq4.25}
			\mE\left[\exp\Bigl(\gamma\|u_h^{\ell}\|^2_{H^1}\Bigr) \right]
			&\leq \mE\left[\exp\left(2\gamma\|u_h^0\|^2_{H^1} + 8\gamma^2 L_0^2k \|u^0_h\|^2_{L^2}\right)\right.\\\nonumber
			&\qquad \left.\times \exp\Bigl(2L_0^2\gamma\sum_{n=1}^{\ell}|\Delta W_{n-1}|^2\Bigr) \times\exp\Bigl(\bigl[Z_{t_{\ell}} -2\bigl<\tilde{Z}\bigr>_{t_{\ell}}\bigr]\Bigr)\right].
		\end{align}
		Using the H\"older inequality to separate the product on the right-hand side of \eqref{eq4.25} yields
		\begin{align}\label{equu3.28}
			\mE\left[\exp\Bigl(\gamma\|u_h^{\ell}\|^2_{H^1}\Bigr) \right]
			&\leq \Bigl(\mE\left[\exp\left(4\gamma\|u_h^0\|^2_{H^1} + 16\gamma^2 L_0^2k \|u^0_h\|^2_{L^2}\right)\right]\Bigr)^{1/2}\\\nonumber
			&\qquad\times \biggl(\mE\Bigl[\exp\Bigl(8L_0^2\gamma\sum_{n=1}^M|\Delta W_n|^2\Bigr)\Bigr]\biggr)^{1/4} \times \Bigl(\mE\Bigl[\exp\Bigl(\bigl[4\tilde{Z}_{t_{\ell}} -\frac{1}{2}\bigl<4\tilde{Z}\bigr>_{t_{\ell}}\bigr]\Bigr)\Bigr]\Bigr)^{1/4}.
		\end{align}
		Since $\exp\Bigl(4\tilde{Z}_{t} -\frac{1}{2}\bigl<4\tilde{Z}\bigr>_{t}\Bigr)$ is an exponential martingale on $[0,T]$, it suffices to show that the expectation $\mE\Bigl[\exp\Bigl(8L_0^2\gamma\sum_{n=1}^M|\Delta W_n|^2\Bigr)\Bigr]$ is bounded.  To this end, by using the independence of the increments of $W(t)$, we can rewrite
		\begin{align}\label{equation3.23}
			\mE\Bigl[\exp\Bigl(8L_0^2\gamma\sum_{n=1}^M|\Delta W_n|^2\Bigr)\Bigr] &= %\mE\Bigl[\exp\Bigl(8L_0^2\gamma \sum_{n=1}^M (\Delta W_n)^2\Bigr)\Bigr] =
			\prod_{n=1}^M \mE\bigl[e^{8L_0^2\gamma (\Delta W_n)^2}\bigr].
		\end{align}
		Moreover, since $\Delta W_n \sim \mathcal{N}(0,k)$, a simple computation gives
		\begin{align*}
			\mE\bigl[e^{8L_0^2\gamma (\Delta W_n)^2}\bigr] = \frac{1}{\sqrt{1- 16L_0^2\gamma k}},
		\end{align*}
		provided that $1- 16L_0^2\gamma k > 0$. This condition can be fulfilled since the time step $k$ can be chosen as small as possible. For example, we may choose $k < \frac{1}{32 L_0^2\gamma}$. So, $1- 16L_0^2\gamma k \geq \frac12 >0$. Therefore, \eqref{equation3.23} is reduced to
		\begin{align}
			\mE\Bigl[\exp\Bigl(8L_0^2\gamma\sum_{n=1}^M|\Delta W_n|^2\Bigr)\Bigr] &=  (1-16 L_0^2\gamma k)^{\frac{-T}{2k}} = \exp\left(-\frac{T}{2k}\log(1-16L_0^2\gamma k)\right) \leq e^{16 L_0^2\gamma},
		\end{align}
		provide that $k < \frac{1}{32 L_0^2\gamma}$.
		Finally, putting this one into \eqref{equu3.28}, the proof is complete.
		
	\end{proof}
	
	Next, we present the error estimates of the scheme \eqref{Scheme_Standard} in two cases: bounded multiplicative noise and general multiplicative noise. 
	
	\subsection{Full expectation error estimates in the case of bounded multiplicative noise}\label{sub-sec4.2}
	
	In this section, we establish strong error estimates for the fully discrete solution ${u_h^n}$ under the assumption of bounded noise, that is, when the diffusion coefficient $G$ satisfies assumption~\eqref{Assump_Lineargrowth}. Typical examples include nonlinear diffusions such as $G(u)=\sin(u)$, $G(u)=\cos(u)$, and $G(u)=\frac{u^2}{u^2+1}$. Under this condition, the error analysis is carried out in terms of full expectations, leading to strong convergence results in appropriate $L^p$ norms.
	\subsubsection{Sub-second moment error estimates}
	First, we derive the optimal error estimates of $\{u_h^n\}$ in the  $L_{\omega}^{p}L_t^{\infty}L_x^2$- and $L_{\omega}^{p}L_t^{2}H_x^2$-norm for $0<p<2$, which are the sub-second moment error estimates.
	
	\begin{theorem}\label{Theorem_sub_moment} 
		Let $u$ be the variational solution to \eqref{Weak_formulation} and $\{u_h^{n}\}_{n=1}^M$ be generated by \eqref{Scheme_Standard}. Suppose that $G$ satisfies conditions \eqref{Assump_Lipschitz} and \eqref{Assump_Lineargrowth} and that $u_0 \in  L^{24}(\Omega; H^1(D))\cap L^{8}(\Omega; H^2(D))$. 
		Additionally, for any $0 < q < 0.9$, assume that $L_0 <\frac{\sqrt{\nu}}{4\sqrt{ 5q}}$ and $\mE\left[\exp\left(4\gamma\|u_h^0\|^2_{H^1} + 16\gamma^2 L_0^2k \|u^0_h\|^2_{L^2}\right)  \right] <\infty$, where $\gamma = 20 q$. Moreover, suppose that 
		there exists a constant $C>0$ such that
		\begin{align}
			\left(\mE\left[\|u_0 - u_h^0\|^{2q}_{H^1}\right]\right)^{\frac{1}{2q}} 
			\leq C h,
		\end{align}

		Then, it holds that
		\begin{align}\label{equu310}
			&\left(\mE\left[\max_{1\leq n\leq M}\|u(t_n) - u_h^n\|^{2q}_{H^1}\right]\right)^{\frac{1}{2q}} 
			\leq \widehat{C}_1\,\left(k^{\frac12} + h\right),
		\end{align}
		where $\widehat{C}_1 = C(q, u_0, T, C_G)$ is a positive constant.
	\end{theorem}
	
	\begin{proof}
		Without loss of generality, we assume in the proof that $u^0_h = u_0$. 
		
		Denote $e^n:= u(t_n) - u_h^n = \theta^n + \varepsilon^n$, where
		\begin{align*}
			\theta^n = u(t_n) - P_h u(t_n),\qquad \varepsilon^n = P_h u(t_n) - u_h^n.
		\end{align*}
		Subtracting \eqref{Scheme_Standard} from \eqref{Weak_formulation}, we obtain for all $\phi_h\in V_h$
		\begin{align}\label{eq3.14}
			&\bigl(e^{n+1} - e^n, \phi_h\bigr) + \bigl(\nab(e^{n+1} - e^n), \nab\phi_h\bigr) + \nu k\bigl( e^{n+1},  \phi_h\bigr) \\\nonumber
			&= \nu\int_{t_n}^{t_{n+1}} ((u(t_{n+1}) - u(s)), \phi_h)\, ds + \int_{t_n}^{t_{n+1}} (F(u(s)) - F(u_h^{n+1}), \nab \phi_h)\, ds  \\\nonumber
			&\quad + \left(\int_{t_n}^{t_{n+1}} (G(u(s)) - G(u_h^n))\, dW(s), \phi_h\right).
		\end{align}
		Using the elliptic projection orthogonality \eqref{elliptic_projection}, the left-hand side of \eqref{eq3.14} is re-written as follow:
		\begin{align*}
			\bigl(e^{n+1} - e^n, \phi_h\bigr) + \bigl(\nab(e^{n+1} - e^n), \nab\phi_h\bigr) + \nu k\bigl( e^{n+1},  \phi_h\bigr) &=	(\varepsilon^{n+1} - \varepsilon^n, \phi_h) + (\nab (\varepsilon^{n+1} - \varepsilon^n) , \nab \phi_h)  \\\nonumber
			&\qquad+ \nu k (\varepsilon^{n+1}, \phi_h) + \nu k (\theta^{n+1}, \phi_h).
		\end{align*}
		
		Next, replacing this into \eqref{eq3.14} and taking $\phi_h = \varepsilon^{n+1}\in V_h$ and then using the identity $2a(a-b) = a^2-b^2 + (a-b)^2$, we get
		\begin{align}\label{eq3.16}
			&\frac12\bigl[\|\varepsilon^{n+1}\|^2_{H^1} - \|\varepsilon^n\|^2_{H^1} + \|\varepsilon^{n+1} - \varepsilon^n\|^2_{H^1}\bigr]+ \nu k\|\varepsilon^{n+1}\|^2_{L^2} \\\nonumber
			&=-  \nu k (\theta^{n+1}, \varepsilon^{n+1}) + \nu\int_{t_n}^{t_{n+1}} ((u(t_{n+1}) - u(s)), \varepsilon^{n+1})\, ds\\\nonumber
			&\qquad + \int_{t_n}^{t_{n+1}} (F(u(s)) - F(u_h^{n+1}), \nab \varepsilon^{n+1})\, ds  \\\nonumber
			&\qquad + \left(\int_{t_n}^{t_{n+1}} (G(u(s)) - G(u_h^n))\, dW(s), \varepsilon^{n+1}\right)\\\nonumber
			&=-  \nu k (\theta^{n+1}, \varepsilon^{n+1}) + \nu\int_{t_n}^{t_{n+1}} ((u(t_{n+1}) - u(s)), \varepsilon^{n+1})\, ds\\\nonumber
			&\qquad + \int_{t_n}^{t_{n+1}} (F(u(s)) - F(u(t_{n+1})), \nab \varepsilon^{n+1})\, ds 
			\\\nonumber
			&\qquad + \int_{t_n}^{t_{n+1}} (F(u(t_{n+1})) - F(u_h^{n+1}), \nab \varepsilon^{n+1})\, ds  \\\nonumber
			&\qquad + \left(\int_{t_n}^{t_{n+1}} (G(u(s)) - G(u_h^n))\, dW(s), \varepsilon^{n+1}\right)\\\nonumber
			&:= X_1 + X_2 + X_3 + X_4 + X_5.
		\end{align}

		Next steps, we estimate the terms on the right-hand side of \eqref{eq3.16}. To estimate $X_1$, using the Cauchy-Schwarz inequality and \eqref{projection_ineq}, we have
		\begin{align*}
			X_1 &\leq \nu k \| \theta^{n+1}\|^2_{L^2}  + \frac{\nu k}{4}\|\varepsilon^{n+1}\|^2_{L^2}\\\nonumber
			&\leq \nu k Ch^{4}\|u(t_{n+1})\|^2_{H^{2}}  + \frac{\nu k}{4}\|\varepsilon^{n+1}\|^2_{L^2}.
		\end{align*}
		Using the Cauchy-Schwarz inequality, we also obtain
		\begin{align*}
			X_2 &\leq \nu \int_{t_{n}}^{t_{n+1}} \|u(t_{n+1}) - u(s)\|^2_{L^2}\, ds + \frac{\nu k}{4} \|\varepsilon^{n+1}\|^2_{L^2}.
		\end{align*}

		Coming to $X_3$, using the definition of $F$, we split $X_3$ as follows:
		\begin{align}\label{X_3}
			X_3 &=  \int_{t_{n}}^{t_{n+1}} (u(s) - u(t_{n+1}), \p_x \varepsilon^{n+1} + \p_y \varepsilon^{n+1})\, ds \\\nonumber
			&\qquad+  \frac12\int_{t_{n}}^{t_{n+1}} ((u(s) - u(t_{n+1}))(u(s) + u(t_{n+1})), \p_x \varepsilon^{n+1} + \p_y \varepsilon^{n+1})\, ds \\\nonumber
			&\leq k\|\varepsilon^{n+1}\|^2_{H^1} + C\int_{t_{n}}^{t_{n+1}} \|u(s) - u(t_{n+1})\|^2_{L^2}\, ds \\\nonumber
			&\qquad+ C\int_{t_{n}}^{t_{n+1}} \|u(s) - u(t_{n+1})\|^2_{L^4}\|u(s) + u(t_{n+1})\|^2_{L^4}\,ds\\\nonumber
			&\leq k\|\varepsilon^{n+1}\|^2_{H^1} + C\int_{t_{n}}^{t_{n+1}} \|u(s) - u(t_{n+1})\|^2_{L^2}\, ds \\\nonumber
			&\qquad+ CC_e^2\int_{t_{n}}^{t_{n+1}} \|u(s) - u(t_{n+1})\|^2_{H^1}\|u(s) + u(t_{n+1})\|^2_{H^1}\,ds,
		\end{align}
		where the last inequality of \eqref{X_3} is obtained by employing the Sobolev embedding inequality \eqref{Lady_ineq}.
		
		About $X_4$, we also use the formula of $F$ and then the integration by parts as follows:
		\begin{align*}
			X_4 &= k (e^{n+1}, \p_x \varepsilon^{n+1} + \p_y\varepsilon^{n+1}) + \frac{k}{2} (e^{n+1}(u(t_{n+1}) + u_h^{n+1}), \p_x \varepsilon^{n+1} + \p_y \varepsilon^{n+1})\\\nonumber
			&= k (\varepsilon^{n+1}, \p_x \varepsilon^{n+1} + \p_y\varepsilon^{n+1}) + k (\theta^{n+1}, \p_x \varepsilon^{n+1} + \p_y\varepsilon^{n+1}) \\\nonumber
			&\qquad+ \frac{k}{2} (\varepsilon^{n+1}(u(t_{n+1}) + u_h^{n+1}), \p_x \varepsilon^{n+1} + \p_y \varepsilon^{n+1}) + \frac{k}{2} (\theta^{n+1}(u(t_{n+1}) + u_h^{n+1}), \p_x \varepsilon^{n+1} + \p_y \varepsilon^{n+1})\\\nonumber
			&= 0+ k (\theta^{n+1}, \p_x \varepsilon^{n+1} + \p_y\varepsilon^{n+1}) \\\nonumber
			&\qquad+ \frac{k}{2} (\varepsilon^{n+1}(u(t_{n+1}) + u_h^{n+1}), \p_x \varepsilon^{n+1} + \p_y \varepsilon^{n+1}) + \frac{k}{2} (\theta^{n+1}(u(t_{n+1}) + u_h^{n+1}), \p_x \varepsilon^{n+1} + \p_y \varepsilon^{n+1})\\\nonumber
			&:= X_{4,1} + X_{4,2} + X_{4,3}.
		\end{align*}
		Next, we estimate $X_{4,1}, X_{4,2}, X_{4,3}$ as follows: Using the Cauchy-Schwarz inequality and \eqref{projection_ineq} we get
		\begin{align*}
			X_{4,1} &\leq k \|\varepsilon^{n+1}\|^2_{H^1} + Ckh^4 \|u(t_{n+1})\|^2_{H^2}.
		\end{align*}
		
		Similarly, using the inequality \eqref{Lady_ineq} and \eqref{projection_ineq}, we also have
		\begin{align*}
			X_{4,3} &\leq k \|\varepsilon^{n+1}\|^2_{H^1} + CC_e^2kh^2 \|u(t_{n+1})\|^2_{H^2}\|u(t_{n+1}) + u_h^{n+1}\|^2_{H^1}.
		\end{align*}

		To control $X_{4,2}$, using the Cauchy-Schwarz inequality and the Sobolev embedding inequality \eqref{Lady_ineq}, we obtain
		\begin{align*}
			X_{4,2} &\leq \frac{k}{2}\|\p_x\varepsilon^{n+1} + \p_y\varepsilon^{n+1}\|_{L^2}\|\varepsilon^{n+1}\|_{L^4}\|u(t_{n+1}) + u_h^{n+1}\|_{L^4} \\\nonumber
			&\leq k\|\varepsilon^{n+1}\|_{H^1}\|\varepsilon^{n+1}\|_{L^4}(\|u(t_{n+1})\|_{L^4} + \|u_h^{n+1}\|_{L^4})\\\nonumber
			&\leq kC_e\|\varepsilon^{n+1}\|_{H^1}\|\varepsilon^{n+1}\|_{H^1}(\|u(t_{n+1})\|_{H^1} + \|u_h^{n+1}\|_{H^1})\\\nonumber
			&\leq \frac{k C_e^2}{4}\|\varepsilon^{n+1}\|^2_{H^1} + k\|\varepsilon^{n+1}\|^2_{H^1}(\|u(t_{n+1})\|^2_{H^1} + \|u_h^{n+1}\|^2_{H^1}).
		\end{align*}
		
		Finally, we estimate the noise term $X_5$ as follows. 
		\begin{align*}
			X_{5} &=  \left(\int_{t_n}^{t_{n+1}} (G(u(s)) - G(u_h^n))\, dW(s), \varepsilon^{n+1} - \varepsilon^n\right) +  \left(\int_{t_n}^{t_{n+1}} (G(u(s)) - G(u_h^n))\, dW(s),  \varepsilon^n\right)\\\nonumber
			&=  \left(\int_{t_n}^{t_{n+1}} (G(u(s)) - G(u(t_{n})))\, dW(s), \varepsilon^{n+1} - \varepsilon^n\right) +\bigl((G(u(t_n)) - G(u_h^n))\Delta W_n, \varepsilon^{n+1} - \varepsilon^n\bigr) \\\nonumber
			&\qquad+  \left(\int_{t_n}^{t_{n+1}} (G(u(s)) - G(u^n))\, dW(s),  \varepsilon^n\right)\\\nonumber
			&\leq  2\left\|\int_{t_n}^{t_{n+1}} (G(u(s)) - G(u(t_{n})))\, dW(s)\right\|^2_{L^2} +2\left\|(G(u(t_n)) - G(u_h^n))\Delta W_n\right\|^2_{L^2} + \frac14\|\varepsilon^{n+1} - \varepsilon^{n}\|^2_{L^2} \\\nonumber
			&\qquad+  \left(\int_{t_n}^{t_{n+1}} (G(u(s)) - G(u^n))\, dW(s),  \varepsilon^n\right).
		\end{align*}
		
		To fit with the setup of the stochastic Gronwall inequality in Lemma \ref{Stochastic_Gronwall}, we need to add the terms $\pm2k\|G(u(t_n)) - G(u_h^n)\|^2_{L^2}$ on the right-hand side of $X_5$, and then using the condition \eqref{Assump_Lipschitz} as follows:
		\begin{align*}
			X_5 &\leq  2\left\|\int_{t_n}^{t_{n+1}} (G(u(s)) - G(u(t_{n})))\, dW(s)\right\|^2_{L^2}  + \frac14\|\varepsilon^{n+1} - \varepsilon^{n}\|^2_{L^2} \\\nonumber
			&\qquad+2\left\|(G(u(t_n)) - G(u_h^n))\Delta W_n\right\|^2_{L^2} - 2k\|G(u(t_n)) - G(u_h^n)\|^2_{L^2} + 2k\|G(u(t_n)) - G(u_h^n)\|^2_{L^2}\\\nonumber
			&\qquad+  \left(\int_{t_n}^{t_{n+1}} (G(u(s)) - G(u^n))\, dW(s),  \varepsilon^n\right)\\\nonumber
			&\leq  2\left\|\int_{t_n}^{t_{n+1}} (G(u(s)) - G(u(t_{n})))\, dW(s)\right\|^2_{L^2}  + \frac14\|\varepsilon^{n+1} - \varepsilon^{n}\|^2_{L^2} \\\nonumber
			&\qquad+2\left\|(G(u(t_n)) - G(u_h^n))\Delta W_n\right\|^2_{L^2} - 2k\|G(u(t_n)) - G(u_h^n)\|^2_{L^2} + 2kC_G^2\|e^{n}\|^2_{L^2}\\\nonumber
			&\qquad+  \left(\int_{t_n}^{t_{n+1}} (G(u(s)) - G(u^n))\, dW(s),  \varepsilon^n\right)\\\nonumber
			&\leq  2\left\|\int_{t_n}^{t_{n+1}} (G(u(s)) - G(u(t_{n})))\, dW(s)\right\|^2_{L^2}  + \frac14\|\varepsilon^{n+1} - \varepsilon^{n}\|^2_{L^2} \\\nonumber
			&\qquad+2\left\|(G(u(t_n)) - G(u_h^n))\Delta W_n\right\|^2_{L^2} - 2k\|G(u(t_n)) - G(u_h^n)\|^2_{L^2} \\\nonumber
			&\qquad+ 2kC_G^2\|\varepsilon^{n}\|^2_{L^2} + 2kC_G^2C h^{4}\|u(t_n)\|^2_{H^{2}}\\\nonumber
			&\qquad+  \left(\int_{t_n}^{t_{n+1}} (G(u(s)) - G(u^n))\, dW(s),  \varepsilon^n\right).
		\end{align*}

		Now, collecting and substituting all of the estimates on $X_1, ..., X_5$ into \eqref{eq3.16} and absorbing the like terms to the left-hand side of \eqref{eq3.16}, we obtain

		\begin{align}\label{eq_3.14}
			&\frac{1}{2}\left[\|\varepsilon^{n+1}\|^2_{H^1} - \|\varepsilon^n\|^2_{H^1}\right] + \frac14\|\varepsilon^{n+1} - \varepsilon^n\|^2_{H^1} + \frac{\nu k}{2}\|\varepsilon^{n+1}\|^2_{L^2} \\\nonumber
			&\leq Ckh^2\bigl(\nu h^2 + h^2 + C_e^2 \|u(t_{n+1}) + u_h^{n+1}\|^2_{H^1}\bigr)\|u(t_{n+1})\|^2_{H^2} + kC_G^2C h^{4}\|u(t_n)\|^2_{H^{2}}\\\nonumber
			&\qquad +\nu \int_{t_{n}}^{t_{n+1}} \|u(t_{n+1}) - u(s)\|^2_{L^2}\, ds \\\nonumber
			&\qquad+  CC_e^2\int_{t_{n}}^{t_{n+1}} \|u(s) - u(t_{n+1})\|^2_{H^1}\|u(s) + u(t_{n+1})\|^2_{H^1}\,ds \\\nonumber
			&\qquad+ 2\left\|\int_{t_n}^{t_{n+1}} (G(u(s)) - G(u(t_{n})))\, dW(s)\right\|^2_{L^2} \\\nonumber
			&\qquad +2\left\|(G(u(t_n)) - G(u_h^n))\Delta W_n\right\|^2_{L^2} - 2k\|G(u(t_n)) - G(u_h^n)\|^2_{L^2} \\\nonumber
			&\qquad+  \left(\int_{t_n}^{t_{n+1}} (G(u(s)) - G(u^n))\, dW(s),  \varepsilon^n\right)\\\nonumber
			&\qquad + k \left(3 + \frac{C_e^2}{4} + \|u(t_{n+1})\|^2_{H^1} + \|u_h^{n+1}\|^2_{H^1}\right)\|\varepsilon^{n+1}\|^2_{H^1} + 2C_G^2k \|\varepsilon^{n}\|^2_{L^2}.
		\end{align}

		Applying the summation $\sum_{n=0}^{\ell}$ for any $0\leq \ell <M$, we obtain
		\begin{align}\label{eq3.15}
			&	\|\varepsilon^{\ell+1}\|^2_{H^1} + \frac12\sum_{n=0}^{\ell}\|\varepsilon^{n+1} - \varepsilon^n\|^2_{H^1} + {\nu k}\sum_{n=0}^{\ell} \| \varepsilon^{n+1}\|^2_{L^2} \\\nonumber
			&\leq F_{\ell} + M_{\ell} + \sum_{n=0}^{\ell} G_n\|\varepsilon^{n}\|^2_{L^2},
		\end{align}
		
		where 
		\begin{align*}
			F_{\ell} &:= \sum_{n=0}^{\ell}\biggl[ Ckh^2\bigl(\nu h^2 + h^2 + C_e^2 \|u(t_{n+1}) + u_h^{n+1}\|^2_{H^1}\bigr)\|u(t_{n+1})\|^2_{H^2} + kCC_G^2 h^{4}\|u(t_n)\|^2_{H^{2}}\\\nonumber
			&\qquad +2\nu \int_{t_{n}}^{t_{n+1}} \|u(t_{n+1}) - u(s)\|^2_{L^2}\, ds \\\nonumber
			&\qquad+  CC_e^2\int_{t_{n}}^{t_{n+1}} \|u(s) - u(t_{n+1})\|^2_{H^1}\|u(s) + u(t_{n+1})\|^2_{H^1}\,ds \\\nonumber
			&\qquad+ 4\left\|\int_{t_n}^{t_{n+1}} (G(u(s)) - G(u(t_{n})))\, dW(s)\right\|^2_{L^2}\biggr]\\\nonumber
			&\qquad + 2k \left(3 + \frac{C_e^2}{4} + \|u(t_{\ell+1})\|^2_{H^1} + \|u_h^{\ell+1}\|^2_{H^1}\right)\|\varepsilon^{\ell+1}\|^2_{H^1} ,\\\nonumber
			M_{\ell} &:= \sum_{n=0}^{\ell} Z_{n},\\\nonumber
			Z_n&:= 4\left\|(G(u(t_n)) - G(u_h^n))\Delta W_n\right\|^2_{L^2} - 4k\|G(u(t_n)) - G(u_h^n)\|^2_{L^2} \\\nonumber
			&\qquad+ 4 \left(\int_{t_n}^{t_{n+1}} (G(u(s)) - G(u^n))\, dW(s),  \varepsilon^n\right),\\\nonumber
			G_n&:=  2k \left(3 + \frac{C_e^2}{4} + \|u(t_{n+1})\|^2_{H^1} + \|u_h^{n+1}\|^2_{H^1} + 2C_G^2\right).
		\end{align*}
		
		First, to apply the stochastic Gronwall inequality, we assume that 
		$\{M_{\ell};\, \ell \ge 0\}$ is a martingale (this will be verified at the end of the proof). 
		Applying the stochastic Gronwall inequality \eqref{ineq2.4} with 
		$\alpha = \frac{10}{9}$, $\beta = 10$, and $0 < q < 0.9$ 
		to \eqref{eq3.15}, we obtain
		
		\begin{align}\label{eq3.13}
			\Bigl(\mE\bigl[\sup_{0 \leq \ell \leq M}\|\varepsilon^{\ell}\|^{2q}_{H^1}\bigr]\Bigr)^{\frac{1}{2q}} 
			\leq \Bigl(1+ \frac{1}{1- \alpha q}\Bigr)^{\frac{1}{2\alpha q}} \left(\mE\left[\exp\left( \beta q\sum_{n = 0}^{M-1}{G}_n\right)\right]\right)^{\frac{1}{2\beta q}} \,\Bigl(\mE\Bigl[\sup_{0 \leq \ell < M} F_{\ell}\Bigr]\Bigr)^{\frac12}.
		\end{align}
		
		Now, we proceed to estimate the right-hand side of \eqref{eq3.13}. Using  Lemma \ref{lemma_expo_moment_H} and Lemma \ref{lemma_exponential_discrete}, we control the second term as follows:
		\begin{align*}
			\mE\left[\exp\left( 10 q\sum_{n = 0}^{M-1}{G}_n\right)\right] &= \mE\left[\exp\left( 10 q\sum_{n = 0}^{M-1}2k\left(3 + \frac{C_e^2}{4} + \|u(t_{n+1})\|^2_{H^1} + \|u_h^{n+1}\|^2_{H^1} + 2C_G^2\right)\right)\right]\\\nonumber
			&= \exp\left(20 q\left(3+\frac{C_e^2}{4} + 2C_G^2\right)T\right) \mE\left[\exp\left( 20q k\sum_{n = 0}^{M}\left(\|u(t_{n})\|^2_{H^1} + \|u_h^{n}\|^2_{H^1}\right) \right)\right].
		\end{align*}
		Then, using the discrete Jensen inequality,we have
		\begin{align*}
			\mE\left[\exp\left( 10q\sum_{n = 0}^{M-1}{G}_n\right)\right] 
			&\leq \exp\left(20 q\left(3+\frac{C_e^2}{4} + 2C_G^2\right)T\right)  k\sum_{n=0}^{M}\mE\left[\exp\left(20 q\left(\|u(t_{n})\|^2_{H^1} + \|u_h^{n}\|^2_{H^1}\right)\right)\right]\\\nonumber
			&\leq \exp\left(20 q\left(3+\frac{C_e^2}{4} + 2C_G^2\right)T\right)  (C_{exp,1} + C_{exp,2})T:= \tilde{C}_{0},
		\end{align*}
		provided that $20q \leq \frac{\nu}{4L_0^2}$. This condition is fullfilled by using the hypothesis $L_0 \leq \frac{\sqrt{\nu}}{4 \sqrt{5 q}}$.
		
		Next, we estimate $\Bigl(\mE\Bigl[\sup_{0 \leq \ell < M} F_{\ell}\Bigr]\Bigr)^{\frac12}$ as follows: Using the definition of $F_{\ell}$, we have
		\begin{align*}
			\mE\Bigl[\sup_{0 \leq \ell < M} F_{\ell}\Bigr] &\leq  \sum_{n=0}^{M-1}\mE\biggl[ Ckh^2\bigl(\nu h^2 + h^2 + C_e^2 \|u(t_{n+1}) + u_h^{n+1}\|^2_{H^1}\bigr)\|u(t_{n+1})\|^2_{H^2} \\\nonumber
			&\qquad + kCC_G^2 h^{4}\|u(t_n)\|^2_{H^{2}}\biggr]+   \sum_{n=0}^{M-1}\mE\biggl[2\nu \int_{t_{n}}^{t_{n+1}} \|u(t_{n+1}) - u(s)\|^2_{L^2}\, ds \\\nonumber
			&\qquad+  CC_e^2\int_{t_{n}}^{t_{n+1}} \|u(s) - u(t_{n+1})\|^2_{H^1}\|u(s) + u(t_{n+1})\|^2_{H^1}\,ds \\\nonumber
			&\qquad+ 4\left\|\int_{t_n}^{t_{n+1}} (G(u(s)) - G(u(t_{n})))\, dW(s)\right\|^2_{L^2}\biggr]\\\nonumber
			&\qquad + 2k \mE\left[\sup_{0 \leq \ell < M}\left(3 + \frac{C_e^2}{4} + \|u(t_{\ell+1})\|^2_{H^1} + \|u_h^{\ell+1}\|^2_{H^1}\right)\|\varepsilon^{\ell+1}\|^2_{H^1}\right]\\\nonumber
			&:= Q_1 + Q_2 + Q_3.
		\end{align*}

		Using Lemma \ref{lem:moment-bound:H1} and Lemma \ref{lemma_highmoment_discrete}, and Lemma \ref{lem:moment-bound:H2} we obtain
		\begin{align*}
			Q_1 &\leq Ch^2\left(\nu C_{2,1} + C_{2,1}  + C_e^2(C_{1,2} + C_{4,2})C_{2,2}\right)T:= C_{Q_1}h^2.
		\end{align*}
		
		Next, using Lemma \ref{lem:Lipschitz:H^1:u}, Lemma \ref{lem:moment-bound:H1}, the It\^o isometry and the condition \eqref{Assump_Lipschitz} with $m=0$, we get
		\begin{align*}
			Q_2 &=2\nu\sum_{n = 0}^{M-1}\int_{t_n}^{t_{n+1}}\mE\left[\|u(t_{n+1}) - u(s)\|^2_{L^2}\right]\, ds\\\nonumber
			&\quad+C C_e^2\sum_{n = 0}^{M-1}\int_{t_n}^{t_{n+1}}\mE\left[\left(\|u(s)\|^2_{H^1} + \| u(t_{n+1})\|^2_{H^1}\right)\|u(t_{n+1}) - u(s)\|^2_{H^1}\right]\, ds\\\nonumber
			&\qquad+ 4\sum_{n = 0}^{M-1} \int_{t_{n}}^{t_{n+1}} \mE\left[\|G(u(s)) - G(u(t_n))\|^2_{L^2}\right]\, ds\\\nonumber
			&\leq 2\nu T C_{3,1}k + CC_e^2 T C_{1,2}C_{3,2}k + 4C_G^2C_{3,1} k := C_{Q_2}k.
		\end{align*}

		To control $Q_3$, we recall Lemma \ref{lem:moment-bound:H1} and Lemma \ref{lemma_highmoment_discrete} and use the Cauchy-Schwarz inequality as follows. 
		\begin{align*}
			Q_3 &\leq 2\left(3 + \frac{C_e^2}{4}\right)(C_{1,1} + C_{4,1})k + 2k \left\{\mE\left[\sup_{0 \leq \ell \leq M}\|u(t_{\ell})\|^4_{H^1}\right]\right\}^{\frac12} \left\{\mE\left[\sup_{0 \leq \ell \leq M}\|\varepsilon^{\ell+1}\|^4_{H^1}\right]\right\}^{\frac12}\\\nonumber
			&\qquad +  2k \left\{\mE\left[\sup_{0 \leq \ell \leq M}\|u_h^{\ell}\|^4_{H^1}\right]\right\}^{\frac12} \left\{\mE\left[\sup_{0 \leq \ell \leq M}\|\varepsilon^{\ell+1}\|^4_{H^1}\right]\right\}^{\frac12}\\\nonumber
			&\leq 2\left(3 + \frac{C_e^2}{4}\right)(C_{1,1} + C_{4,1})k + 2(C_{1,2} + C_{4,2})(C_{1,2} + C_{4,2}) k
			:=C_{Q_3}k.
		\end{align*}
		
		Substituting all the estimates from $Q_1, Q_2, Q_3$ into \eqref{eq3.13} we arrive at
		\begin{align*}
			&\Bigl(\mE\bigl[\sup_{0 \leq \ell \leq M}\|\varepsilon^{\ell}\|^{2q}_{H^1}\bigr]\Bigr)^{\frac{1}{2q}} \leq \widehat{C}_1\left(k + h^{2}\right),
		\end{align*}
		where 
		\begin{align*}
			\widehat{C}_1 =\left(\frac{2(1-q)}{1-2q}\right)^{\frac{1}{4q}}\tilde{C}_0\max\left\{C_{Q_1}, C_{Q_2}, C_{Q_3}\right\}.
		\end{align*}
		
		We now establish that the sequence ${M_\ell}$ defines a martingale. To do so, we invoke It\^o's isometry and assumption~\eqref{Assump_Lipschitz}, and then apply the Burkholder–Davis–Gundy inequality to derive
		\begin{align}\label{eq_3.16}
			\mE[|M_{\ell}|] 
			&\leq 8k\sum_{n=0}^{\ell} \mE\bigl[\|G(u(t_n)) - G(u_h^n)\|^2_{L^2}\bigr] \\\nonumber
			&\qquad+ 4\mE\left[\left|\sum_{n = 0}^{\ell} \left(\int_{t_{n}}^{t_{n+1}}\bigl(G(u(s)) - G(u_h^n)\bigr)\, dW(s), \varepsilon^n\right)\right|\right]\\\nonumber
			&\leq 8C_G^2T(C_{1,1} + C_{4,1}) + \left(\mE\left[\sum_{n = 0}^{M-1}\int_{t_{n}}^{t_{n+1}} \|G(u(s)) - G(u_h^n)\|^2_{L^2}\|\varepsilon^n\|^2_{L^2}\, ds\right]\right)^{\frac12} \\\nonumber
			&\leq 8C_G^2T(C_{1,1} + C_{4,1}) + C_G^2T(C_{1,2} + C_{4,2})(C_{1,2} + C_{4,2}) <\infty.
		\end{align}
		
		In addition, for any $0\leq n \leq M-1$, using the martingale property of the It\^o integrals, we have
		\begin{align*}
			\mE[Z_n] &=  4\mE\bigl[\|G(u(t_n)) - G(u_h^n)\|^2_{L^2}|\Delta W_{n}|^2 \bigr]  { -  4k\mE\bigl[ \|G(u(t_n)) - G(u_h^n)\|^2_{L^2} \bigr] } \\\nonumber
			&\qquad \quad  + 4\mE\biggl[\biggl(\int_{t_n}^{t_{n+1}}\bigl(G(u(s))- G(u_h^n)\bigr)\,dW(s),\varepsilon^n\biggr)\biggr]\\\nonumber
			&= 4k \mE \bigl[\|G(u(t_n)) - G(u_h^n)\|^2_{L^2} \bigr] 
			{ - 4k\mE \bigl[ \|G(u(t_n)) - G(u_h^n)\|^2_{L^2} \bigr]  } + 0  =0.
		\end{align*}
		Then,  the conditional expectation of $M_{\ell}$ given $\{M_n\}_{n=0}^{\ell-1}$ is
		\begin{align}\label{eq3.17}
			\mE\bigl[M_{\ell} | M_0, M_1, \cdots, M_n \bigr]  
			&= \mE\bigl[Z_0+Z_1 + \cdots + Z_n| M_0, M_1, \cdots, M_n \bigr] \\\nonumber
			&\qquad+ \mE \bigl[Z_{n+1} + \cdots+Z_{\ell}| M_0, M_1,\cdots, M_{n} \bigr]\\\nonumber
			&= M_{n} + \mE\bigl[Z_{n+1} + \cdots + Z_{\ell} \bigr] = M_n.
		\end{align}
		
		Thus,  we conclude that $\bigl\{M_{\ell}; \ell \geq 0\bigr\}$ is a martingale using \eqref{eq_3.16} and \eqref{eq3.17}.
		
		The proof is finished by combining the triangle inequality, \eqref{equu310}, and \eqref{projection_ineq}.

	\end{proof}
	
	\subsubsection{Higher moment error estimates}
	The sub-second moment error estimates obtained in Theorem \ref{Theorem_sub_moment} imply a strong convergence in the $L_{\omega}^{p}L_t^{\infty}L_x^2$- and $L_{\omega}^{p}L_t^{2}H_x^2$-norm for $0<p<2$.  We note that these sub-second moment estimates are consequences of using the stochastic Gronwall inequality \eqref{ineq2.4}. However, we will demonstrate below that a bootstrap argument can overcome such a limitation to obtain higher moment estimates. {In turn, they allow for establishing a strong convergence in the $L_{\omega}^{p}L_t^{\infty}L_x^2$-norm for $0<p <4$, which is the goal of this subsection}.
	
	\begin{theorem}\label{Theorem_higher_moment} 
		Let $u$ be the variational solution to \eqref{Weak_formulation} and $\{u_h^{n}\}_{n=1}^M$ be generated by \eqref{Scheme_Standard}. Suppose that $G$ satisfies conditions \eqref{Assump_Lineargrowth} and \eqref{Assump_Lipschitz} and that $u_0 \in  L^{48}(\Omega; H^1(D))\cap L^{16}(\Omega; H^2(D))$. 
		Additionally, for any $0 < q < 0.9$, assume that $L_0 <\frac{\sqrt{\nu}}{4\sqrt{ 10q}}$ and $\mE\left[\exp\left(4\gamma\|u_h^0\|^2_{H^1} + 16\gamma^2 L_0^2k \|u^0_h\|^2_{L^2}\right)  \right] <\infty$, where $\gamma = 40 q$. Moreover, suppose that 
		there exists a constant $C>0$ such that
		\begin{align}
			\left(\mE\left[\|u_0 - u_h^0\|^{4q}_{H^1}\right]\right)^{\frac{1}{4q}} 
			\leq C h,
		\end{align} 
		
		Then, there holds
		\begin{align}\label{eq3.20}
			\left(\mE\left[\max_{1\leq n\leq M}\|u(t_n) - u_h^n\|^{4q}_{H^1}\right]\right)^{\frac{1}{4q}} 
			\leq \widehat{C}_2\left(k^{\frac12} + h\right),
		\end{align}
		where the constant $\widehat{C}_2 = C(q,u_0,T,C_G)>0$.
	\end{theorem}
	\begin{proof} Without loss of generality, we assume that $u_h^0 = u_0$ in the proof. First of all, testing \eqref{eq_3.14} with $\|\varepsilon^{n+1}\|^2_{H^1}$ and using the identity $2a(a-b) = a^2 - b^2 + (a-b)^2$ we obtain
		\begin{align}\label{eq3.21}
			&\frac{1}{4}\left[\|\varepsilon^{n+1}\|^4_{H^1} - \|\varepsilon^n\|^4_{H^1}\right] + \frac14\left(\|\varepsilon^{n+1}\|^2_{H^1} - \|\varepsilon^n\|^2_{H^1}\right)^2 \\\nonumber
			&\qquad+  \frac14\|\varepsilon^{n+1} - \varepsilon^n\|^2_{H^1}\|\varepsilon^{n+1}\|^2_{H^1} + \frac{\nu k}{2}\|\varepsilon^{n+1}\|^2_{L^2}\|\varepsilon^{n+1}\|^2_{H^1} \\\nonumber
			&\leq Ckh^2\left[\bigl(\nu h^2 + h^2 + C_e^2 \|u(t_{n+1}) + u_h^{n+1}\|^2_{H^1}\bigr)\|u(t_{n+1})\|^2_{H^2} + C_G^2 h^{2}\|u(t_n)\|^2_{H^{2}}\right]\|\varepsilon^{n+1}\|^2_{H^1}\\\nonumber
			&\qquad +\nu \int_{t_{n}}^{t_{n+1}} \|u(t_{n+1}) - u(s)\|^2_{L^2}\|\varepsilon^{n+1}\|^2_{H^1}\, ds \\\nonumber
			&\qquad+  CC_e^2\int_{t_{n}}^{t_{n+1}} \|u(s) - u(t_{n+1})\|^2_{H^1}\|u(s) + u(t_{n+1})\|^2_{H^1}\|\varepsilon^{n+1}\|^2_{H^1}\,ds \\\nonumber
			&\qquad+ 2\left\|\int_{t_n}^{t_{n+1}} (G(u(s)) - G(u(t_{n})))\, dW(s)\right\|^2_{L^2}\|\varepsilon^{n+1}\|^2_{H^1} \\\nonumber
			&\qquad +2\left\|(G(u(t_n)) - G(u_h^n))\Delta W_n\right\|^2_{L^2}\|\varepsilon^{n+1}\|^2_{H^1} - 2k\|G(u(t_n)) - G(u_h^n)\|^2_{L^2}\|\varepsilon^{n+1}\|^2_{H^1} \\\nonumber
			&\qquad+  \left(\int_{t_n}^{t_{n+1}} (G(u(s)) - G(u^n))\, dW(s),  \varepsilon^n\right)\|\varepsilon^{n+1}\|^2_{H^1} \\\nonumber
			&\qquad + k \left(3 + \frac{C_e^2}{4} + \|u(t_{n+1})\|^2_{H^1} + \|u_h^{n+1}\|^2_{H^1}\right)\|\varepsilon^{n+1}\|^4_{H^1} + 2C_G^2k \|\varepsilon^{n}\|^2_{L^2}\|\varepsilon^{n+1}\|^2_{H^1}\\\nonumber
			&:= I_1 + I_2 + ... + I_9.
		\end{align}

		Now, we estimate the right-hand side of \eqref{eq3.21}. Using the discrete Young inequality, we have
		\begin{align*}
			I_1 + I_2 + I_3 + I_6 +I_8 + I_9
			&\leq k\|\varepsilon^{n+1}\|^4_{H^1} + Ckh^{4}\left[\|u(t_{n+1}) + u_h^{n+1}\|^4_{H^{1}}\|u(t_{n+1})\|^4_{H^2} + C_G^4 \|u(t_n)\|^4_{H^{2}}\right]\\\nonumber
			&\qquad + C\nu^2 \int_{t_{n}}^{t_{n+1}} \|u(t_{n+1}) - u(s)\|^4_{L^2}\, ds \\\nonumber
			&\qquad + CC_e^4\int_{t_n}^{t_{n+1}}\|u(t_{n+1}) - u(s)\|^4_{H^1}\|u(s) + u(t_{n+1})\|^4_{H^1}\, ds\\\nonumber
			&\qquad + 8C_G^4k\|\varepsilon^n\|^4_{L^2} + k \left(3 + \frac{C_e^2}{4} + \|u(t_{n+1})\|^2_{H^1} + \|u_h^{n+1}\|^2_{H^1}\right)\|\varepsilon^{n+1}\|^4_{H^1}. 
		\end{align*}
		
		Next, to control $I_4, I_5$, and $I_7$, we add and substract the terms $\|\varepsilon^{n}\|^2_{H^1}$ as follows: 
		\begin{align*}
			I_4 + I_5 + I_7
			&= 2\left\|\int_{t_n}^{t_{n+1}} (G(u(s)) - G(u(t_{n})))\, dW(s)\right\|^2_{L^2}\left(\|\varepsilon^{n+1}\|^2_{H^1} - \|\varepsilon^{n}\|^2_{H^1}\right)\\\nonumber
			&\qquad +2\left\|(G(u(t_n)) - G(u_h^n))\Delta W_n\right\|^2_{L^2}\left(\|\varepsilon^{n+1}\|^2_{H^1} - \|\varepsilon^{n}\|^2_{H^1}\right)\\\nonumber
			&\qquad+  \left(\int_{t_n}^{t_{n+1}} (G(u(s)) - G(u_h^n))\, dW(s),  \varepsilon^n\right)\left(\|\varepsilon^{n+1}\|^2_{H^1} - \|\varepsilon^{n}\|^2_{H^1}\right)\\\nonumber
			&\qquad + 2\left\|\int_{t_n}^{t_{n+1}} (G(u(s)) - G(u(t_{n})))\, dW(s)\right\|^2_{L^2}\|\varepsilon^{n}\|^2_{H^1}\\\nonumber
			&\qquad +2\left\|(G(u(t_n)) - G(u_h^n))\Delta W_n\right\|^2_{L^2}\|\varepsilon^{n}\|^2_{H^1}\\\nonumber
			&\qquad+  \left(\int_{t_n}^{t_{n+1}} (G(u(s)) - G(u_h^n))\, dW(s),  \varepsilon^n\right)\|\varepsilon^{n}\|^2_{H^1}.
		\end{align*}
		
		Using the discrete Young inequality on the first three terms on the right-hand side of $I_4 + I_5 + I_7$, we obtain
		\begin{align}\label{eq4.40}
			I_4 + I_5 + I_7
			&\leq 16\left\|\int_{t_n}^{t_{n+1}} (G(u(s)) - G(u(t_{n})))\, dW(s)\right\|^4_{L^2} +16\left\|(G(u(t_n)) - G(u_h^n))\Delta W_n\right\|^4_{L^2}\\\nonumber
			&\qquad+  4\left\|\int_{t_n}^{t_{n+1}} (G(u(s)) - G(u_h^n))\, dW(s)\right\|^2_{L^2} \|\varepsilon^n\|^2_{H^1}\\\nonumber
			&\qquad +2\left\|(G(u(t_n)) - G(u_h^n))\Delta W_n\right\|^2_{L^2}\|\varepsilon^{n}\|^2_{H^1}\\\nonumber
			&\qquad+  \left(\int_{t_n}^{t_{n+1}} (G(u(s)) - G(u_h^n))\, dW(s),  \varepsilon^n\right)\|\varepsilon^{n}\|^2_{H^1}  + \frac{3}{16}\left(\|\varepsilon^{n+1}\|^2_{H^1} - \|\varepsilon^n\|^2_{H^1}\right)^{2}.
		\end{align}
		
		Next, to use the stochastic Gronwall inequality at the end, we add and substract the terms $\pm 48k^2\|G(u(t_n)) - G(u_h^n)\|^4_{L^2}\pm4\int_{t_{n}}^{t_{n+1}}\|G(u(s)) - G(u^n_h)\|^2_{L^2}\|\varepsilon^n\|^2_{H^1}\, ds \pm 2k\|G(u(t_n)) - G(u_h^n)\|^2_{L^2}\|\varepsilon^n\|^2_{H^1}$ on the right-hand side of \eqref{eq4.40} to obtain
		\begin{align*}
			I_4 + I_5 + I_7
			&\leq 16\left\|\int_{t_n}^{t_{n+1}} (G(u(s)) - G(u(t_{n})))\, dW(s)\right\|^4_{L^2} +16\left\|(G(u(t_n)) - G(u_h^n))\Delta W_n\right\|^4_{L^2}\\\nonumber
			&\qquad - 48k^2\left\|G(u(t_n)) - G(u_h^n)\right\|^4_{L^2}\\\nonumber
			&\qquad+  4\left\|\int_{t_n}^{t_{n+1}} (G(u(s)) - G(u_h^n))\, dW(s)\right\|^2_{L^2} \|\varepsilon^n\|^2_{H^1} \\\nonumber
			&\qquad\qquad- 4\int_{t_{n}}^{t_{n+1}}\|G(u(s)) - G(u^n_h)\|^2_{L^2}\|\varepsilon^n\|^2_{H^1}\, ds\\\nonumber
			&\qquad +2\left\|(G(u(t_n)) - G(u_h^n))\Delta W_n\right\|^2_{L^2}\|\varepsilon^{n}\|^2_{H^1} - 2k\|G(u(t_n)) - G(u_h^n)\|^2_{L^2}\|\varepsilon^n\|^2_{H^1}\\\nonumber
			&\qquad+  \left(\int_{t_n}^{t_{n+1}} (G(u(s)) - G(u_h^n))\, dW(s),  \varepsilon^n\right)\|\varepsilon^{n}\|^2_{H^1}\\\nonumber
			&\qquad + 4\int_{t_{n}}^{t_{n+1}}\|G(u(s)) - G(u^n_h)\|^2_{L^2}\|\varepsilon^n\|^2_{H^1}\, ds + 2k\|G(u(t_n)) - G(u_h^n)\|^2_{L^2}\|\varepsilon^n\|^2_{H^1}\\\nonumber
			&\qquad +48k^2\left\|G(u(t_n)) - G(u_h^n)\right\|^4_{L^2}+ \frac{3}{16}\left(\|\varepsilon^{n+1}\|^2_{H^1} - \|\varepsilon^n\|^2_{H^1}\right)^{2}\\\nonumber
			&\leq 16\left\|\int_{t_n}^{t_{n+1}} (G(u(s)) - G(u(t_{n})))\, dW(s)\right\|^4_{L^2} \\\nonumber
			&\qquad+ \left[16\left\|(G(u(t_n)) - G(u_h^n))\Delta W_n\right\|^4_{L^2} - 48k^2\|G(u(t_n)) - G(u^n_h)\|^4_{L^2}\right]\\\nonumber
			&\qquad+  \left[4\left\|\int_{t_n}^{t_{n+1}} (G(u(s)) - G(u_h^n))\, dW(s)\right\|^2_{L^2} \|\varepsilon^n\|^2_{H^1} \right.\\\nonumber
			&\qquad\qquad\left.- 4\int_{t_{n}}^{t_{n+1}}\|G(u(s)) - G(u^n_h)\|^2_{L^2}\|\varepsilon^n\|^2_{H^1}\, ds\right]\\\nonumber
			&\qquad +\left[2\left\|(G(u(t_n)) - G(u_h^n))\Delta W_n\right\|^2_{L^2}\|\varepsilon^{n}\|^2_{H^1} - 2k\|G(u(t_n)) - G(u_h^n)\|^2_{L^2}\|\varepsilon^n\|^2_{H^1}\right]\\\nonumber
			&\qquad+  \left(\int_{t_n}^{t_{n+1}} (G(u(s)) - G(u_h^n))\, dW(s),  \varepsilon^n\right)\|\varepsilon^{n}\|^2_{H^1}\\\nonumber
			&\qquad + 16 C_G^2\int_{t_{n}}^{t_{n+1}}\|u(s) - u(t_n)\|^4_{L^2}\, ds + 20C_G^2k\|\varepsilon^n\|^4_{L^2} + 384 CC_G^4 k^2 h^{4}\|u(t_{n})\|^4_{H^{2}} \\\nonumber
			&\qquad+ 384 C_G^4k^2 \|\varepsilon^{n}\|^4_{L^2} + \frac{3}{16}\left(\|\varepsilon^{n+1}\|^2_{L^2} - \|\varepsilon^n\|^2_{L^2}\right)^{2}.
		\end{align*}

		Next step, collecting all the estimates from $I_1, ..., I_{9}$ and then replacing them to \eqref{eq3.21} and absorbing the like terms to the left-hand side of \eqref{eq3.21}, and then applying the summation $\sum_{n = 0}^{\ell}$ for any $0\leq \ell \leq M-1$, we obtain the following inequality:
		\begin{align}\label{eq3.22}
			&	\frac{1}{4}\|\varepsilon^{\ell+1}\|^4_{H^1} + \frac{1}{16}\sum_{n = 0}^{\ell}\left(\|\varepsilon^{n+1}\|^2_{H^1} - \|\varepsilon^n\|^2_{H^1}\right)^2 \\\nonumber
			&\leq  C kh^{4}\sum_{n = 0}^{\ell}\left[\|u(t_{n+1}) + u_h^{n+1}\|^4_{H^{1}}\|u(t_{n+1})\|^4_{H^2} + C_G^4 \|u(t_n)\|^4_{H^{2}}\right]\\\nonumber
			&\qquad + C\nu^2\sum_{n = 0}^{\ell}\int_{t_n}^{t_{n+1}}\|u(t_{n+1}) - u(s)\|^4_{L^2}\, ds \\\nonumber
			&\qquad + CC_e^4\sum_{n = 0}^{\ell}\int_{t_n}^{t_{n+1}}\left(\|u(s)\|^4_{H^1} + \| u(t_{n+1})\|^4_{H^1}\right)\|u(t_{n+1}) - u(s)\|^4_{H^1}\, ds\\\nonumber
			&\qquad+ 16\sum_{n = 0}^{\ell}\left\|\int_{t_n}^{t_{n+1}} (G(u(s)) - G(u(t_{n})))\, dW(s)\right\|^4_{L^2} \\\nonumber
			&\qquad+\sum_{n = 0}^{\ell} \left[16\left\|(G(u(t_n)) - G(u_h^n))\Delta W_n\right\|^4_{L^2} - 48k^2\|G(u(t_n)) - G(u^n_h)\|^4_{L^2}\right]\\\nonumber
			&\qquad+  \sum_{n = 0}^{\ell}\left[4\left\|\int_{t_n}^{t_{n+1}} (G(u(s)) - G(u_h^n))\, dW(s)\right\|^2_{L^2} \|\varepsilon^n\|^2_{H^1} \right.\\\nonumber
			&\qquad\qquad\left.- 4\int_{t_{n}}^{t_{n+1}}\|G(u(s)) - G(u^n_h)\|^2_{L^2}\|\varepsilon^n\|^2_{H^1}\, ds\right]\\\nonumber
			&\qquad +\sum_{n = 0}^{\ell}\left[2\left\|(G(u(t_n)) - G(u_h^n))\Delta W_n\right\|^2_{L^2}\|\varepsilon^{n}\|^2_{H^1} - 2k\|G(u(t_n)) - G(u_h^n)\|^2_{L^2}\|\varepsilon^n\|^2_{H^1}\right]\\\nonumber
			&\qquad+ \sum_{n = 0}^{\ell} \left(\int_{t_n}^{t_{n+1}} (G(u(s)) - G(u_h^n))\, dW(s),  \varepsilon^n\right)\|\varepsilon^{n}\|^2_{H^1}  \\\nonumber
			&\qquad + 16 C_G^2\sum_{n = 0}^{\ell}\int_{t_{n}}^{t_{n+1}}\|u(s) - u(t_n)\|^4_{L^2}\, ds + 384 CC_G^4 k^2 h^{4}\sum_{n = 0}^{\ell}\|u(t_{n})\|^4_{H^{2}}\\\nonumber
			&\qquad + C_G^2\left(20 + 8C_G^2+ 384C_G^2k\right)k\sum_{n = 0}^{\ell}\|\varepsilon^n\|^4_{L^2} \\\nonumber
			&\qquad+ k\sum_{n = 0}^{\ell}\left( 4 + \frac{C_e^2}{4} + \|u(t_{n+1})\|^2_{H^1} + \|u_h^{n+1}\|^2_{H^1} \right)\|\varepsilon^{n+1}\|^4_{H^1}.
		\end{align}
		
		With denoting $\mathcal{U}_{n+1}:= 4 + \frac{C_e^2}{4} + \|u(t_{n+1})\|^2_{H^1} + \|u_h^{n+1}\|^2_{H^1}$, the last term on the right-hand side of \eqref{eq3.22} can be further analyzed as follows: first, adding the terms $\pm k\mathcal{U}_{\ell+1}\|\varepsilon^{\ell}\|^4_{H^1}$, and then using the discrete Young inequality, we obtain
		\begin{align}\label{eq_3.23}
			k\sum_{n = 0}^{\ell}\mathcal{U}_{n+1}\|\varepsilon^{n+1}\|^4_{H^1}	&=k\sum_{n = 0}^{\ell-1}\mathcal{U}_{n+1}\|\varepsilon^{n+1}\|^4_{H^1} + k\mathcal{U}_{\ell+1}\|\varepsilon^{\ell+1}\|^4_{H^1}\\\nonumber
			&= k\sum_{n = 0}^{\ell-1}\mathcal{U}_{n+1}\|\varepsilon^{n+1}\|^4_{H^1} + k\mathcal{U}_{\ell+1}\|\varepsilon^{\ell}\|^4_{H^1} \\\nonumber
			&\qquad+ k\mathcal{U}_{\ell+1}\left(\|\varepsilon^{\ell+1}\|^2_{H^1} + \|\varepsilon^{\ell}\|^2_{H^1}\right)\left(\|\varepsilon^{\ell+1}\|^2_{H^1} - \|\varepsilon^{\ell}\|^2_{H^1}\right)\\\nonumber
			&\leq k\sum_{n = 0}^{\ell}\left(\mathcal{U}_{n} + \mathcal{U}_{n+1}\right)\|\varepsilon^{n}\|^4_{H^1} \\\nonumber
			&\qquad+ 8k^2\mathcal{U}^2_{\ell+1}\left(\|\varepsilon^{\ell+1}\|^2_{H^1} + \|\varepsilon^{\ell}\|^2_{H^1}\right)^2 + \frac{1}{32}\left(\|\varepsilon^{\ell+1}\|^2_{H^1} - \|\varepsilon^{\ell}\|^2_{H^1}\right)^2\\\nonumber
			&\leq k\sum_{n = 0}^{\ell}\left(8 + \frac{C_e^2}{2} + \|u(t_{n+1})\|^2_{H^1} + \|u(t_{n})\|^2_{H^1} + \|u_h^{n+1}\|^2_{H^1} + \|u_h^n\|^2_{H^1}\right)\|\varepsilon^{n}\|^4_{H^1} \\\nonumber
			&\qquad+ 8k^2\mathcal{U}^2_{\ell+1}\left(\|\varepsilon^{\ell+1}\|^2_{H^1} + \|\varepsilon^{\ell}\|^2_{H^1}\right)^2 + \frac{1}{32}\left(\|\varepsilon^{\ell+1}\|^2_{H^1} - \|\varepsilon^{\ell}\|^2_{H^1}\right)^2.
		\end{align}

		We note that the last term on the right-hand side of \eqref{eq_3.23} will be absorbed into the second term on the left-hand side of \eqref{eq3.22}.
		
		Next, substituting \eqref{eq_3.23} into \eqref{eq3.22}, we derive the following form, which is ready for employing the stochastic Gronwall inequality \eqref{Stochastic_Gronwall}:
		\begin{align}\label{eq3.23}
			&	\|\varepsilon^{\ell+1}\|^4_{H^1} + \frac{1}{8}\sum_{n = 0}^{\ell-1} \left(\|\varepsilon^{n+1}\|^2_{H^1} - \|\varepsilon^n\|^2_{H^1}\right)^2 \leq \mathcal{S}_{\ell} + \mathcal{J}_{\ell} + \sum_{n = 0}^{\ell} \mathcal{X}_n \|\varepsilon^n\|^4_{H^1},
		\end{align}
		where,
		\begin{align*}
			\mathcal{S}_{\ell}&:= C kh^{4}\sum_{n = 0}^{\ell}\left[\|u(t_{n+1}) + u_h^{n+1}\|^4_{H^{1}}\|u(t_{n+1})\|^4_{H^2} + C_G^4 \|u(t_n)\|^4_{H^{2}}\right]\\\nonumber
			&\qquad + C(\nu^2+C_G^2)\sum_{n = 0}^{\ell}\int_{t_n}^{t_{n+1}}\|u(t_{n+1}) - u(s)\|^4_{L^2}\, ds \\\nonumber
			&\qquad + CC_e^4\sum_{n = 0}^{\ell}\int_{t_n}^{t_{n+1}}\left(\|u(s)\|^4_{H^1} + \| u(t_{n+1})\|^4_{H^1}\right)\|u(t_{n+1}) - u(s)\|^4_{H^1}\, ds\\\nonumber
			&\qquad+ 64\sum_{n = 0}^{\ell}\left\|\int_{t_n}^{t_{n+1}} (G(u(s)) - G(u(t_{n})))\, dW(s)\right\|^4_{L^2} + 32k^2\mathcal{U}^2_{\ell+1}\left(\|\varepsilon^{\ell+1}\|^2_{H^1} + \|\varepsilon^{\ell}\|^2_{H^1}\right)^2,\\\nonumber
			\mathcal{J}_{\ell} &:= \sum_{n = 0}^{\ell} J_n,\\\nonumber
			J_{n}&:= \sum_{n = 0}^{\ell} \left[16\left\|(G(u(t_n)) - G(u_h^n))\Delta W_n\right\|^4_{L^2} - 48k^2\|G(u(t_n)) - G(u^n_h)\|^4_{L^2}\right]\\\nonumber
			&\qquad+  \sum_{n = 0}^{\ell}\left[4\left\|\int_{t_n}^{t_{n+1}} (G(u(s)) - G(u_h^n))\, dW(s)\right\|^2_{L^2} \|\varepsilon^n\|^2_{H^1} \right.\\\nonumber
			&\qquad\qquad\left.- 4\int_{t_{n}}^{t_{n+1}}\|G(u(s)) - G(u^n_h)\|^2_{L^2}\|\varepsilon^n\|^2_{H^1}\, ds\right]\\\nonumber
			&\qquad +\sum_{n = 0}^{\ell}\left[2\left\|(G(u(t_n)) - G(u_h^n))\Delta W_n\right\|^2_{L^2}\|\varepsilon^{n}\|^2_{H^1} - 2k\|G(u(t_n)) - G(u_h^n)\|^2_{L^2}\|\varepsilon^n\|^2_{H^1}\right]\\\nonumber
			&\qquad+ \sum_{n = 0}^{\ell} \left(\int_{t_n}^{t_{n+1}} (G(u(s)) - G(u_h^n))\, dW(s),  \varepsilon^n\right)\|\varepsilon^{n}\|^2_{H^1}  ,\\\nonumber
			\mathcal{X}_n &= C_G^2(20 + 8C_G^2 + 384 C_G^2k) +8 + \frac{C_e^2}{2} + \|u(t_{n+1})\|^2_{H^1} + \|u(t_{n})\|^2_{H^1} + \|u_h^{n+1}\|^2_{H^1} + \|u_h^n\|^2_{H^1}.
		\end{align*}
		
		Next, we suppose for the moment that $\{\mathcal{J}_{\ell}: \ell \geq 0\}$ is a martingale. So, the stochastic Gronwall inequality \eqref{Stochastic_Gronwall} can be employed to \eqref{eq3.23} with the following parameters: $ \alpha = \frac{10}{9}, \beta = 10$, $0<q<0.9$, which gives us the following inequality
		\begin{align}\label{eq3.24}
			\left(\mE\left[\sup_{0 \leq \ell \leq M}\|\varepsilon^{\ell}\|^{4q}_{H^1}\right]\right)^{\frac{1}{4q}} 
			\leq \Bigl(1+ \frac{1}{1- \alpha q}\Bigr)^{\frac{1}{4\alpha q}} \left(\mE\left[\exp\left( \beta q\sum_{n = 0}^{M-1}{\mathcal{X}}_n\right)\right]\right)^{\frac{1}{4\beta q}} \,\Bigl(\mE\Bigl[\sup_{0 \leq \ell < M} \mathcal{S}_{\ell}\Bigr]\Bigr)^{\frac{1}{4}}.
		\end{align}
		To get the error estimate \eqref{eq3.20}, we continue to estimate the right-hand side of \eqref{eq3.24} as follows: firstly, the exponential term can be bounded by using Lemma \ref{lemma_expo_moment_H} and Lemma \ref{lemma_exponential_discrete} and the discrete Jensen inequality as follows:
		\begin{align*}
			\mE\left[\exp\left( 10q\sum_{n = 0}^{M-1}\mathcal{X}_n\right)\right] &\leq \exp\left(10q\left(C_G^2(20 + 8C_G^2 + 384 C_G^2k) +8 + \frac{C_e^2}{2} \right)T \right)\\\nonumber
			&\qquad\times\mE\left[k\sum_{n = 0}^{M-1} \exp\left(10q\left(\|u(t_{n+1})\|^2_{H^1} + \|u(t_{n})\|^2_{H^1} + \|u_h^{n+1}\|^2_{H^1} + \|u_h^n\|^2_{H^1}\right)\right)\right]\\\nonumber
			&\leq \exp\left(10q\left(C_G^2(20 + 8C_G^2 + 384 C_G^2k) +8 + \frac{C_e^2}{2} \right)T \right)2\left(C_{exp,1} + C_{exp,2}\right)T\\\nonumber
			&:= \tilde{C}_0^{40q},
		\end{align*}
		where $40 q \leq \frac{\nu}{4 L_0^2}$ is guaranteed because $L_0 \leq \frac{\sqrt{\nu}}{4\sqrt{10q}}$. 
		
		To obtain the rates of convergence as in \eqref{eq3.20}, we estimate $\Bigl(\mE\Bigl[\sup_{0 \leq \ell < M} \mathcal{S}_{\ell}\Bigr]\Bigr)^{\frac14}$ as follows:
		\begin{align*}
			\mE\Bigl[\sup_{0 \leq \ell < M} \mathcal{S}_{\ell}\Bigr] &:\leq Y_1 + Y_2+ Y_3 + Y_4 + Y_5.
		\end{align*}
		
		Using Lemma \ref{lem:moment-bound:H1}, Lemma \ref{lem:moment-bound:H2} and Lemma \ref{lemma_highmoment_discrete}, we obtain
		\begin{align*}
			Y_1 &= C kh^{4}\sum_{n = 0}^{M-1}\mE\left[\|u(t_{n+1}) + u_h^{n+1}\|^4_{H^{1}}\|u(t_{n+1})\|^4_{H^2} + C_G^4 \|u(t_n)\|^4_{H^{2}}\right]\\\nonumber
			&\leq C\left((C_{1,4} + C_{4,3})C_{2,4} + C_{2,2}\right)h^{4} := C_{Y_1} h^{4}.
		\end{align*}
		
		Using Lemma \ref{lem:Lipschitz:H^1:u}, and Lemma \ref{lem:moment-bound:H1}, we obtain
		\begin{align*}
			Y_2 + Y_3 &=  C(\nu^2+C_G^2)\sum_{n = 0}^{\ell}\int_{t_n}^{t_{n+1}}\mE\left[\|u(t_{n+1}) - u(s)\|^4_{L^2}\right]\, ds \\\nonumber
			&\qquad + CC_e^4\sum_{n = 0}^{\ell}\int_{t_n}^{t_{n+1}}\left[\left(\|u(s)\|^4_{H^1} + \| u(t_{n+1})\|^4_{H^1}\right)\|u(t_{n+1}) - u(s)\|^4_{H^1}\right]\, ds\\\nonumber
			&\leq C(\nu^2+C_G^2)TC_{3,2} k^2 + CC_e^4 T2C_{1,4} k^2:=C_{Y_2,Y_3} k^2.
		\end{align*}
		
		Using Burkholder-Davis-Gundy inequality, the condition \eqref{Assump_Lipschitz} with $m=0$, and Lemma \ref{lem:Lipschitz:H^1:u}, we also get
		\begin{align*}
			Y_4 &= 64\sum_{n = 0}^{M-1} \mE\left[\left\|\int_{t_n}^{t_{n+1}} (G(u(s)) - G(u(t_{n})))\, dW(s)\right\|^4_{L^2}\right],\\\nonumber
			&\leq C\sum_{n = 0}^{M-1}\left(\mE\left[\int_{t_n}^{t_{n+1}} \left\|G(u(s)) - G(u(t_{n}))\right\|^2_{L^2}\, ds\right]\right)^{2},\\\nonumber
			&\leq CC_G^4k^3:= C_{Y_4} k^2.
		\end{align*}
		
		Finally, using Lemma \ref{lem:moment-bound:H1}, and Lemma \ref{lemma_highmoment_discrete}, we have
		\begin{align*}
			Y_5 &\leq 32k^2\mE\left[\sup_{0 \leq \ell \leq M-1}\left(4 + \frac{C_e^2}{4} + \|u(t_{n+1})\|^2_{H^1} + \|u_h^{n+1}\|^2_{H^1} \right)^2\left(\|\varepsilon^{\ell+1}\|^2_{H^1} + \|\varepsilon^{\ell}\|^2_{H^1}\right)^2\right]\\\nonumber
			&\leq C(C_{1,4} + C_{4,3}):= C_{Y_5} k^2. 
		\end{align*}

		Having gathered all estimates from $Y_1, ..., Y_5$, we obtain the desired estimate \eqref{eq3.20}.
		
		Lastly, it is time to verify that $\{\mathcal{J}_{\ell}\}$ is a martingale. First, using  the Burkholder-Davis-Gundy inequality, the assumption \eqref{Assump_Lipschitz}, Lemma \ref{lem:moment-bound:H1}, and Lemma \ref{lemma_highmoment_discrete}, we can easily obtain
		\begin{align}\label{eq_3.26}
			\mE[|\mathcal{J}_{\ell}|]  <\infty.
		\end{align}
		
		Furthermore, for any $0\leq n \leq M-1$, using the fact that $\mE[(\Delta W_n)^4] = 3k^2$, the independence of the increments $\Delta W_n$, It\^o isometry, and the martingale property of It\^o integrals,  we also have
		\begin{align*}
			\mE[J_n] &=  4\mE\left[16\left\|(G(u(t_n)) - G(u_h^n))\Delta W_n\right\|^4_{L^2} - 48k^2\|G(u(t_n)) - G(u^n_h)\|^4_{L^2}\right]\\\nonumber
			&\qquad+  4\mE\left[4\left\|\int_{t_n}^{t_{n+1}} (G(u(s)) - G(u_h^n))\, dW(s)\right\|^2_{L^2} \|\varepsilon^n\|^2_{H^1} \right.\\\nonumber
			&\qquad\qquad\left.- 4\int_{t_{n}}^{t_{n+1}}\|G(u(s)) - G(u^n_h)\|^2_{L^2}\|\varepsilon^n\|^2_{H^1}\, ds\right]\\\nonumber
			&\qquad +4\mE\left[2\left\|(G(u(t_n)) - G(u_h^n))\Delta W_n\right\|^2_{L^2}\|\varepsilon^{n}\|^2_{H^1} - 2k\|G(u(t_n)) - G(u_h^n)\|^2_{L^2}\|\varepsilon^n\|^2_{H^1}\right]\\\nonumber
			&\qquad+  4\mE\left[\left(\int_{t_n}^{t_{n+1}} (G(u(s)) - G(u_h^n))\, dW(s),  \varepsilon^n\right)\|\varepsilon^{n}\|^2_{H^1}\right]\\\nonumber
			&= 0 + 0 + 0 + 0 = 0.
		\end{align*}
		Then,  the conditional expectation of $\mathcal{J}_{\ell}$ given $\{\mathcal{J}_n\}_{n=0}^{\ell-1}$ is
		\begin{align}\label{eq3.27}
			\mE\bigl[\mathcal{J}_{\ell} | \mathcal{J}_0, \mathcal{J}_1, ..., \mathcal{J}_n \bigr]  
			&= \mE\bigl[J_0 + J_1 + ... + J_n| \mathcal{J}_0, \mathcal{J}_1, ... , \mathcal{J}_n \bigr] \\\nonumber
			&\qquad+ \mE \bigl[J_{n+1} + ... + J_{\ell}| \mathcal{J}_0, \mathcal{J}_1, ... , \mathcal{J}_n \bigr]\\\nonumber
			&= \mathcal{J}_{n} + \mE\bigl[J_{n+1} + ... + J_{\ell} \bigr] = \mathcal{J}_n + 0 = \mathcal{J}_n.
		\end{align}
		
		Thus, we conclude that $\bigl\{\mathcal{J}_{\ell}; \ell \geq 0\bigr\}$ is a martingale using \eqref{eq_3.26} and \eqref{eq3.27}. 
		
		The proof is completed by combining the triangle inequality, \eqref{eq3.20} and \eqref{projection_ineq}.
		
	\end{proof}

	\subsection{Partial expectation error estimates in the case of general multiplicative noise}\label{sub-sec4.3}
	In this part, we establish some error estimates for the fully discrete scheme \ref{Scheme_Standard} in the case of general multiplicative noise. In this case, the diffusive function $G$ may not be bounded so the exponential stabilities from Lemma \ref{lemma_expo_moment_H} and Lemma \ref{lemma_exponential_discrete} are unclear how to obtain. So, to control the nonlinearity, we introduce the following sequence of subsets of the sample space
	\begin{align}
		{\Omega}_{\rho, m} := \left\{\omega \in \Omega; \, \sup_{t \leq t_m} \|u(t)\|^2_{H^1} \leq \rho \right\} \cap \left\{\omega \in \Omega; \, \sup_{0\leq n \leq m} \|u_h^n\|^2_{H^1} \leq \rho \right\},
	\end{align}
	where $u$ is the variational solution from Theorem \ref{thm:well-posed} and for some $\rho>0$ specified later. We observe that ${\Omega}_{\rho,0} \supset {\Omega}_{\rho,1} \supset ... \supset {\Omega}_{\rho,\ell}$. 
	
	\begin{remark}
		In the error estimate stated in Theorem~\ref{Theorem_partial_expectation_error}, we choose
		\begin{align}
			\rho(k) := \frac{\ln(\ln(1/k))}{8T} >0.
		\end{align}
		
		With this choice of $\rho(k)$, an application of the Markov inequality together with Lemma~\ref{lem:moment-bound:H1} and Lemma \ref{lemma_highmoment_discrete} yields
		\begin{align*}
			\mP(\Omega_{\rho,M}^c)
			&\le \frac{8T}{\ln(\ln(1/k))}
			\mE\!\left[
			\sup_{0\le t\le T}
			\|u(t)\|_{H^1}^2 + \sup_{0\leq n \leq M}\|u_h^n\|^2_{H^1}
			\right].
		\end{align*}
		Since the expectation on the right-hand side is finite, we conclude that
		\[
		\mP(\Omega_{\rho,M}^c) \longrightarrow 0
		\qquad \text{as } k \to 0,
		\]
		and hence
		\[
		\mP(\Omega_{\rho,M}) \longrightarrow 1
		\qquad \text{as } k \to 0.
		\]
		
		Therefore, the convergence of the numerical solutions implied by Theorem~\ref{Theorem_partial_expectation_error} is convergence in probability.
	\end{remark}

	\bigskip
	
	\begin{theorem}\label{Theorem_partial_expectation_error} 
		Let $u$ be the variational solution from Theorem \eqref{thm:well-posed} and $\{(u_h^{n},\theta_h^n)\}_{n=1}^M$ be generated by \eqref{Scheme_Standard}. Let $u_0 \in L^2(\Omega; H^2(D))\cap L^{34}(\Ome;\mV)$. Assume that $G$ satisfies the condition \eqref{Assump_Lipschitz}. Moreover, suppose that 
		there exists a constant $C>0$ such that
		\begin{align}
			\left(\mE\left[\|u_0 - u_h^0\|^{2}_{H^1}\right]\right)^{\frac{1}{2}} 
			\leq C h,
		\end{align} 
		Then, there holds
		\begin{align}\label{eq3.30}
			&\left(\max_{1\leq n\leq M}\mE\left[\mathbf{1}_{{\Omega}_{\rho, n}}\|u(t_n) - u_h^n\|^{2}_{H^1}\right]\right)^{\frac{1}{2}} 
			\leq \widehat{C}_3\sqrt{\ln(1/k)}\left(k^{\frac12} + h\right),
		\end{align}
		where $\widehat{C}_3= C(u_0, T)$ is a positive constant. 
	\end{theorem}
	\begin{proof}
		First, multiplying \eqref{eq_3.14} with the indicator function $\mathbf{1}_{\Omega_{\rho, n}}$ and using the fact that $\mathbf{1}_{\Omega_{\rho, n}} \leq 1$, we obtain
		\begin{align}\label{eq_3.31}
			&\frac{1}{2}\left[\mathbf{1}_{\Omega_{\rho, n}}\|\varepsilon^{n+1}\|^2_{H^1} - \|\varepsilon^n\|^2_{H^1}\right] + \frac14\mathbf{1}_{\Omega_{\rho, n}}\|\varepsilon^{n+1} - \varepsilon^n\|^2_{H^1} + \frac{\nu k}{2}\mathbf{1}_{\Omega_{\rho, n}}\|\varepsilon^{n+1}\|^2_{L^2} \\\nonumber
			&\leq Ckh^2\mathbf{1}_{\Omega_{\rho, n}}\bigl(\nu h^2 + h^2 + C_e^2 \|u(t_{n+1}) + u_h^{n+1}\|^2_{H^1}\bigr)\|u(t_{n+1})\|^2_{H^2} + kC_G^2C h^{4}\mathbf{1}_{\Omega_{\rho, n}}\|u(t_n)\|^2_{H^{2}}\\\nonumber
			&\qquad +\nu \int_{t_{n}}^{t_{n+1}} \mathbf{1}_{\Omega_{\rho, n}}\|u(t_{n+1}) - u(s)\|^2_{L^2}\, ds \\\nonumber
			&\qquad+  CC_e^2\int_{t_{n}}^{t_{n+1}}\mathbf{1}_{\Omega_{\rho, n}} \|u(s) - u(t_{n+1})\|^2_{H^1}\|u(s) + u(t_{n+1})\|^2_{H^1}\,ds \\\nonumber
			&\qquad+ 2\mathbf{1}_{\Omega_{\rho, n}}\left\|\int_{t_n}^{t_{n+1}} (G(u(s)) - G(u(t_{n})))\, dW(s)\right\|^2_{L^2} \\\nonumber
			&\qquad +2\mathbf{1}_{\Omega_{\rho, n}}\left\|(G(u(t_n)) - G(u_h^n))\Delta W_n\right\|^2_{L^2} - 2k\mathbf{1}_{\Omega_{\rho, n}}\|G(u(t_n)) - G(u_h^n)\|^2_{L^2} \\\nonumber
			&\qquad+  \mathbf{1}_{\Omega_{\rho, n}}\left(\int_{t_n}^{t_{n+1}} (G(u(s)) - G(u^n))\, dW(s),  \varepsilon^n\right)\\\nonumber
			&\qquad + k\mathbf{1}_{\Omega_{\rho, n}} \left(3 + \frac{C_e^2}{4} + \|u(t_{n+1})\|^2_{H^1} + \|u_h^{n+1}\|^2_{H^1}\right)\|\varepsilon^{n+1}\|^2_{H^1} + 2C_G^2k\mathbf{1}_{\Omega_{\rho, n}} \|\varepsilon^{n}\|^2_{L^2}.
		\end{align}
		
		Next, using the simple inequality $\mathbf{1}_{\Omega_{\rho, n+1}} \leq \mathbf{1}_{\Omega_{\rho, n}}$ for all $n =0, ..., M-1$, taking the expectation, and then applying the summation $\sum_{n=0}^{\ell}$, we get
		\begin{align}\label{eq3.32}
			&\frac{1}{2}\mE\left[\mathbf{1}_{\Omega_{\rho, \ell+1}}\|\varepsilon^{\ell+1}\|^2_{H^1}   +\frac12\sum_{n = 0}^{\ell}\mathbf{1}_{\Omega_{\rho, n}}\|\varepsilon^{n+1} - \varepsilon^n\|^2_{H^1} + \nu k\sum_{n = 0}^{\ell}\mathbf{1}_{\Omega_{\rho, n+1}}\|\varepsilon^{n+1}\|^2_{L^2}\right] \\\nonumber
			&\leq Ckh^2\sum_{n = 0}^{\ell}\mE\left[\mathbf{1}_{\Omega_{\rho, n}}\bigl(\nu h^2 + h^2 + C_e^2 \|u(t_{n+1}) + u_h^{n+1}\|^2_{H^1}\bigr)\|u(t_{n+1})\|^2_{H^2}\right] \\\nonumber
			&\qquad + kC_G^2C h^{4}\sum_{n = 0}^{\ell}\mE\left[\mathbf{1}_{\Omega_{\rho, n}}\|u(t_n)\|^2_{H^{2}}\right] +\nu \sum_{n = 0}^{\ell}\int_{t_{n}}^{t_{n+1}} \mE\left[\mathbf{1}_{\Omega_{\rho, n}}\|u(t_{n+1}) - u(s)\|^2_{L^2}\right]\, ds \\\nonumber
			&\qquad+  CC_e^2\sum_{n = 0}^{\ell}\int_{t_{n}}^{t_{n+1}}\mE\left[\mathbf{1}_{\Omega_{\rho, n}} \|u(s) - u(t_{n+1})\|^2_{H^1}\|u(s) + u(t_{n+1})\|^2_{H^1}\right]\,ds \\\nonumber
			&\qquad+ 2\sum_{n = 0}^{\ell}\mE\left[\mathbf{1}_{\Omega_{\rho, n}}\left\|\int_{t_n}^{t_{n+1}} (G(u(s)) - G(u(t_{n})))\, dW(s)\right\|^2_{L^2}\right] \\\nonumber
			&\qquad +\sum_{n = 0}^{\ell}\mE\left[2\mathbf{1}_{\Omega_{\rho, n}}\left\|(G(u(t_n)) - G(u_h^n))\Delta W_n\right\|^2_{L^2} - 2k\mathbf{1}_{\Omega_{\rho, n}}\|G(u(t_n)) - G(u_h^n)\|^2_{L^2} \right]\\\nonumber
			&\qquad+  \mathbf{1}_{\Omega_{\rho, n}}\left(\int_{t_n}^{t_{n+1}} (G(u(s)) - G(u^n))\, dW(s),  \varepsilon^n\right)\\\nonumber
			&\qquad + k\sum_{n = 0}^{\ell}\mE\left[\mathbf{1}_{\Omega_{\rho, n}} \left(3 + \frac{C_e^2}{4} + \|u(t_{n+1})\|^2_{H^1} + \|u_h^{n+1}\|^2_{H^1}\right)\|\varepsilon^{n+1}\|^2_{H^1}\right] + 2C_G^2k\sum_{n = 0}^{\ell}\mE\left[\mathbf{1}_{\Omega_{\rho, n}} \|\varepsilon^{n}\|^2_{L^2}\right]\\\nonumber
			&:= \mathcal{Z}_1 + \mathcal{Z}_2 + ... + \mathcal{Z}_9.
		\end{align}
		
		Using the martingale property of It\^o integrals, we immediately conclude that $\mathcal{Z}_6 =0=\mathcal{Z}_7$. Next, using the Cauchy-Schwarz inequality, Lemma \ref{lem:moment-bound:H1}, Lemma \ref{lemma_highmoment_discrete} and Lemma \ref{lem:moment-bound:H2}, we estimate $\mathcal{Z}_1$ and $\mathcal{Z}_2$ as follows:
		\begin{align*}
			\mathcal{Z}_1 + \mathcal{Z}_2 &\leq Ch^2(C_e^2(C_{1,4} + C_{4,2})C_{2,2} + C_{2,1})T := C_{\mathcal{Z}_1,\mathcal{Z}_2} h^{2}.
		\end{align*}
		
		Now, using Lemma \ref{lem:Lipschitz:H^1:u}, the Cauchy-Schwarz inequality, Lemma \ref{lem:moment-bound:H1}, and Lemma \ref{lemma_highmoment_discrete}, we have
		\begin{align*}
			\mathcal{Z}_3 + \mathcal{Z}_4 &\leq \nu T C_{3,1} k + CC_e^2 TC_{1,2}C_{3,2} k :=C_{\mathcal{Z}_3,\mathcal{Z}_4} k.
		\end{align*}
		
		In order to estimate $\mathcal{Z}_5$, first we use the It\^o isometry and the condition \eqref{Assump_Lipschitz}, then we also apply Lemma \ref{lem:Lipschitz:H^1:u} to get
		\begin{align*}
			\mathcal{Z}_5 &\leq 2\sum_{n = 0}^{\ell}\mE\left[\int_{t_n}^{t_{n+1}} \left\|G(u(s)) - G(u(t_{n}))\right\|^2_{L^2}\, ds\right] \\\nonumber
			&\leq 2C_G^2\sum_{n = 0}^{\ell}\mE\left[\int_{t_n}^{t_{n+1}} \left\|u(s) - u(t_{n})\right\|^2_{L^2}\, ds\right] \\\nonumber
			&\leq 2C_G^2T C_{3,1} k := C_{\mathcal{Z}_5} k.
		\end{align*}
		
		The term $\mathcal{Z}_9$ is retained until the final step to facilitate the application of the standard discrete Gronwall inequality. In contrast, $\mathcal{Z}_8$ is treated by adding and subtracting $\|\varepsilon^n\|^2_{H^1}$, yielding
		\begin{align*}
			\mathcal{Z}_8&=k\sum_{n = 0}^{\ell}\mE\left[\mathbf{1}_{\Omega_{\rho, n}}\left(3 + \frac{C_e^2}{4} + \|u(t_{n+1})\|^2_{H^1} + \|u_h^{n+1}\|^2_{H^1}  \right)\left(\|\varepsilon^{n+1}\|^2_{H^1} - \|\varepsilon^{n}\|^2_{H^1}\right)\right] \\\nonumber
			&\qquad+  k\sum_{n = 0}^{\ell}\mE\left[\mathbf{1}_{\Omega_{\rho, n}}\left( 3 + \frac{C_e^2}{4} + \|u(t_{n+1})\|^2_{H^1} + \|u_h^{n+1}\|^2_{H^1} \right) \|\varepsilon^{n}\|^2_{H^1}\right]\\\nonumber
			&=k\sum_{n = 0}^{\ell}\mE\left[\mathbf{1}_{\Omega_{\rho, n}}\left(3 + \frac{C_e^2}{4} + \|u(t_{n+1})\|^2_{H^1} + \|u_h^{n+1}\|^2_{H^1} \right)\left(\|\varepsilon^{n+1}\|_{H^1} + \|\varepsilon^{n}\|_{H^1}\right)\left(\|\varepsilon^{n+1}\|_{H^1} - \|\varepsilon^{n}\|_{H^1}\right)\right] \\\nonumber
			&\qquad+  k\sum_{n = 0}^{\ell}\mE\left[\mathbf{1}_{\Omega_{\rho, n}}\left( 3 + \frac{C_e^2}{4} + \|u(t_{n+1})\|^2_{H^1} + \|u_h^{n+1}\|^2_{H^1} \right) \|\varepsilon^{n}\|^2_{H^1}\right]\\\nonumber
			&:= \mathcal{Z}_{8,1} + \mathcal{Z}_{8,2}.
		\end{align*}
		
		Next, using the triangle inequality, Young inequality, and then Lemma \ref{lem:moment-bound:H1}, and Lemma \ref{lemma_highmoment_discrete} to estimate $\mathcal{Z}_{8,1}$, we get
		\begin{align*}
			\mathcal{Z}_{8,1} &\leq k\sum_{n = 0}^{\ell}\mE\left[\mathbf{1}_{\Omega_{\rho, n}}\left(3 + \frac{C_e^2}{4} + \|u(t_{n+1})\|^2_{H^1} + \|u_h^{n+1}\|^2_{H^1} \right)\left(\|\varepsilon^{n+1}\|_{H^1} + \|\varepsilon^{n}\|_{H^1}\right)\left(\|\varepsilon^{n+1} - \varepsilon^n\|_{H^1} \right)\right] \\\nonumber
			&\leq \frac{1}{8} \sum_{n = 0}^{\ell}\mE\left[\mathbf{1}_{\Omega_{\rho, n}}\|\varepsilon^{n+1} - \varepsilon^n\|^2_{H^1}\right] \\\nonumber
			&\qquad+ 2k^2\sum_{n = 0}^{\ell}\mE\left[\mathbf{1}_{\Omega_{\rho, n}}\left( 3 + \frac{C_e^2}{4} + \|u(t_{n+1})\|^2_{H^1} + \|u_h^{n+1}\|^2_{H^1} \right)^2\left(\|\varepsilon^{n+1}\|_{H^1} + \|\varepsilon^{n}\|_{H^1}\right)^2\right] \\\nonumber
			&\leq \frac{1}{8} \sum_{n = 0}^{\ell}\mE\left[\mathbf{1}_{\Omega_{\rho, n}}\|\varepsilon^{n+1} - \varepsilon^n\|^2_{H^1}\right] +  CT(9 + C_e^4 + C_{1,4} + C_{4,3}) (C_{1,2} + C_{4,2}) k\\\nonumber
			&:= \frac{1}{8} \sum_{n = 0}^{\ell}\mE\left[\mathbf{1}_{\Omega_{\rho, n}}\|\varepsilon^{n+1} - \varepsilon^n\|^2_{H^1}\right] +  C_{\mathcal{Z}_{8,1}}k,
		\end{align*}
		where the term $\frac{1}{8} \sum_{n = 0}^{\ell}\mE\left[\mathbbm{1}_{\Omega_{\rho, n}}\|\varepsilon^{n+1} - \varepsilon^n\|^2_{L^2}\right]$ on the right-hand side of $\mathcal{Z}_{8,1}$ will be combined to the left-hand side of \eqref{eq3.32}.
		
		Next, using the triangle inequality and then the definition of $\Omega_{\rho, n}$, we have
		\begin{align*}
			\mathcal{Z}_{8,2} &\leq k\sum_{n = 0}^{\ell}\mE\left[\mathbf{1}_{\Omega_{\rho, n}}\left( 3 + \frac{C_e^2}{4} + 2\|u(t_{n+1}) - u(t_n)\|^2_{H^1} +2 \|u_h^{n+1}- u_h^n\|^2_{H^1} \right) \|\varepsilon^{n}\|^2_{H^1}\right] \\\nonumber
			&\qquad+ k\sum_{n = 0}^{\ell}\mE\left[\mathbf{1}_{\Omega_{\rho, n}}\left( 3 + \frac{C_e^2}{4} + 2\|u(t_{n})\|^2_{H^1} + 2\|u_h^{n}\|^2_{H^1} \right) \|\varepsilon^{n}\|^2_{H^1}\right]\\\nonumber
			&\leq k\sum_{n = 0}^{\ell}\mE\left[\mathbf{1}_{\Omega_{\rho, n}}2\left(  \|u(t_{n+1}) - u(t_n)\|^2_{H^1} + \|u_h^{n+1}- u_h^n\|^2_{H^1} \right) \|\varepsilon^{n}\|^2_{H^1}\right] \\\nonumber
			&\qquad+ k\sum_{n = 0}^{\ell}\mE\left[\mathbf{1}_{\Omega_{\rho, n}}\left( 3 + \frac{C_e^2}{4} + 4\rho\right) \|\varepsilon^{n}\|^2_{H^1}\right].
		\end{align*} 
		
		The second term
		\[
		k\sum_{n = 0}^{\ell}\mathbb{E}\!\left[\mathbf{1}_{\Omega_{\rho, n}}\left( 3 + \frac{C_e^2}{4} + 4\rho\right) \|\varepsilon^{n}\|^2_{H^1}\right]
		\]
		of $\mathcal{Z}_{8,2}$ is ready for the application of the discrete Gronwall inequality, while the first term needs to be further analyzed. To this end, we add and subtract
		$\pm {P}_h u(t_{n+1}) \pm {P}_h u(t_n)$ to the first term and then using Lemma \ref{lem:moment-bound:H1}, Lemma \ref{lemma_highmoment_discrete} and Lemma \ref{lem:Lipschitz:H^1:u} as follows:
		\begin{align}\label{eq4.55}
			&	k\sum_{n = 0}^{\ell}\mE\left[\mathbf{1}_{\Omega_{\rho, n}}2\left(  \|u(t_{n+1}) - u(t_n)\|^2_{H^1} + \|u_h^{n+1}- u_h^n\|^2_{H^1} \right) \|\varepsilon^{n}\|^2_{H^1}\right]\\\nonumber
			&= 2k\sum_{n = 0}^{\ell}\mE\left[\mathbf{1}_{\Omega_{\rho, n}}\left(  \|u(t_{n+1}) - u(t_n)\|^2_{H^1} \right) \|\varepsilon^{n}\|^2_{H^1}\right] + 2k\sum_{n = 0}^{\ell}\mE\left[\mathbf{1}_{\Omega_{\rho, n}}\left(  \|u_h^{n+1}- u_h^n\|^2_{H^1} \right) \|\varepsilon^{n}\|^2_{H^1}\right]\\\nonumber
			&\leq \sum_{n = 0}^{\ell}\mE\left[\|u(t_{n+1}) - u(t_n)\|^4_{H^1}\right] + k^2\sum_{n = 0}^{\ell}\mE\left[\|\varepsilon^n\|^4_{H^1}\right] \\\nonumber
			&\qquad + 2k\sum_{n = 0}^{\ell}\mE\left[\mathbf{1}_{\Omega_{\rho, n}}  \|u_h^{n+1}- u_h^n\|_{H^1}\left(\|u_h^{n+1} - u^n_h\|_{H^1} \|\varepsilon^{n}\|^2_{H^1}\right) \right]\\\nonumber
			&\leq \sum_{n = 0}^{\ell}\mE\left[\|u(t_{n+1}) - u(t_n)\|^4_{H^1}\right] + k^2\sum_{n = 0}^{\ell}\mE\left[\|\varepsilon^n\|^4_{H^1}\right] \\\nonumber
			&\qquad + 2k\sum_{n = 0}^{\ell}\mE\left[\mathbf{1}_{\Omega_{\rho, n}}  \left(\|\varepsilon^{n+1}- \varepsilon^n\|_{H^1} + \|P_h(u(t_{n+1}) - u(t_n))\|_{H^1}\right)\left(\|u_h^{n+1} - u^n_h\|_{H^1} \|\varepsilon^{n}\|^2_{H^1}\right) \right]\\\nonumber
			&= C_{3,2}Tk + T(C_{1,2} + C_{4,2})k\\\nonumber
			&\qquad + 2k\sum_{n = 0}^{\ell}\mE\left[\mathbf{1}_{\Omega_{\rho, n}}  \left(\|\varepsilon^{n+1}- \varepsilon^n\|_{H^1}\right)\left(\|u_h^{n+1} - u^n_h\|_{H^1} \|\varepsilon^{n}\|^2_{H^1}\right) \right]\\\nonumber
			&\qquad + 2k\sum_{n = 0}^{\ell}\mE\left[\mathbf{1}_{\Omega_{\rho, n}}  \left(\|u(t_{n+1}) - u(t_n)\|_{H^1}\right)\left(\|u_h^{n+1} - u^n_h\|_{H^1} \|\varepsilon^{n}\|^2_{H^1}\right) \right]\\\nonumber
			&\leq C_{3,2}Tk + T(C_{1,2} + C_{4,2})k\\\nonumber
			&\qquad + \frac{1}{16}\sum_{n = 0}^{\ell}\mE\left[\mathbf{1}_{\Omega_{\rho, n}}  \|\varepsilon^{n+1}- \varepsilon^n\|^2_{H^1} \right] + 16k^2 \sum_{n = 0}^{\ell} \mE\left[\mathbf{1}_{\Omega_{\rho, n}} \|u_h^{n+1} - u^n_h\|^2_{H^1} \|\varepsilon^{n}\|^4_{H^1}\right]\\\nonumber
			&\qquad + k\sum_{n = 0}^{\ell}\mE\left[\mathbf{1}_{\Omega_{\rho, n}}  \left(\|u(t_{n+1}) - u(t_n)\|^2_{H^1}\|u_h^{n+1} - u^n_h\|^2_{H^1}\right) \right] + k\sum_{n=0}^{\ell} \mE\left[\mathbf{1}_{\Omega_{\rho, n}} \|\varepsilon^n\|^4_{H^1}\right]\\\nonumber
			&\leq C_{3,2}Tk + T(C_{1,2} + C_{4,2})k\\\nonumber
			&\qquad + \frac{1}{16}\sum_{n = 0}^{\ell}\mE\left[\mathbf{1}_{\Omega_{\rho, n}}  \|\varepsilon^{n+1}- \varepsilon^n\|^2_{H^1} \right] + 16TC_{4,2}(C_{1,4} + C_{4,3}) k \\\nonumber
			&\qquad + C_{3,2}C_{4,2}T k + k\sum_{n=0}^{\ell} \mE\left[\mathbf{1}_{\Omega_{\rho, n}} 4\rho\|\varepsilon^n\|^2_{H^1}\right]\\\nonumber
			&:= C_{\mathcal{Z}_{8,2}} k + \frac{1}{16}\sum_{n = 0}^{\ell}\mE\left[\mathbf{1}_{\Omega_{\rho, n}}  \|\varepsilon^{n+1}- \varepsilon^n\|^2_{H^1} \right] + k\sum_{n=0}^{\ell} \mE\left[\mathbf{1}_{\Omega_{\rho, n}} 4\rho\|\varepsilon^n\|^2_{H^1}\right].
		\end{align}
		
		It should be noted that the term $\frac{1}{16}\sum_{n = 0}^{\ell}\mE\left[\mathbf{1}_{\Omega_{\rho, n}}  \|\varepsilon^{n+1}- \varepsilon^n\|^2_{H^1} \right]$ above will be absorbed to the left-hand side of \eqref{eq3.32} at the end after we collect all the estimates. 
		
		Now, substituting all the estimates from $\mathcal{Z}_1, ..., \mathcal{Z}_9$ into \eqref{eq3.32} we arrive at
		\begin{align}\label{eq3.33}
			&\frac{1}{2}\mE\left[\mathbf{1}_{\Omega_{\rho, \ell+1}}\|\varepsilon^{\ell+1}\|^2_{H^1}   +\frac18\sum_{n = 0}^{\ell}\mathbf{1}_{\Omega_{\rho, n}}\|\varepsilon^{n+1} - \varepsilon^n\|^2_{H^1}\right] \\\nonumber
			&\leq C_{\mathcal{Z}_1,\mathcal{Z}_2} h^{2} + \left(C_{\mathcal{Z}_3,\mathcal{Z}_4} + C_{\mathcal{Z}_5} + C_{\mathcal{Z}_{8,1}} + C_{\mathcal{Z}_{8,2}}\right) k\\\nonumber
			&\qquad +\left( 3 + \frac{C_e^2}{4} + 2C_G^2 + 8\rho\right) k\sum_{n = 0}^{\ell}\mE\left[\mathbf{1}_{\Omega_{\rho, n}} \|\varepsilon^{n}\|^2_{H^1}\right].
		\end{align}
		
		Using the discrete deterministic Gronwall inequality to \eqref{eq3.33}, we have
		\begin{align}\label{eq3.34}
			\frac{1}{2}\mE\left[\mathbf{1}_{\Omega_{\rho, \ell+1}}\|\varepsilon^{\ell+1}\|^2_{H^1}\right] 
			&\leq \left\{C_{\mathcal{Z}_1,\mathcal{Z}_2} h^{2} + \left(C_{\mathcal{Z}_3,\mathcal{Z}_4} + C_{\mathcal{Z}_5} + C_{\mathcal{Z}_{8,1}} + C_{\mathcal{Z}_{8,2}}\right) k\right\}\\\nonumber
			&\qquad\times\exp\left(T\left( 3 + \frac{C_e^2}{4} + 2C_G^2 + 8\rho\right)\right)\\\nonumber
			&=  \left\{C_{\mathcal{Z}_1,\mathcal{Z}_2} h^{2} + \left(C_{\mathcal{Z}_3,\mathcal{Z}_4} + C_{\mathcal{Z}_5} + C_{\mathcal{Z}_{8,1}} + C_{\mathcal{Z}_{8,2}}\right) k\right\}\\\nonumber
			&\qquad\times\exp\left(T\left( 3 + \frac{C_e^2}{4} + 2C_G^2\right)\right)\, \ln(1/k)\\\nonumber
			&:= \widehat{C}_3 \ln(1/k)\,\left(k + h^{2}\right),
		\end{align}
		where $\rho = \frac{\ln(\ln(1/k))}{8T}$ has been used to obtain the first equality of \eqref{eq3.34}.
		
		Then, the desired estimate \eqref{eq3.30} is obtained by using the triangle inequality and combining with \eqref{eq3.34} and \eqref{projection_ineq}.

		This finishes the proof.
		
	\end{proof}

	\section{Numerical experiments} \label{sec-5}
	
	In this section, we present a series of numerical experiments designed to
	validate the theoretical results developed in the previous sections.
	Throughout all experiments, we set the computational domain
	$D=(0,1)^2\subset\mathbb{R}^2$, the final time $T=1$, the damping parameter
	$\nu=1$, and $u_0 = 0$.
	
	The Wiener process $W(t)$ in \eqref{eq1.1} is taken to be real-valued and is
	simulated using a reference time step of size $k_0=2^{-12}$.
	We consider two choices of the diffusion coefficient $G$, namely
	$G(u)=\alpha\sin(1+u)$ and $G(u)=\alpha u$ with $\alpha>0$.
	These two forms correspond to the bounded and unbounded multiplicative noise
	regimes, respectively, and are consistent with the assumptions employed in the
	error estimates of Theorem~\ref{Theorem_higher_moment} and
	Theorem~\ref{Theorem_partial_expectation_error}. 
	
	For all numerical tests, expectations are approximated using a standard Monte
	Carlo method with $J=400$ independent sample paths.
	The spatial discretization is carried out using conforming $P_1$ piecewise
	linear finite elements, and homogeneous Dirichlet boundary conditions are
	imposed on $u$ in all simulations.

	We implement the Main Algorithm and compute the errors of the numerical solution in the specified norm below:
	\begin{align*}
		L^2_{\omega}L^{\infty}_tH^1_x(u)&:=	\left(\mE\left[\max_{1\leq n \leq M}\|u(t_n) - u^n_h\|^2_{H^1}\right]\right)^{1/2} \\
		&\approx \left(\frac{1}{J}\sum_{j= 1}^J\left(\max_{1\leq n \leq M} \|u^n_{h_{ref}}(\omega_j) - u^n_h(\omega_j)\|^2_{H^1}\right)\right)^{1/2}.
	\end{align*}
	
	We aim to verify the convergence orders predicted in
	Theorems~\ref{Theorem_higher_moment} and~\ref{Theorem_partial_expectation_error}.
	The error bounds established in these theorems are of the form
	$C(k^{1/2}+h)$, which naturally suggests choosing the time step size as
	$k=h^2$ in order to achieve an overall convergence rate of order $O(h)$.
	
	Since exact solutions of the stochastic BBM equation are not available, the
	errors are evaluated by comparing the numerical solution
	$u_h^n(\omega_j)$ with a reference solution
	$u_{h_{\mathrm{ref}}}^n(\omega_j)$ at the same sample path $\omega_j$.
	The reference solution is computed on a finer spatial mesh with mesh size
	$h_{\mathrm{ref}}=h/2$, so that the error is approximated by the difference
	between numerical solutions obtained on two consecutive mesh levels. Moreover, to solve the nonlinear system generated by the fully implicit scheme from the Main Algorithm, we employ a fixed-point iterative solver, e.g. \cite[See Algorithm 3]{feng2024high} and \cite[See Section 4.1]{feng2017finite}, which performs better for the stochastic problems requiring sampling than the traditional Newton method. 
	
	The resulting errors are reported in Tables~\ref{tab:1} and \ref{tab:2}, which confirm the first-order convergence rate predicted by the theoretical analysis.
	
	\begin{table}[htb]
		\renewcommand{\arraystretch}{1.2}
		\centering
		\footnotesize
		\begin{tabular}{|c|c|c|}
			\hline
			\hline
			$(k,h)$ & $L^2_{\omega}L^{\infty}_tH_x^1(u)$ error & Order \\
			\hline
			\hline
			$(k,h)=(2^{-2},2^{-1})$ & 1.2576e0 &  \\
			\hline
			$(k,h)=(2^{-4},2^{-2})$ & 7.8161e-1 & 0.6861 \\
			\hline
			$(k,h)=(2^{-6},2^{-3})$ & 4.0850e-1 & 0.9361 \\
			\hline
			$(k,h)=(2^{-8},2^{-4})$ & 2.28818e-1 & 0.8362 \\
			\hline
			$(k,h)=(2^{-10},2^{-5})$ & 1.119194e-1 & 0.9409 \\
			\hline
		\end{tabular}
		\caption{Errors and rates of convergence for Main Algorithm with $k = h^2$ and $G(u) = \frac{1}{4}u$.}
		\label{tab:1}
	\end{table}
	
	\begin{table}[htb]
		\renewcommand{\arraystretch}{1.2}
		\centering
		\footnotesize
		\begin{tabular}{|c|c|c|}
			\hline
			\hline
			$(k,h)$ & $L^2_{\omega}L^{\infty}_tH_x^1(u)$ error & Order \\
			\hline
			\hline
			$(k,h)=(2^{-2},2^{-1})$ & 1.25228e0 &  \\
			\hline
			$(k,h)=(2^{-4},2^{-2})$ & 7.72239-1 & 0.6974 \\
			\hline
			$(k,h)=(2^{-6},2^{-3})$ & 4.02038e-1 & 0.9383 \\
			\hline
			$(k,h)=(2^{-8},2^{-4})$ & 2.21259e-1 & 0.8617 \\
			\hline
			$(k,h)=(2^{-10},2^{-5})$ & 1.11343e-1 & 0.9634 \\
			\hline
		\end{tabular}
		\caption{Errors and rates of convergence for Main Algorithm with $k = h^2$ and $G(u) = \frac{1}{10}\sin(1+u)$.}
		\label{tab:2}
	\end{table}

	\section*{Declarations of Funding}\,
	
	The author Liet Vo was supported by the National Science Foundation (NSF) under Grant No. DMS-2530211.
	
	\section*{Data Availability} \,
	
	Data sharing is not applicable to this article as no datasets were generated or analysed during the current study.
	
	\section*{Declaration of Conflict of Interest}\,
	
	The authors have no conflict of interest.

	\appendix
	
	\section{Useful results}\label{appendxiA}
	
	First, we state the following discrete stochastic Gronwall inequality from \cite[Theorem 1]{kruse2018discrete}, which plays a vital role in our error analysis in Section \ref{sec-4}.
	
	\begin{lemma}{\cite[Theorem 1]{kruse2018discrete}}\label{Stochastic_Gronwall}
		Let $\left\{M_n\right\}_{n \in \mathbb{N}}$ be an $\left\{\mathcal{F}_n\right\}_{n \in \mathbb{N}}$-martingale satisfying $M_0=0$ on a filtered probability space $\left(\Omega, \mathcal{F},\left\{\mathcal{F}_n\right\}_{n \in \mathbb{N}}, \mathbb{P}\right)$. Let $\left\{X_n\right\}_{n \in \mathbb{N}},\left\{F_n\right\}_{n \in \mathbb{N}}$, and $\left\{G_n\right\}_{n \in \mathbb{N}}$ be sequences of nonnegative and adapted random variables with $\mathbb{E}\left[X_0\right]<\infty$ such that
		\begin{align}\label{ineq2.3}
			X_n \leq F_n+M_n+\sum_{k=0}^{n-1} G_{\ell} X_{\ell} \quad \text { for all } n \in \mathbb{N}.
		\end{align}
		Then, for any $q \in(0,1)$ and a pair of conjugate numbers $\alpha, \beta \in[1, \infty]$, i.e.,  $\frac{1}{\alpha}+\frac{1}{\beta}=1$, satisfying $q \alpha<1$,  there holds 
		\begin{align}\label{ineq2.4}
			\mathbb{E}\left[\sup _{0 \leq \ell \leq n} X_{\ell}^q\right] \leq\left(1+\frac{1}{1-\alpha q}\right)^{\frac{1}{\alpha}}\left\|\prod_{\ell=0}^{n-1}\left(1+G_{\ell}\right)^q\right\|_{L^{\beta}(\Omega)}\left(\mathbb{E}\left[\sup _{0 \leq \ell \leq n} F_{\ell}\right]\right)^q.
		\end{align}
	\end{lemma}

	\section{Proof of Theorem \ref{thm:well-posed}} \label{auxiliary}
	
	We now turn to the proof of the well-posedness as stated in Theorem \ref{thm:well-posed}. Following closely the framework of \cite{de1998stochastic,debussche2011local,ferrario2024uniqueness,flandoli1995martingale,glatt2026long,menaldi2002stochastic,ondrejat2004uniqueness}, the main argument of Theorem \ref{thm:well-posed} essentially consists of three main steps as follows:
	
	Step 1: we consider a regularized version of \eqref{eq1.1} by adding a dissipation effect through a parameter $\varepsilon$, namely,
	\begin{align} \label{eqn:BBM:regularized}
		du^\varepsilon  - d(\Delta u^\varepsilon)+ \varepsilon \triangle^2 u^\varepsilon dt+ [\nu u^\varepsilon + \div( F(u^\varepsilon) )]\, dt&=  G(u^\varepsilon)\,dW(t), 
		&&\mbox{in }(0,T)\times D,\,a.s.\\ 
		u^\varepsilon(0) &= u_0, &&\mbox{on } D,\, a.s. \notag 
	\end{align}
	In Proposition \ref{prop:BBM:regularized:martingale} below, we construct martingale solutions $u^\varepsilon$ for \eqref{eqn:BBM:regularized}. 
	
	Step 2: We demonstrate that $\Law(u^\varepsilon)$ is tight as $\varepsilon\to 0$, thereby producing the existence of martingale solutions $u$ of \eqref{eq1.1}. This is summarized in Proposition \ref{prop:BBM:martingale}.

	Step 3: Lastly, we establish the path-wise uniqueness of $u$ through Lemma \ref{lem:BBM:martinglae:uniqueness}, which together with the existence result from Step 2 concludes Theorem \ref{thm:well-posed}, giving the global well-posedness of \eqref{eq1.1}.

	For the reader's convenience, in what follows, we recapitulate suitable compactness results that we will employ to construct martingale solutions of \eqref{eq1.1}. Given $\alpha\in(0,1)$ and $p\in(1,\infty)$, we denote by $W^{\alpha,p}(0,T;X)$ the Sobolev space of all $u\in L^p(0,T;X)$ such that
	\begin{align*}
		\int_0^T \int_0^T\frac{\|u(t)-u(s)\|_X}{|t-s|^{1+\alpha p}}dt\,ds<\infty,
	\end{align*}
	endowed with the norm
	\begin{align*}
		\|u\|^p_{W^{\alpha,p}(0,T;X)} =\int_0^T\|u(s)\|^p_Xds+\int_0^T \int_0^T\frac{\|u(t)-u(s)\|^p_X}{|t-s|^{1+\alpha p}}dt\,ds.
	\end{align*}
	Also, for $p\ge 1$, $W^{1,p}(0,T;X)$ denotes the Sobolev space of $u\in L^p(0,T;X)$ such that $\frac{du}{dt}\in L^p(0,T;X)$ endowed with the norm
	\begin{align*}
		\|u\|^p_{W^{1,p}(0,T;X)} =\int_0^T\|u(s)\|^p_Xds+\int_0^T \Big\|\frac{du}{dt}(s)\Big\|^p_X ds.
	\end{align*}
	
	\begin{lemma}{\cite[Theorem 2.1]{flandoli1995martingale}} \label{lem:compact:B_0.B.B_1}
		Let $B_0\subset B \subset B_1$ be Banach spaces, $B_0$ and $B_1$ reflexive, with compact embedding of $B_0$ in $B$. Let $p\in(1,\infty)$ and $\alpha\in(0,1)$ be given. Then, the embedding of $L^p(0,T;B_0)\cap W^{\alpha,p}(0,T;B_1)$ in $L^p(0,T;B)$ is compact.

	\end{lemma}
	
	\begin{lemma}{\cite[Theorem 2.2]{flandoli1995martingale}} \label{lem:compact:B_0.B} Let $B_0\subset B$ be Banach spaces with compact embedding of $B_0$ in $B$. Let $p\in(1,\infty)$ and $\alpha\in(0,1)$ be given satisfying $\alpha p >1$. Then, for all $\gamma\in(0,\alpha p-1)$, the embedding $W^{\alpha,p}(0,T;B_0)\subset C^\gamma(0,T;B_0)$ is continuous and the embedding $C^\gamma(0,T;B_0)\subset C(0,T;B)$ is compact.
	\end{lemma}

	Turning back to the regularized equation \eqref{eqn:BBM:regularized}, the existence of martingale solutions for \eqref{eqn:BBM:regularized} is guaranteed in the following proposition.

	\begin{proposition} \label{prop:BBM:regularized:martingale}
		For every $u_0\in H^1$, $\varepsilon>0$ and $T>0$, there exists a martingale solutions of \eqref{eqn:BBM:regularized}. That is there exist a stochastic basis $(\Omega^\varepsilon, \mathcal{F}^\varepsilon, \mathbb{P}^\varepsilon, \{\mathcal{F}^\varepsilon_t\}_{t \geq 0})$, a Wiener process $W^\varepsilon(t)$ and an adapted process $u^\varepsilon$ on this basis satisfying that $\P^\varepsilon$-a.s.
		\begin{align*}
			u^\varepsilon\in L^\infty(0,T;H^1(D))\cap L^2(0,T;H^2(D))\cap C(0,T;H^{-r}(D)),\quad r>2,
		\end{align*}
		and that for all $v\in H^2$, the following holds $\P^\varepsilon$-a.s.
		\begin{align*}
			(u^\varepsilon(t),v)+(\nab u^\varepsilon(t),\nab v) & +\varepsilon\int_0^t (\triangle u^\varepsilon(s),\triangle v)ds+\nu \int_0^t ( u^\varepsilon(s), v)ds+ \int_0^t \big(\div F( u^\varepsilon(s)), v\big)ds \\
			&= (u_0,v)+(\nab u_0,\nab v) +\int_0^t \big(G(u^\varepsilon(s)) d W^\varepsilon(s),v\big),\quad \textup{a.e.}\,\, t\in[0,T].
		\end{align*}
	\end{proposition}
	
	Since the proof of Proposition \ref{prop:BBM:regularized:martingale} is standard, making use of a usual Galerkin approach, we only sketch the argument without going into details. Specifically, recall the projection space $H_n(D) = \textup{span}\{e_k:|k|\le n\}$ and consider the following equation in $H_n(D)$
	\begin{align}   \label{eqn:BBM:regularized:P_n}
		du^\varepsilon_n  - d(\Delta u^\varepsilon_n)+ \varepsilon \Delta^2 u^\varepsilon_n dt+ [\nu u^\varepsilon_n + P_n\div( F(u^\varepsilon_n) )]\, dt&=  P_n G(u^\varepsilon_n)\,dW(t), 
		\\ 
		u^\varepsilon_n(0) &= P_n u_0. \notag 
	\end{align}
	It is not difficult to see that \eqref{eqn:BBM:regularized:P_n} satisfies a priori estimate for all $p\ge 1$
	\begin{align*}
		\E \Big[ \sup_{t\in[0,T]}\|u^\varepsilon_n(t)\|^{2p}_{H^1}  \Big] + \E\Big[ \int_0^T \| u^\varepsilon_n(s)\|^2_{H^2}\|u^\varepsilon_n(s)\|^{2p-2}_{H^1}  d s\Big] \le C,
	\end{align*}
	where the positive constant $C=C(T,u_0,\varepsilon,p)$ does not depend on the Galerkin parameter $n$.
	Since this is a finite-dimensional system (in $H_n(D)$), given an arbitrary stochastic basis $(\Omega^\varepsilon, \mathcal{F}^\varepsilon, \mathbb{P}^\varepsilon, \{\mathcal{F}^\varepsilon_t\}_{t \geq 0}\}$ and a Wiener process $W^\varepsilon(t) )$, there always exists a unique strong solution for \eqref{eqn:BBM:regularized:P_n}. Then, by sending $n$ to infinity, the existence of martingale solutions for \eqref{eqn:BBM:regularized} can be carried out using an argument similar to the proof of Proposition \ref{prop:BBM:martingale}.

	We now assert that $u^\varepsilon$ satisfies uniform-in-$\varepsilon$ moment bounds. This is summarized below through Lemma \ref{lem:moment-bound:u^epsilon}, whose proof is similar to that of Lemma \ref{lem:moment-bound:H1} and thus is omitted.
	\begin{lemma} \label{lem:moment-bound:u^epsilon}
		Let $u^\varepsilon$ be a martingale solution of \eqref{eqn:BBM:regularized}. Then, the following holds
		\begin{align} \label{ineq:moment-bound:u^epsilon}
			\E \Big[ \sup_{t\in[0,T]}\|u^\varepsilon(t)\|^{2p}_{H^1}  \Big] + \varepsilon\E\Big[ \int_0^T \| u^\varepsilon(s)\|^2_{H^2}\|u^\varepsilon(s)\|^{2p-2}_{H^1}  d s\Big] \le C,
		\end{align}
		for a positive constant $C=C(u_0,T)$ independent of $\varepsilon$.
	\end{lemma}

	We now turn to the construction of martingale solutions for \eqref{eq1.1} through Proposition \ref{prop:BBM:martingale}, stated and proven next.

	\begin{proposition} \label{prop:BBM:martingale}
		For every $u_0\in H^1$, there exists a martingale solution of \eqref{eq1.1}. That is there exist a stochastic basis $(\Omega^0, \mathcal{F}^0, \mathbb{P}^0, \{\mathcal{F}^0_t\}_{t \geq 0})$, a Wiener process $W^0(t)$ and an adapted process $u$ on this basis satisfying that $\P^0$-a.s.
		\begin{align*}
			u\in L^\infty(0,T;H^1(D))\cap C_w(0,T;H^1(D))\cap C(0,T;H^{-r}(D)),\quad r>2,
		\end{align*}
		and that for all $v\in H^1$, the following holds $\P^0$-a.s.
		\begin{align*}
			(u(t),v)+(\nab u(t),\nab v) & +\nu \int_0^t ( u(s), v)ds+ \int_0^t \big(\div F( u(s)), v\big)ds \\
			&= (u_0,v)+(\nab u_0,\nab v) +\int_0^t \big(G(u(s)), d W^0(s),v\big),\quad \textup{a.e.}\,\, t\in[0,T].
		\end{align*}        
	\end{proposition}
	\newcommand{\ueps}{u^\varepsilon}
	
	\begin{proof}
		Following the approach of \cite[Theorem 4.1]{de1998stochastic} and \cite[Theorem 3.1]{flandoli1995martingale}, for notational simplicity, we recast the solution $u^\varepsilon$ of \eqref{eqn:BBM:regularized} as 
		\begin{align*}
			&\M \ueps (t) \\
			&\quad  = \M u_0 -\varepsilon\int_0^t \Delta^2 \ueps (s) ds - \nu \int_0^t \ueps (s) ds -\int_0^t \div F(\ueps(s))ds+\int_0^t G(\ueps(s))dW^\varepsilon(s)\\
			& \quad = \M u_0+J_1^\varepsilon+\dots+J_4^\varepsilon.
		\end{align*}
		Note that we interpret the above equation as an identity in $H^{-2}(D)$ in the sense of Proposition \ref{prop:BBM:regularized:martingale}. In view of estimate \eqref{ineq:moment-bound:u^epsilon} with $p=1$, we infer
		\begin{align*}
			\E \Big[\| J_1^\varepsilon \|_{W^{1,2}(0,T;H^{-2}(D))}^2\Big] \le C\varepsilon^2\E\Big[\int_0^T \|\ueps(s)\|^2_{H^2} ds\Big]\le C\varepsilon,
		\end{align*}
		where $C=C(u_0,T)$ is a positive constant independent of $\varepsilon$. Likewise,
		\begin{align*}
			\E \Big[\| J_2^\varepsilon \|_{W^{1,2}(0,T;L^2(D))}^2\Big] \le C\E\Big[\int_0^T \|\ueps(s)\|^2_{L^2} ds\Big]\le C\E \Big[ \sup_{t\in[0,T]}\|u^\varepsilon(t)\|^{2}_{H^1}  \Big] \le C.
		\end{align*}
		Concerning $J_3^\varepsilon$, given $v\in H^1$, we employ Holder's inequality and the embedding $H^1\subset L^4$ (in dimension two) to estimate
		\begin{align*}
			\big( \div(F(\ueps)),v\big) \le \|\ueps\|_{L^4}\|\nab \ueps\|_{L^2}\|v\|_{L^4} \le C\|u\|^2_{H^1}\|v\|_{H^1}.
		\end{align*}
		So, using the same argument for $J_2^\varepsilon$, we deduce
		\begin{align*}
			\E \Big[\| J_3^\varepsilon \|_{W^{1,2}(0,T;H^{-1}(D))}\Big] \le C\E\Big[\int_0^T \|u\|^2_{H^1}ds\Big] \le C\E \Big[ \sup_{t\in[0,T]}\|u^\varepsilon(t)\|^{2}_{H^1}  \Big] \le C.
		\end{align*}
		Turning to the noise term $J_4^\varepsilon$, we employ \cite[Lemma 2.1]{flandoli1995martingale} to infer for any $\alpha\in(0,1/2)$ and $p\ge 2$
		\begin{align*}
			\E \Big[\| J_4^\varepsilon \|_{W^{\alpha,p}(0,T;L^2(D))}^p\Big] \le C(\alpha,p)\E\Big[ \int_0^T\|G(\ueps(s)\|^p_{L^2}ds\Big].
		\end{align*}
		From \eqref{cond:|G(u)|_Hm<|u|_Hm} and estimate \eqref{ineq:moment-bound:u^epsilon}, we get
		\begin{align*}
			\E \Big[\| J_4^\varepsilon \|_{W^{\alpha,p}(0,T;L^2(D))}^p\Big] \le C(\alpha,p)\E \Big[ \sup_{t\in[0,T]}\|u^\varepsilon(t)\|^{p}_{H^1}  \Big] \le C(\alpha,p).
		\end{align*}
		
		Now, from the above estimates and \eqref{ineq:moment-bound:u^epsilon}, we have the uniform bound in $\varepsilon$
		\begin{align*}
			&\E \Big[\int_0^T \|\ueps(s)\|^2_{H^1}ds\Big]+ \E\Big[\|\ueps\|_{W^{\alpha,2}(0,T;H^{-2}(D))}  \Big]\\
			&\quad \le \E \Big[\int_0^T \|\ueps(s)\|^2_{H^1}ds\Big]+ \E\Big[\|\M\ueps\|_{W^{\alpha,2}(0,T;H^{-2}(D))}  \Big]\le C.
		\end{align*}
		In view of Lemma \ref{lem:compact:B_0.B.B_1}, we deduce that $\Law(\ueps)$ is tight in $L^2(0,T;L^2(D))$. Moreover, for any $\gamma\in (0,1)$, we apply the first embedding result of Lemma \ref{lem:compact:B_0.B} to infer that  
		\begin{align*}
			\E\Big[ \|\ueps\|_{C^{\gamma}(0,T;H^{-2}(D))} \Big] \le  \E\Big[ \|\M\ueps\|_{C^{\gamma}(0,T;H^{-2}(D))} \Big] \le C.
		\end{align*}
		It follows from the second embedding result of Lemma \ref{lem:compact:B_0.B} that $\Law(\ueps)$ is also tight in $C(0,T;H^{-r}(D))$ for every $r>2$. 
		
		By the Skorohod Theorem, up to a subsequence (still denoted by $\varepsilon$), there exist a stochastic basis $(\tilde\Omega, \tilde{\mathcal{F}}, \tilde{\P}, \{\tilde{\mathcal{F}}_t\}_{t \geq 0} )$ and random variables $\tilde{u}^\varepsilon$ and $\tilde u$ with values in $L^\infty(0,T;H^1(D))\cap C(0,T;H^{-r}(D))$ such that $\Law(\ueps)=\Law(\tilde{u}^\varepsilon)$ and that $\tilde{\P}$-a.s.,
		\begin{align}\label{lim:utilde.epsilon->utilde:L^2.cap.C}
			\tilde{u}^\varepsilon\to \tilde{u}\quad\text{strongly in}\quad L^2(0,T;L^2(D))\cap C(0,T;H^{-r}(D)).
		\end{align}
		As a consequence, 
		\begin{align} \label{lim:Mueps->Mu}
			\M\tilde{u}^\varepsilon\to \M\tilde{u}\quad\text{strongly in}\quad L^2(0,T;H^{-2}(D))\cap C(0,T;H^{-r-2}(D)).
		\end{align}
		Moreover, in view of \cite[Lemma 1.4, page 263]{temam2024navier}, it holds that 
		\begin{align*}
			L^\infty(0,T;H^1(D))\cap C(0,T;H^{-r}(D))\subset C_w(0,T;H^1(D)).
		\end{align*}
		This implies that $\tilde{u}^\varepsilon$ and $\ueps$ are both weakly continuous functions from $[0,T]$ to $H^1$. Also, thanks to \eqref{ineq:moment-bound:u^epsilon}, we have the moment bound
		\begin{align}  \label{ineq:moment-bound:utilde.epsilon}
			\E \Big[ \sup_{t\in[0,T]}\|\tilde u^\varepsilon(t)\|^{2p}_{H^1}  \Big] + \varepsilon\E\Big[ \int_0^T \| \tilde u^\varepsilon(s)\|^2_{H^2}\|\tilde u^\varepsilon(s)\|^{2p-2}_{H^1}  d s\Big] \le C.
		\end{align}
		In particular, up to a subsequence (still denoted by $\varepsilon$), it is not difficult to deduce that $\tilde{\P}$-a.s.,
		\begin{align} \label{lim:utilde.epsilon->u.tilde:weak*}
			\tilde{u}^\varepsilon \rightharpoonup ^*\tilde{u}\quad\text{in}\quad L^\infty(0,T;H^1(D)).
		\end{align}

		Next, define
		\begin{align*}
			M^\varepsilon(t) & = \M \ueps(t)- \M u_0 +\int_0^t \Big[\varepsilon\Delta^2 \ueps (s) + \nu  \ueps (s) + \div F(\ueps(s))\Big]ds = \int_0^t G(\ueps(s))dW^\varepsilon(s),
		\end{align*}
		and
		\begin{align}\label{form:Mtilde.epsilon}
			\tilde{M}^\varepsilon(t) & = \M \tilde{u}^\varepsilon (t)- \M u_0 +\int_0^t \Big[\varepsilon\Delta^2 \tilde{u}^\varepsilon (s) + \nu  \tilde{u}^\varepsilon (s) + \div F(\tilde{u}^\varepsilon(s))\Big]ds.
		\end{align}
		We note that $\tilde{M}^\varepsilon$ is a square integrable martingale with respect to the filtration $\sigma\{\tilde{u}^\varepsilon(s),0\le s\le t\}$. Indeed, thanks to $\Law(\ueps)=\Law(\tilde{u}^\varepsilon)$, for all bounded continuous function $\phi$ on $L^2(0,s;L^2(D))$ or $C(0,s;H^{-r}(D))$, and $a,b\in H^r(D)$, it holds that
		\begin{align*}
			\E\Big[ \big(\tilde{M}^\varepsilon(t)- \tilde{M}^\varepsilon(s),a \big)\phi(\tilde{u}^\varepsilon|_{[0,s]}) \Big] = \E\Big[ \big(M^\varepsilon(t)- M^\varepsilon(s),a \big)\phi(u^\varepsilon|_{[0,s]}) \Big]=0,
		\end{align*}
		and that
		\begin{align*}
			&\E\Big[ \Big(\big(\tilde{M}^\varepsilon(t),a\big)\big(\tilde{M}^\varepsilon(t),b\big)- \big(\tilde{M}^\varepsilon(s),a \big)\big(\tilde{M}^\varepsilon(s),b\big)\\
			&\hspace{2cm}-\int_s^t \big( G(\tilde{u}^\varepsilon (\ell))^*a,G(\tilde{u}^\varepsilon (\ell))^*b  \big)d\ell \Big)\phi(\tilde{u}^\varepsilon|_{[0,s]}) \Big]\\
			& = \E\Big[ \Big(\big( M^\varepsilon(t),a\big)\big( M^\varepsilon(t),b\big)- \big( M^\varepsilon(s),a \big)\big( M^\varepsilon(s),b\big)\\
			&\hspace{2cm}-\int_s^t \big( G( u^\varepsilon (\ell))^*a,G( u^\varepsilon (\ell))^*b  \big)d\ell \Big)\phi( u^\varepsilon|_{[0,s]}) \Big] = 0.
		\end{align*}
		Also, by Burkholder's inequality, for all $p\ge 1$
		\begin{align*}
			\E \Big[\sup_{t\in[0,T]}\|\tilde{M}^\varepsilon(t)\|^{2p}_{H^1}  \Big]  =\E \Big[\sup_{t\in[0,T]}\|M^\varepsilon(t)\|^{2p}_{H^1}  \Big]
			&\le C\E \Big[ \int_0^T \|G(\ueps(\ell)) \|^{2p}_{H^1} d\ell \Big] \\&\le C\Big(1+ \E\Big[\int_0^T \|\ueps(\ell)\|^{2p}_{H^1}d\ell\Big]  \Big),
		\end{align*}
		whence
		\begin{align} \label{ineq:moment-bound:Mtilde.epsilon}
			\E \Big[\sup_{t\in[0,T]}\|\tilde{M}^\varepsilon(t)\|^{2p}_{H^1}  \Big] &\le C(u_0,p,T).
		\end{align}
		Now, let us define 
		\begin{align} \label{form:Mtilde}
			\tilde{M}(t) = \M \tilde{u} - \M u_0+\int_0^t \Big[ \nu  \tilde{u} (s) + \div F(\tilde{u}(s))\Big]ds,
		\end{align}
		where equation \eqref{form:Mtilde} is understood $\tilde{\P}$-a.s. as an identity in $H^{-r-2}(D)$. We claim that for all $a\in H^{r+2}(D)$, $\tilde{\P}$-a.s.,
		\begin{align} \label{lim:Mtilde.epsilon->Mtilde}
			\big( \tilde{M}^\varepsilon(t),a  \big)\to \big( \tilde{M}(t),a  \big),\quad \varepsilon\to 0.
		\end{align}
		In view of the convergence $\M \tilde{u}^\varepsilon\to \M \tilde{u}$ in $C(0,T;H^{-r-2})$ from \eqref{lim:Mueps->Mu}, we readily have
		\begin{align*}
			\big( \M \tilde{u}^\varepsilon(t)- \M \tilde{u}(t),a  \big)\to 0,\quad \varepsilon\to 0.
		\end{align*}
		Since $a\in H^{r+2}$, we employ limit \eqref{lim:utilde.epsilon->utilde:L^2.cap.C} to infer
		\begin{align*}
			\varepsilon \int_0^t \big(\Delta ^2 \tilde{u}^\varepsilon(\ell),a\big)d\ell & =  \varepsilon \int_0^t \big(\tilde{u}^\varepsilon(\ell),\Delta ^2 a\big)d\ell\\
			&\le \varepsilon\| a\|_{H^4}\int_0^t \|\tilde{u}^{\varepsilon}(\ell)-\tilde{u}(\ell)\|_{L^2}d\ell+ \varepsilon \| a\|_{H^4}\int_0^t \|\tilde{u}(\ell)\|_{L^2}d\ell 
			\to 0.
		\end{align*}
		Likewise,
		\begin{align*}
			\nu\int_0^t \big(\tilde{u}^{\varepsilon}(\ell)-\tilde{u}(\ell),a\big)d\ell \to 0.
		\end{align*}
		Turning to the difference involving $\div F$, we employ integration by parts and obtain (recalling $F(u)= \langle u+\frac{1}{2}u^2,u+\frac{1}{2}u^2\rangle^T$)
		\begin{align*}
			&\int_0^t \big(\div (F(\tilde{u}^\varepsilon)-F(\tilde{u})),a\big)d\ell\\
			&= \int_0^t \big( (\partial_x  +\partial_y)(\tilde{u}^\varepsilon - \tilde{u}), a\big)+ \big( (\tilde{u}^\varepsilon - \tilde{u})(\partial_x  +\partial_y)(\tilde{u}^\varepsilon - \tilde{u}), a\big)\\
			&\hspace{3cm}+ \big( \tilde{u}^\varepsilon - \tilde{u},(\partial_x  +\partial_y)\tilde{u} \,  a\big) + \big((\partial_x  +\partial_y) (\tilde{u}^\varepsilon - \tilde{u}),\tilde{u}a\big)     d\ell \\
			& = \int_0^t -\big( \tilde{u}^\varepsilon - \tilde{u},(\partial_x  +\partial_y)  a\big)-\frac{1}{2} \big( |\tilde{u}^\varepsilon - \tilde{u}|^2,(\partial_x  +\partial_y) a\big)  - \big(\tilde{u}^\varepsilon - \tilde{u},\tilde{u}\,(\partial_x  +\partial_y) a\big)     d\ell.
		\end{align*}
		Using Sobolev embedding $H^2(D)\subset L^\infty(D)$ in dimension two, we get
		\begin{align*}
			\Big|\int_0^t \big(\div (F(\tilde{u}^\varepsilon)-F(\tilde{u})),a\big)d\ell \Big|\le C\int_0^t \big(\|\tilde{u}^\varepsilon - \tilde{u}\|_{L^2}+\|\tilde{u}^\varepsilon - \tilde{u}\|^2_{L^2}    +\|\tilde{u}^\varepsilon - \tilde{u}\|_{L^2}\|\tilde{u}\|_{L^2}\big)\|a\|_{H^3}d\ell,
		\end{align*}
		which also converges to 0, by virtue of the strong convergence \eqref{lim:utilde.epsilon->utilde:L^2.cap.C}. Altogether with expressions \eqref{form:Mtilde.epsilon} and \eqref{form:Mtilde}, we establish \eqref{lim:Mtilde.epsilon->Mtilde} as claimed. Now, thanks to the bound \eqref{ineq:moment-bound:Mtilde.epsilon}, the continuity and boundedness of $\phi$ in $C(0,T;H^{-r-2})$, the sequence $\{ \big(\tilde{M}^\varepsilon(t)- \tilde{M}^\varepsilon(s),a \big)\phi(\tilde{u}^\varepsilon|_{[0,s]})\}_{\varepsilon> 0}$ is uniformly integrtable. In light of the Vitali Convergence Theorem, we deduce
		\begin{align*}
			\E\Big[ \big(\tilde{M}^\varepsilon(t)- \tilde{M}^\varepsilon(s),a \big)\phi(\tilde{u}^\varepsilon|_{[0,s]}) \Big] \to \E\Big[ \big(\tilde{M}(t)- \tilde{M}(s),a \big)\phi(\tilde{u}|_{[0,s]}) \Big],\quad \varepsilon\to 0,
		\end{align*}
		implying 
		\begin{align*}
			\E\Big[ \big(\tilde{M}(t)- \tilde{M}(s),a \big)\phi(\tilde{u}|_{[0,s]}) \Big] = 0.
		\end{align*}
		A similar argument can be carried out to see that for all $a,b\in H^{r+2}(D)$,
		\begin{align*}
			&\E\Big[ \Big(\big(\tilde{M}(t),a\big)\big(\tilde{M}(t),b\big)- \big(\tilde{M}(s),a \big)\big(\tilde{M}(s),b\big)\\
			&\hspace{2cm}-\int_s^t \big( G(\tilde{u} (\ell))^*a,G(\tilde{u} (\ell))^*b  \big)d\ell \Big)\phi(\tilde{u}|_{[0,s]}) \Big]= 0.
		\end{align*}
		It follows that $\tilde{M}$ defined in \eqref{form:Mtilde} is a square $H^{-r-2}$-integrable martingale 
		with the quadratic variation $\int_0^t G(\tilde{u})G(\tilde{u})^*ds$ on the filtration $\sigma\{\tilde{u}(s),0\le s\le t \}$. Moreover, from \eqref{ineq:moment-bound:utilde.epsilon} and \eqref{lim:utilde.epsilon->u.tilde:weak*}, it holds that
		\begin{align*}
			\E \Big[ \sup_{t\in[0,T]}\|\tilde u(t)\|^{2p}_{H^1}  \Big] \le C.
		\end{align*}
		In turn, this implies by virtue of Burkholder's inequality that $\tilde{M}$ takes values in $L^\infty(0,T;H^1(D))$. 
		
		Finally, we apply the martingale representation theorem \cite[Theorem 8.2]{da2014stochastic} to establish the existence of martingale solutions for \eqref{eq1.1}.

	\end{proof}

	Lastly, we proceed to establish the pathwise uniqueness of a martingale solution via Lemma \ref{lem:BBM:martinglae:uniqueness}. In light of \cite{ondrejat2004uniqueness}, Lemma \ref{lem:BBM:martinglae:uniqueness} together with the existence result of Proposition \ref{prop:BBM:martingale} concludes the global well-posedness of the strong solutions for \eqref{eq1.1}.
	
	\begin{lemma} \label{lem:BBM:martinglae:uniqueness}

		Given $u^0\in H^1$, let $u_1$ and $u_2$ be two martingale solutions of \eqref{eq1.1} adapted to the same stochastic basis $(\Omega^0, \mathcal{F}^0, \mathbb{P}^0, \{\mathcal{F}^0_t\}_{t \geq 0})$ with the same initial condition $u^0$. Then, $\P^0$-a.s., $u_1(t)=u_2(t)$ for all $t\in[0,T]$.

	\end{lemma}
	\begin{proof} The uniqueness argument follows closely the approach of \cite[Proposition 3.4]{ferrario2024uniqueness} tailored to the setting of \eqref{eq1.1}. First of all, let $u_1^0,u_2^0\in H^1$ be given and $u_1$ and $u_2$ be two martingale solutions of \eqref{eq1.1} with initial conditions $u^1_0,u^2_0$, respectively. We claim that the following holds
		\begin{align}\label{ineq:BBM:martinglae:uniqueness}
			&\E\Big[ \exp\Big\{-  C\int_0^t\|u_1(s)\|_{H^1}ds  -Ct\Big\}\big(\|u_1(t)-u_2(t)\|^2_{L^2}+\|\nab u_1(t)-\nab u_2(t)\|^2_{L^2}\big)\Big] \notag \\
			&\le \|u_1^0-u_2^0\|^2_{L^2}+\|\nab u_1^0-\nab u_2^0\|^2_{L^2},\quad t\ge 0,
		\end{align}
		for all positive constant $C$ sufficiently large independent of $u_1^0,u_2^0$ and $t$.
		
		Indeed, setting $\ubar = u_1-u_2$, we observe from \eqref{eq1.1} that $\ubar$ satisfies the equation
		\begin{align*}
			d\M\ubar +\big(\nu\,\ubar  +\div(F(u_1))-\div(F(u_1))\big)dt = \big(G(u_1)-G(u_2)\big)dW^0(t).
		\end{align*}
		Let $C$ be given and be chosen later. Similar to identity \eqref{eqn:d(|u|^2_L2+|nab.u|^2_L2)}, a routine calculation gives
		\begin{align*}
			& d\Big[\exp\Big\{-  C\int_0^t\|u_1(s)\|_{H^1}ds-Ct  \Big\}\big(\|\ubar\|^2_{L^2}+\|\nab \ubar\|^2_{L^2}\big)\Big]\\
			&= -C(\|u_1\|_{H^1}+1)\exp\Big\{-  C\int_0^t\|u_1(s)\|_{H^1}ds-Ct  \Big\}\big(\|\ubar\|^2_{L^2}+\|\nab \ubar\|^2_{L^2}\big)dt\\
			&\qquad + \exp\Big\{-  C\int_0^t\|u_1(s)\|_{H^1}ds -Ct \Big\}\Big[-2\nu\|\ubar\|^2_{L^2}dt-2\big(\div (F(u_1)-F(u_2)),\ubar\big)dt\\
			&\hspace{3cm}+2\big(G(u_1)-G(u_2),\ubar\big)dW^0+\|\M^{-\frac{1}{2}}\big(G(u_1)-G(u_2)\big)\|^2_{L^2}dt\Big].
		\end{align*}
		We employ \eqref{Assump_Lipschitz} to infer
		\begin{align*}
			\|\M^{-\frac{1}{2}}\big(G(u_1)-G(u_2)\big)\|^2_{L^2} \le \|G(u_1)-G(u_2)\|^2_{L^2} \le C_G^2\|\ubar\|^2_{L^2},
		\end{align*}
		where $C_G$ is the constant as in \eqref{Assump_Lipschitz}. Concerning the term involving $\div F$, we invoke an integration by parts while making use of \eqref{Lady_ineq} to obtain
		\begin{align*}
			\big(\div (F(u_1)-F(u_2)),\ubar\big) = \frac{1}{2}\big(\ubar^2,(\partial_x  +\partial_y) u_1\big)\le \tilde{C}\|\ubar\|^2_{H^1}\|u_1\|_{H^1}.
		\end{align*}
		In the above, $\tilde{C}$ is an absolute constant independent of $\ubar$ and $u_1$. It follows that
		\begin{align*}
			&d\Big[\exp\Big\{-  C\int_0^t\|u_1(s)\|_{H^1}ds-Ct  \Big\}\big(\|\ubar\|^2_{L^2}+\|\nab \ubar\|^2_{L^2}\big)\Big]\\
			&= -C(\|u_1\|_{H^1}+1)\exp\Big\{-  C\int_0^t\|u_1(s)\|_{H^1}ds-Ct  \Big\}\big(\|\ubar\|^2_{L^2}+\|\nab \ubar\|^2_{L^2}\big)dt\\
			&\qquad + \exp\Big\{-  C\int_0^t\|u_1(s)\|_{H^1}ds -Ct \Big\}\Big[\tilde{C}\big(\|\ubar\|^2_{L^2}+\|\nab \ubar\|^2_{L^2}\big)\big(\|u_1\|_{H^1}+1\big)dt\\
			&\hspace{3cm}+2\big(G(u_1)-G(u_2),\ubar\big)dW^0\Big].
		\end{align*}
		As a consequence, for all $C$ sufficiently large, we infer
		\begin{align*}
			&d\Big[\exp\Big\{-  C\int_0^t\|u_1(s)\|_{H^1}ds-Ct  \Big\}\big(\|\ubar\|^2_{L^2}+\|\nab \ubar\|^2_{L^2}\big)\Big]\\
			&\le 2\big(G(u_1)-G(u_2),\ubar\big)dW^0,
		\end{align*}
		whence
		\begin{align*}
			\E\Big[\exp\Big\{-  C\int_0^t\|u_1(s)\|_{H^1}ds-Ct  \Big\}\big(\|\ubar(t)\|^2_{L^2}+\|\nab \ubar(t)\|^2_{L^2}\big)\Big] \le \|\ubar^0\|^2_{L^2}+\|\nab \ubar^0\|^2_{L^2}.
		\end{align*}
		This establishes \eqref{ineq:BBM:martinglae:uniqueness}, as claimed.
		
		Turning to the uniqueness property, in view of \eqref{ineq:BBM:martinglae:uniqueness}, given $u^1_0=u^2_0$, we get
		\begin{align*}
			\E\Big[\exp\Big\{-  C\int_0^t\|u_1(s)\|_{H^1}ds-Ct  \Big\}\big(\|\ubar(t)\|^2_{L^2}+\|\nab \ubar(t)\|^2_{L^2}\big)\Big] =0,\quad t\in[0,T].
		\end{align*}
		As a consequence, $\P^0$-a.s.
		\begin{align*}
			\|u_1(t)-u_2(t)\|_{H^1}=0,\quad t\in [0,T]\cap \mQ  .
		\end{align*}
		Since $u_1-u_2$ is an element in $C_w(0,T;H^1(D))$, we immediately deduce that $\P^0$-a.s., \begin{align*}
			\|u_1(t)-u_2(t)\|_{H^1}=0,\quad t\in [0,T]  .
		\end{align*} 
		The proof is thus finished.
	\end{proof}

	%			\textbf{Acknowledgments.} 
	
	%\printbibliography[heading=none]
	\bibliographystyle{abbrv}
	\bibliography{references}

\end{document}